\documentclass[10pt]{amsart}
 
\usepackage{amsthm,amsmath,amssymb}
\usepackage{graphicx}
\usepackage[comma, sort&compress]{natbib}
\usepackage{tabularx}
\usepackage{multirow}
\usepackage{accents}
\usepackage{amssymb}
\usepackage{natbib}
\usepackage{tikz}
\usepackage{xcolor}
\usetikzlibrary{shapes,arrows,positioning}

\usepackage{paralist}
\RequirePackage{hyperref}

\newtheorem{theorem}{Theorem}

\newtheorem{lemma}[theorem]{Lemma}

\newtheorem{claim}[theorem]{Claim}

\theoremstyle{definition} 
\newtheorem{definition}{Definition}
\theoremstyle{remark}

\numberwithin{equation}{section}
\numberwithin{theorem}{section}
\numberwithin{example}{section}
\numberwithin{definition}{section}

\numberwithin{figure}{section}

\newcommand{\argmax}{\mathop{\mathrm{argmax}}}    
\newcommand{\argmin}{\mathop{\mathrm{argmin}}}    
\newcommand{\barPhi}{\bar{\Phi}}

\newcommand{\frakz}{\mathfrak{z}}
\newcommand{\frakc}{\mathfrak{c} } 
\newcommand{\frake}{\mathfrak{e} }

\newcommand{\cA}{\mathcal{A}}

\newcommand{\cS}{\mathcal{S}}

\newcommand{\cM}{\mathcal{M}}
\newcommand{\cJ}{\mathcal{J}}
\newcommand{\cC}{\mathcal{C}}
\newcommand{\cE}{\mathcal{E}}
\newcommand{\cI}{\mathcal{I}}
\newcommand{\cK}{\mathcal{K}}
\newcommand{\cP}{\mathcal{P}}
\newcommand{\cN}{\mathcal{N}}
\newcommand{\Hf}{H}
\newcommand{\bG}{{\bf G}}

\newcommand{\bZ}{{\bf Z}}

\newcommand{\bX}{{\bf X}}
\newcommand{\bx}{{\bf x}}

\newcommand{\boldi}{{\bf i}}
\newcommand{\var}{\text{var}}

\newcommand{\bR}{\mathbb{R}}
\newcommand{\bN}{\mathbb{N}}

\newcommand{\barS}{{\bar{S}}}
\newcommand{\barxi}{{\bar{\xi}}}

\newcommand{\bE}{\mathbb{E}}
\newcommand{\bEH}{\mathbb{E}_{H}}
\newcommand{\bEIH}{\mathbb{E}_{I, H}}
\newcommand{\PH}{P_{H}}
\newcommand{\hattheta}{\hat{\theta}}

\newcommand{\figref}[1]{Figure~\ref{fig:#1}}

\newcommand{\secref}[1]{Section~\ref{sec:#1}}
\newcommand{\secsref}[1]{Sections~\ref{sec:#1}}
\newcommand{\secssref}[1]{\ref{sec:#1}}

\newcommand{\appref}[1]{Appendix~\ref{app:#1}}
\newcommand{\appsref}[1]{Appendices~\ref{app:#1}}
\newcommand{\appssref}[1]{\ref{app:#1}}

\newcommand{\defref}[1]{Definition~\ref{def:#1}}

\newcommand{\clmref}[1]{Claim~\ref{clm:#1}}
\newcommand{\clmsref}[1]{Claims~\ref{clm:#1}}

\newcommand{\lemref}[1]{Lemma~\ref{lem:#1}}
\newcommand{\lemsref}[1]{Lemmas~\ref{lem:#1}}
\newcommand{\lemssref}[1]{\ref{lem:#1}}

\newcommand{\thmref}[1]{Theorem~\ref{thm:#1}}

\title[]{Singularity-agnostic incomplete U-statistics for  testing polynomial constraints in Gaussian covariance matrices}

\author[D.~Leung]{Dennis Leung} 
\address{University of Melbourne}
\email{dennis.leung@unimelb.edu.au}

\author[N.~Sturma]{Nils Sturma} 
\address{Technical University of Munich}
\email{nils.sturma@tum.de}

\begin{document}

\begin{abstract}

  Testing the goodness-of-fit of a model with its defining functional
  constraints in the parameters could date back to \citet{spearman1927}, who analyzed  the famous ``tetrad"  polynomial in the covariance matrix of the observed  variables in a single factor model. Despite its long history,  the  Wald test typically employed to operationalize this approach could produce very inaccurate test sizes in many situations, even when the regular conditions  for the classical normal asymptotics are met and  a very large sample  is available.

 Focusing on testing a polynomial constraint in a Gaussian covariance matrix, we obtained a new   understanding of this  baffling phenomenon: When the null hypothesis is true but ``near-singular", the standardized Wald test exhibits slow weak convergence, owing to the sophisticated dependency structure inherent to the underlying U-statistic  that ultimately drives its limiting distribution; this can also be rigorously explained by a key ratio of moments encoded in the 
  Berry-Esseen bound quantifying the normal approximation error involved. As an  alternative, we advocate the use of an \emph{incomplete} U-statistic to mildly tone down the dependence thereof and render the speed of convergence agnostic to the singularity status of the hypothesis. In parallel,  we develop a Berry-Esseen bound that is mathematically descriptive of the singularity-agnostic nature of our standardized incomplete U-statistic, using some of the finest exponential-type inequalities  in the literature.

\end{abstract}


\subjclass[2000]{}

\maketitle

\section{Introduction}

\subsection{Background} \label{sec:bkground}
We consider  testing a hypothesis   defined by  a  polynomial equality constraint in the covariance matrix $\Theta =  (\theta_{uv})_{1 \leq u, v \leq p}$ of a  $p$-variate centered Gaussian random vector ${\bf X} \sim \cN_p(0, \Theta)$. That is, for a  non-constant polynomial  $f: \mathbb{R}^{p+1 \choose 2} \rightarrow \mathbb{R}$ in the entries $\theta_{uv}$, we consider a hypothesis of the form
\begin{equation} \label{null_poly_contraint_eq}
H: f(\Theta) = 0;
\end{equation}
the domain  of  $f(\cdot)$ is $\mathbb{R}^{p+1 \choose 2}$ due to the symmetry of $\Theta$. To simplify our exposition, we will throughout assume that 
\begin{equation} \label{non_singular_assumption_Theta}
\Theta \text{ is strictly positive definite} \footnote{See \appref{misc_facts} for a discussion on relaxing this assumption.}.
\end{equation}
Testing a hypothesis like $H$ dates  back to the times of \citet{spearman1927} and \citet{wishart1928sampling}, who considered   constraints of  the form 
\begin{multline} \label{tetrad_forms}
\theta_{uv}\theta_{wz} - \theta_{uz}\theta_{vw} = 0, \quad \theta_{uz}\theta_{vw} - \theta_{uw}\theta_{vz} = 0 , \quad \theta_{uv}\theta_{wz} - \theta_{uw}\theta_{vz} = 0\\
 \text{ for all distinct indices }  u, v, w, z
\end{multline}
in a covariance matrix $\Theta$ with $p \geq 4$.
 These polynomials 
in the preceding display  are precisely the $2 \times 2$ minors of $\Theta$ that only involve its off-diagonal entries, which are also known as \emph{tetrads}  in the factor-analysis literature. It can be easily shown that if ${\bf X}$  comes from a one-factor model, which prescribes that the covariance matrix has the structure
\begin{equation} \label{one_factor_model}
\Theta = \Psi + L L^T 
\end{equation}
for a positive definite $p$-by-$p$ diagonal \emph{uniqueness} matrix $\Psi$ and a $p$-by-$1$ \emph{factor loading} matrix $L$ with entries in $\bR$, 
the tetrads must all vanish as in \eqref{tetrad_forms} \citep[Theorem 2.1]{Loera1995}. In fact, selected tetrads  in $\Theta$ also vanish if ${\bf X}$ comes from a Gaussian latent tree  \citep{shiers2016, leung2018algebraic}, which includes the one-factor model as a special case; hence,  the null hypothesis   in \eqref{null_poly_contraint_eq} with $f$ being a  tetrad   forms a basis to test the fit of such models.  For a  factor model with $s \geq 2$ factors,  besides the $(s+1) \times (s+1)$ ``off-diagonal" minors of $\Theta$, there are  other polynomials whose vanishing is also a necessary condition for the model to hold, such as the pentads for $s = 2$ \citep{kelley1935essential} and septads for $s=3$   \citep{DrtonSturmfelsSullivantPTRF2007}. More generally, 
 many  instances of testing similar hypotheses concerning polynomial constraints in Gaussian covariance matrices also arise in graphical modeling and  causal discovery;  we refer the readers to \citet{shafer1993generalization, bollen2000tetrad, chen2014testable, spirtes2000causation, drton2008moments, silva2006learning, treksullivant, chen2017identification} for a non-exhaustive list of references, which also serve as the primary motivations for this work. In what follows, let $\nabla f: \bR^{p+1 \choose 2} \rightarrow \bR^{p+1 \choose 2}$ be the gradient function of $f$; at the true covariance matrix $\Theta$, we  call the hypothesis $H$ 
\begin{enumerate}
\item  \emph{regular}, if $\nabla f(\Theta) \neq 0$, or
\item \emph{singular}, if $\nabla f (\Theta) = 0$.
\end{enumerate}
If the hypothesis $H$ is regular, but the Euclidean norm of its gradient $|\nabla f(\Theta)|$ is close to zero, we  describe $H$ as \emph{near-singular}.
 As common in the literature, we also refer to the true covariance parameter $\Theta$ as a singular point or singularity (with respect to its ambient space of strictly positive definite matrices) if $\nabla f (\Theta) = 0$.


Suppose we  observe independent and identically distributed (i.i.d.) copies
\begin{equation} \label{iid_data_sample}
\bX_1, \dots, \bX_n \sim_{i.i.d.} \cN_p(0, \Theta)
\end{equation}
that have the same distribution as $\bX$. We can compute the 
 sample covariance matrix  $\hat{\Theta}  \equiv  n^{-1}\sum_{i=1}^n \hat{\Theta}_i$, where $\hat{\Theta}_i \equiv (\hat{\theta}_{uv}(\bX_i))_{1 \leq u,v \leq p}$ and for a pair $1\leq u, v \leq p$,
 \begin{equation} \label{uv_prod_fn}
\hat{\theta}_{uv}(\bx) = x_u x_v
 \end{equation}
is a function that multiplies the $u$-th and $v$-th coordinates of a $p$-vector $\bx = (x_1, \dots, x_p)^T$.
We will use $V(\Theta)$ to denote the ${p \choose 2}$-by-${p \choose 2}$ covariance matrix of the Wishart-distributed outer product ${\bf X}{\bf X}^T$  treated as a ${p+1 \choose 2}$-vector of its upper diagonal entries; in particular, it has the structure
\begin{equation} \label{V_structure}
V(\Theta)_{uv, wz} = \theta_{u  w} \theta_{v  z} + \theta_{u  z} \theta_{v  w}  \text{ for any $1\leq  u \leq v \leq p$ and $1\leq  w \leq z \leq p$},
\end{equation}
and is itself strictly positive definite as a consequence of $\Theta$  being so under \eqref{non_singular_assumption_Theta}; see \citet[Proposition 8.2]{eatonmultistat}.  To test \eqref{null_poly_contraint_eq},  it is a standard practice to deploy the Wald test, which is based on the  Student-type statistic
 \begin{equation*} 
\hat{T}_f =  \frac{ \sqrt{n} f(\hat{\Theta})}{\sqrt{ (\nabla f (\hat{\Theta}))^T V(\hat{\Theta}) \nabla f (\hat{\Theta})}}
\end{equation*}
that estimates the standardized statistic
 \begin{equation} \label{wald_stat_standardized}
T_f =  \frac{ \sqrt{n} f(\hat{\Theta})}{\sqrt{ (\nabla f (\Theta))^T V(\Theta) \nabla f (\Theta)}}
\end{equation}
 normalized by the true limiting variance $ (\nabla f (\Theta))^T V(\Theta) \nabla f (\Theta)$.   In the regular case where  $\nabla f(\Theta) \neq 0$, it is well established by the classical theory that  $T_f$ converges weakly to the standard normal distribution $\cN(0, 1)$ under $H$, which also implies $\hat{T}_f$ has the same weak limit by Slutsky's theorem.  However,  when $H$ is singular,  $\hat{T}_f$ could take on  non-standard limiting distributions; for instance,  if all the off-diagonal entries of $\Theta$ are zero, then \citet{drton2016wald} and \citet{pillai2016unexpected}'s results imply that for any monomial $f(\Theta) = \theta_{u_1v_1}^{a_1}\dots \theta_{u_m v_m}^{a_m}$, where the $(u_i, v_i)$'s are distinct pairs of indices such that $u_i \leq v_i$  and $a_1, \dots, a_m$ are positive exponents,  the \emph{Wald statistic} $\hat{T}_f ^2$ has a null limiting distribution of $\frac{1}{(\sum_{r=1}^m a_r)^2 } \chi_1^2$. More generally, for constraints defined by smooth functions that are not necessarily polynomials,  \citet{gaffke1999asymptotic, gaffke2002asymptotic} obtained the asymptotic distribution of $\hat{T}_f ^2$ at singular points under second-order regularity conditions, and a thorough literature review on more  related studies 
 can be found in  \citet{dufour2017wald}.  The latter paper is also perhaps the most accomplished in relation to the context of our present work, as the authors have provided a full characterization of the  limiting distributions of $\hat{T}_f $ for  a polynomial constraint in the model parameters under different so-called \emph{singularity orders} of the hypothesis, as well as  stochastic upper bounds  for all possible limiting distributions.

\begin{figure}[t] \label{fig:warped_size}
\caption{The empirical test sizes (produced by $1000$ repeated experiments) of two types of statistics  with critical values calibrated based on  their  asymptotic null distribution  $\cN(0,1)$, plotted against various target nominal levels.  These statistics test the particular tetrad $f(\Theta) = \theta_{14}\theta_{23} - \theta_{13} \theta_{24}$, and  are computed with a Gaussian data sample of size $n = 100$  generated as in \eqref{iid_data_sample}, with a $4$-by-$4$ covariance matrix $\Theta$  having a one-factor structure as in \eqref{one_factor_model}. The entries in the loading matrix $L$ are  all taken to be $0.2$, with the  uniqueness matrix $\Psi$ picked so that the diagonal entries of $\Theta$ are $1$.
$T_f$ and $\hat{T}_f$ are the  standardized and studentized Wald test statistics. With a computational budget of $N = 2n$, $\sqrt{n}U_{n, N}'/\sigma$ and $\sqrt{n}U_{n, N}'/\hat{\sigma}$ are the incomplete U-statistics defined in \eqref{icu_def}, respectively normalized by the true limiting variance $\sigma^2$ in \eqref{rescaling_factor} and a data-driven estimate $\hat{\sigma}^2$ of it. }
\includegraphics[width=0.7\textwidth]{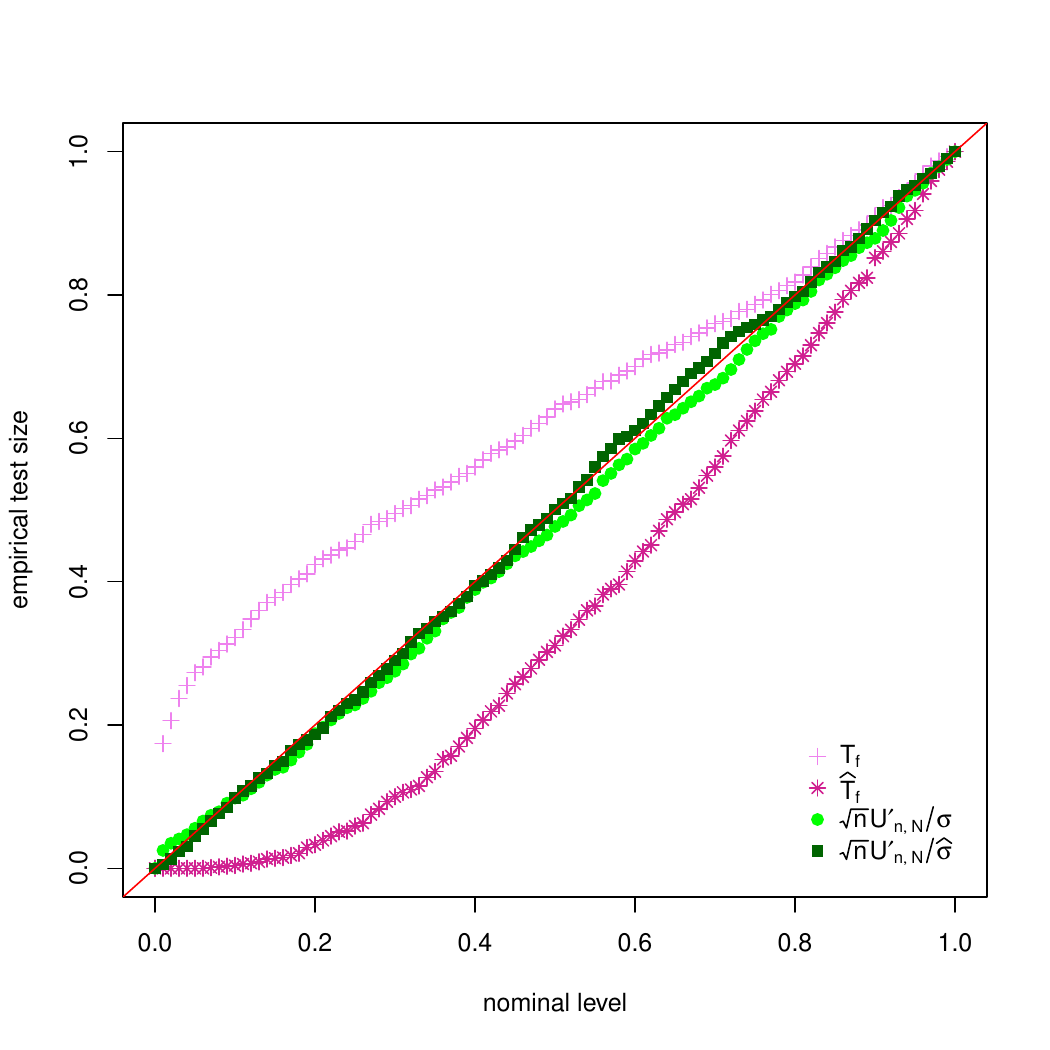}
\end{figure}

However,  most existing works only focus on pinning down the limiting null distributions of $\hat{T}_f$ at singularities, but have not paid much attention to the important scenarios where $H$ is near-singular, i.e. $H$ is regular but $|\nabla f(\Theta)| \approx 0$: While the classical theory guarantees that $T_f$, and hence $\hat{T}_f$, is normally distributed under the null \emph{in the limit},  the near-singularity  can render the weak convergence in question \emph{too slow} for $\cN(0, 1)$  to be accurate as a  reference null  distribution  for any reasonable sample size $n$. For instance, \figref{warped_size} shows that, under a near-singular one-factor model and a sample size of $100$ (which is already moderately large in practice),  for testing a tetrad,  the sizes of $T_f$ and $\hat{T}_f$ calibrated based on their asymptotic null distribution $\cN(0, 1)$ could be very ``off" from the target nominal levels. In fact, for this setting, the curve for $T_f$ doesn't align well with the $45^\circ$  line until $n$ is ramped up to around $400$, and that for $\hat{T}_f$ is still far from aligning.   Note that the one-factor model behind this figure translates into a near-singular $H$ because  the off-diagonal entries of $\Theta$ are $0.04$, which make the gradient vector $\nabla f(\Theta)$ close to zero. Such near-singular situations are arguably even more relevant than exactly singular situations, because unlike the set of singularities  $\{\Theta: \nabla f(\Theta) = 0\}$ which is confined to be a lower-dimensional subset of measure zero within the ambient space, covariance matrices $\Theta$ with the property $|\nabla f(\Theta)| \approx 0$ are prevalent occurrences. In addition, while  stochastic upper bounds on all the possible limiting null distributions of $\hat{T}_f^2$ have been developed in \citep[Theorem 4.1]{dufour2017wald}, in principle they can only be leveraged to offer asymptotically conservative tests in the limit and offer no indictions on how the test sizes  behave in finite samples. 
In light of all these considerations, it would be preferable to have an alternative test to Wald   that is not only \emph{singarity-agnostic} with respect to   the  \emph{limiting distribution}, but also  with respect to the \emph{convergent speed}, in the sense that both aspects are immune to the degree of singularity    of $H$, whether it be regular, nearly singular or exactly singular. In passing, we remark that    \citet{mitchell2019hypothesis} is the only work we are aware of which explicitly addresses near-singularity issues in testing, but their paper concerns the likelihood ratio test which is not our current subject matter. 
This brings us to our recent work  \citet{sturma2022testing} that proposed to use an \emph{incomplete U-statistic} with a high-dimensional symmetric kernel to test a model defined by many functional constraints; for instance, 
to test a one-factor model as in \eqref{one_factor_model}, each element of the kernel is constructed as an unbiased estimator of a given tetrad in \eqref{tetrad_forms}. The proposal was immediately inspired by the  works of \citet{chen2019randomized, SongChenKato2019} on normal approximation for high-dimensional incomplete U-statistics, which are in turn based on the seminal works of \citet{CCKHD13, CCKHD17} on sums of high-dimensional  independent vectors, but could also trace back to a one-dimensional  result in the 80's by \citet[Corollary 1]{janson1984asymptotic}. The latter result implies that, when suitably rescaled, an incomplete U-statistic with an appropriate \emph{computational budget} always converges weakly to the standard normal distribution, regardless of the underlying \emph{degeneracy} status of the kernel; see  \eqref{janson_convg} below. Hence, the main thrust of \citet{sturma2022testing} was to offer a singularity-agnostic test  \emph{with respect to the limiting distribution}. However, while  \citet{sturma2022testing} expanded upon the techniques of \citet{chen2019randomized, SongChenKato2019} to prove a \emph{Berry-Esseen-type}  bound for the accuracy of the high-dimensional normal approximation under  a ``mixed degeneracy" condition, the bound  there was not qualitatively refined enough to  describe another favorable attribute of the  incomplete U-statistic construction -- that it has shown every sign that it is also singularity agnostic \emph{with respect to the convergent speed}, by demonstrating stable and tight test size control in simulation studies. As an aside,  obtaining sharper high-dimensional Berry-Esseen bounds than those originating  in the pioneering works of \citet{CCKHD13, CCKHD17} is an active area of research;  see \citet{fangkoike2021, cckk2022, lopes2022, ccky_2023} for the  state-of-the-art. 

\subsection{New contributions}

In this paper, focusing on a  single polynomial constraint hypothesis on a normal covariance matrix as in \eqref{null_poly_contraint_eq}, we  offer new  insights into  testing  in the  presence of (near-)singularities. Precisely,  we will 
\begin{enumerate}
\item  unravel exactly  how  the limiting normal behavior of $T_f$ is plagued by its slow convergent speed  in the hairy near-singular scenarios, as measured by a Berry-Esseen bound for the underlying U-statistic (\secref{predicament}), and 
\item develop a refined Berry-Esseen (B-E) bound (\thmref{main}) capable of describing  the fact  that our  proposed incomplete U-statistic, when  normalized by its theoretical limiting variance,   is indeed singularity-agnostic \emph{with respect to the convergent speed} when the computational budget is suitably chosen  (\secref{main_results}). 
\end{enumerate}
As it will become  clear later,   these two endeavors complement each other, and the word ``convergent speed" above  encompasses more than just how the distributional approximation accuracy scales with the sample size $n$. In \figref{warped_size}, the two curves in light and dark green align nicely with the $45^\circ$  line, and they respectively  plot the test sizes of $\sqrt{n}U_{n, N}'/\sigma$ and $\sqrt{n}U_{n, N}'/\hat{\sigma}$, which are versions of our proposed incomplete U-statistic normalized by its theoretical limiting variance $\sigma^2$ and the practical data-driven Studentizer $\hat{\sigma}^2$ suggested in \citet[Section 2.3]{sturma2022testing},  respectively. While our main Berry-Esseen bound in \thmref{main} is proven for the former, we believe it sheds light on the excellent test-size control of the latter that we have experienced in practice, and is an important first step towards theoretically justifying   incomplete U-statistics as a strong singularity-agnostic choice for testing $H$. 
Its proof  builds upon the ``iterative conditioning" arguments pioneered by \citet{chen2019randomized, SongChenKato2019} for establishing incomplete U-statistic Berry-Esseen (B-E) bounds, but contains  a   game-changing innovation that differentiates from these prior works: Unlike the aforementioned  works where a    \emph{uniform} B-E bound for the \emph{complete U-statistic}  figures as a proof ingredient,  in order to 
elicit
 a final B-E bound for our incomplete U-statistic that is truly descriptive of its  singularity-agnostic nature,  we have instead used a    \emph{nonuniform} version to crucially tighten the arguments via ``completing the square" (\secref{pf_main_part_one}). 
 Moreover, to successfully pull off the said trick, the nonuniform bound itself amalgamates   some of the finest exponential-type inequalities  in quite an unexpected way, including the moment inequalities of \citet{gine2000exponential, adamczak2006moment}  for canonical U-statistics, the  recent  Bernstein inequalities for  sums of independent sub-Weibull random variables \citep{kuchibhotla2022moving, zhang2022sharper}, as well as the exponential randomized concentration inequalities in the Stein's method literature   \citep{shao2010icm, leung2023another, leung2023nonuniform}.

\subsection{Outline and notation} \label{sec:outline_notation}

\secref{tale} first recaps our proposed incomplete U-statistic in \citet{sturma2022testing}  for testing $H$. Following that, we theoretically decipher how near-singularity can adversely affect the convergent speed of the Wald test, and intuitively explains why one should expect our incomplete U-statistic to be a strong choice for testing $H$ that strikes a balance between being singularity-agnostic and efficient use of the data for power. \secref{main_results} presents our main result and the major parts of its proof. \secref{Bernstein_pf} establishes Bernstein-type tail bounds for the canonical projections in the Hoeffding decomposition of the complete U-statistic, which are useful for rendering the exponential-type nonuniform B-E bound used in  \secref{main_results}. \secref{discussion} concludes this work with a discussion of open issues.

In terms of notation, $P(\cdot)$, $\bE[\cdot]$ and  $\var[\cdot]$ are generic probability, expectation and variance operators, and $P_{H}(\cdot)$, $\bE_{H}[\cdot]$ and  $\var_{H}[\cdot]$ specifically denote these   taken under the null hypothesis $H$ in \eqref{null_poly_contraint_eq}. Given $L \in \{1, \dots, n\}$ and a function $t:(\bR^p)^L \rightarrow \bR$ in $L$  arguments in $\bR^p$, we may  use $\bE[t]$ as shorthand for $\bE[t(\bX_1, \dots, \bX_L)]$ for the data in \eqref{iid_data_sample}, or $\bE_{H}[t]$ if the expectation is taken under $H$. For any $q \geq 1$,  we use $\|Y\|_q = ( \mathbb{E}|Y|^q)^{1/q}$ to denote the $L_p$-norm of  any real-valued random variable $Y$, and $\|Y\|_{q, H}$ is  defined by an expectation taken under $H$. $|\cdot|$ is reserved for denoting the vector norm in the Euclidean space (as  in $|\nabla f(\Theta)|$ above).  Given any natural number $r \in \mathbb{N}$, we will use the shorthand $[r]$ for the set  $\{1, \dots r\}$; moreover, if $r \leq n$,  we let 
$I_{n, r} \equiv \{ (i_1, \dots, i_r): 1 \leq i_1 < \dots < i_r \leq n\}$. $\phi(\cdot)$ and $\Phi(\cdot)$ are respectively the standard normal density and distribution functions, and $\barPhi(\cdot) \equiv 1 - \Phi(\cdot)$; 
$\cN(\mu, v^2)$ and $\cN_r ({\boldsymbol \mu}, V)$  respectively denote the univariate normal distribution with mean $\mu$ and variance $v^2$ and the $r$-variate multvariate normal distributions with mean ${\boldsymbol \mu}$ and covariance matrix $V$.

 $C, c, C_1, c_1, C_2, c_2 \dots$   denote   absolute positive constants, where ``absolute" means they are universal for the statements they appear in; absolute constants likey $C(r)$, $c(r)$, $C_1(r)$, $c_1(r)$, etc. are used to emphasize their exclusive dependence on another quantity such as $r$.  These  constants generally differ in values at different occurences.
 
\section{Why incomplete U-statistics for $H$?} \label{sec:tale}

We shall first recap the construction of our proposed incomplete U-statistic in \citet{sturma2022testing} for testing $H$.

\subsection{Incomplete U-statistic for $H$} 

 Similarly to \citet{sturma2022testing}, a real polynomial $f$ in $\Theta$ of  degree $m$ is first written as 
\begin{equation} \label{form_of_f}
f(\Theta) = a_0 + \sum_{r=1}^m  \sum_{ 1 \leq u_i \leq v_i \leq p \text{ for } i \in [r] }a_{u_1 v_1, \dots, u_r   v_r}  \theta_{u_1 v_1} \dots \theta_{u_r v_r},
\end{equation}
where $a_0$ and  $a_{u_1v_1, \dots, u_r v_r} $'s are the real coefficients; this notation for a multivariate polynomial is somewhat less efficient than the usual multi-index notation because a pair $u_i \leq v_i$ may appear repeatedly in a given combination for $u_1 v_1, \dots, u_r  v_r$, but is adopted for the ease of defining our kernel functions below. 
Recall $\hat{\theta}_{uv}$ in \eqref{uv_prod_fn}, and let $\breve{h}_f: (\mathbb{R}^p)^m  \longrightarrow \mathbb{R}$
be the function in $m$ vector arguments ${\bf x}_1, \dots, {\bf x}_m \in \bR^p$ defined by
\begin{multline} \label{breve_H}
	\breve{h}_f({\bf x}_1, \dots, {\bf x}_m) =  a_0 + \sum_{r=1}^m  \sum_{  1 \leq u_i \leq v_i \leq p \text{ for } i \in [r] }a_{u_1v_1, \dots, u_r v_r} \hat{\theta}_{u_1 v_1}({\bf x}_1) \dots \hat{\theta}_{u_r v_r } ({\bf x}_r),
	\end{multline}
which forms an \emph{unbiased} estimator for $f(\Theta)$ upon  inserting the i.i.d. data from \eqref{iid_data_sample} as $\breve{h}_f({\bf X}_1, \dots, {\bf X}_m)$.
If $\cS_m$ is the set of all permutations of $(1, \dots, m)$, we can then symmetrize $\breve{h}_f$ with respect to its arguments by  defining the function
\begin{equation}\label{polynomial_kernel}
h({\bf x}_1, \dots, {\bf x}_m) = h_f({\bf x}_1, \dots, {\bf x}_m) = \frac{1}{m!} \sum_{\pi \in \cS_m} \breve{h}_f({\bf x}_{\pi_1}, \dots, {\bf x}_{\pi_m}),
\end{equation}
to serve as a \emph{kernel}  for  the U-statistic
\begin{equation} \label{complete_U_stat}
U_n = 
 {n \choose m}^{-1}\sum_{(i_1, \dots, i_m) \in I_{n, m}}  h({\bf X}_{i_1}, \dots, {\bf X}_{i_m}),
\end{equation}
which is also unbiased for $f(\Theta)$; for a classical reference on U-statistic theory, we refer the reader to \citet{korolyuk2013theory}.
 An important  function related to $h(\cdot)$ is the \emph{projection kernel} $g:\bR^p \longrightarrow \bR$ defined by 
 \begin{multline} \label{g_fn_structure}
g({\bf x}) \equiv \bE[h({\bf x}, {\bf X}_2, \dots, {\bf X}_m)] = a_0 +\\
  \frac{1}{m}\sum_{r=1}^{m} \;\ \sum_{j =1}^r  \;\ \sum_{ \substack{ 1 \leq u_i \leq v_i \leq p \\  \text{ for } i \in [r]} }  a_{u_1 v_1, \dots, u_r v_r }  \theta_{u_1 v_1} \dots \theta_{u_{j-1} v_{j-1}} \hat{\theta}_{u_j v_j} (\bx)\theta_{u_{j+1} v_{j+1}} \dots \theta_{u_r v_r}
\end{multline}
For an integer $N$ such that $1 \leq N \leq {n \choose m}$,  let  $\{Z_{(i_1, \dots, i_m)}\}_{(i_1, \dots, i_m) \in  I_{n, m}}$ be ${n \choose m}$ many independent and identically distributed Bernoulli random variables with success probability 
$
p_{n, N} \equiv N{n \choose m}^{-1}
$
such that 
\[
P(Z_{(i_1, \dots, i_m)} = 1) = 1- P(Z_{(i_1, \dots, i_m)} = 0) = p_{n, N};
\] 
moreover, these $Z_{(i_1, \dots, i_m)}$'s are also independent of the data ${\bf X}_1, \dots, \bX_n$ from \eqref{iid_data_sample}. 
In \citet[Section 4]{sturma2022testing}, we have proposed  using the 
\emph{incomplete} U-statistic under \emph{Bernoulli sampling}
\begin{equation}\label{icu_def}
U'_{n, N} =  \frac{1}{\hat{N}}\sum_{(i_1, \dots, i_m) \in I_{n, m}} Z_{(i_1, \dots, i_m)} h({\bf X}_{i_1}, \dots, {\bf X}_{i_m}),
\end{equation}
to test   \eqref{null_poly_contraint_eq},
where  
$\hat{N} \equiv \sum_{(i_1, \dots, i_m) \in I_{n, m}} Z_{(i_1, \dots, i_m)}$ counts the actual number of $h({\bf X}_{i_1}, \dots, {\bf X}_{i_m})$ being summed over; we call $N$ the \emph{computational budget} that induces $\hat{N}$. 
A motivation of our proposal could trace back to a weak convergence result  in   \citet[Corollary 1]{janson1984asymptotic}, which implies that, if $N, n$ tend to $\infty$ in such a way that $\alpha_n \equiv n/N$ converges to a positive constant, then for $m \geq 2$,
 \begin{equation} \label{janson_convg}
\frac{\sqrt{n}U_{n, N}'}{\sigma} \longrightarrow_d \cN(0,   1) \text{ under } H,
\end{equation}
where $\sigma_g^2 \equiv \var_{H}[g({\bf X})] = \bEH[g^2]$, $\sigma_h^2 \equiv \var_{H}[h(\bX_1, \dots, \bX_m)] = \bEH[h^2]$ and
\begin{equation} \label{rescaling_factor}
\sigma^2 \equiv m^2  \sigma_g^2+  \alpha_n \sigma_h^2.
\end{equation}
With  careful calculations, one can check (\lemref{facts}$(i)$) that 
 \begin{equation} \label{sigma_g2_equal_wald_variance}
\sigma_g^2 = \frac{  (\nabla f (\Theta))^T V(\Theta) \nabla f (\Theta)}{m^2},
\end{equation}
 and hence 
 \begin{equation} \label{singular_condition_as_Ustat}
 \text{$H$ is singular if and only if $\sigma_g^2 = 0$}
 \end{equation} 
given that $V(\Theta)$ is strictly positive definite; on the contrary, unlike $\sigma_g$,  the standardizing factor $\sigma$ for $U_{n, N}'$
is always strictly positive, whether $H$ is singular or not, because it must be that $\sigma_h^2 > 0$  under our assumptions (\lemref{facts}$(ii)$).  That the weak limit  in \eqref{janson_convg} is always $\cN(0,1)$ suggests that a test of $H$ based on $U_{n, N}'$ is  singularity-agnostic \emph{with respect to the limiting distribution}. In practice, a data-driven variance estimate $\hat{\sigma}^2$  formed with a ``divide-and-conquer" strategy can be used in place of $\sigma^2$ to normalize $U_{n, N}'$; we refer to  \citet[Section 2.3]{sturma2022testing} for the details.
 
 However, the asymptotic result in \eqref{janson_convg} says nothing about its convergent speed, most naturally gauged by Berry-Esseen-type bounds. 
 To fully appreciate our main result for $U_{n, N}'$ (\thmref{main}) in that respect, we shall first explore why the Wald statistic  with normal quantiles as critical values cannot be expected to perform well when $\nabla f(\Theta) \neq 0$ but ``near singularities" hit, even the classical theory seemingly suggests otherwise by  guaranteeing its limiting normality.
 

 \subsection{The predicament of the Wald test} \label{sec:predicament}
To illustrate our point, we consider the  statistic $T_f$ in \eqref{wald_stat_standardized}
 normalized by the theoretical limiting variance $ (\nabla f (\Theta))^T V(\Theta) \nabla f (\Theta)$, whose convergence to normality under a regular $H$ is guaranteed by the classical theory \citep[Chapter 14.4]{TSH2022}. With simple calculations,  one can   show that
\begin{equation} \label{Wf_approx_by_Un}
T_f  = \frac{\sqrt{n}U_n }{m \sigma_g} + o_p(1),
\end{equation}
i.e. $T_f$ is  essentially the  U-statistic in \eqref{complete_U_stat} up to  a small term converging to zero in probability; see  \appref{leading_term_of_Wf}. Hence, in light of   \eqref{singular_condition_as_Ustat}, the normal weak limit of $T_f$ under a regular  $H$ is essentially driven by the  central limit theorem for complete U-statistics \citep[Theorem 4.2.1]{KB1994Ustat}, which states that 
\begin{equation} \label{complete_Ustat_convergence}
\frac{\sqrt{n}U_n }{m \sigma_g} \longrightarrow_d \cN(0, 1) \text{ when } \sigma_g \neq 0.
\end{equation}
 The normal approximation accuracy of \eqref{complete_Ustat_convergence} is typically gauged by the  Berry-Esseen bound for U-statistics  \citep[Theorem 10.3]{chen2010normal} as in 
\begin{equation} \label{BEustat}
\sup_{z \in \bR}\left| P_{H}\left(\frac{\sqrt{n}U_n}{m \sigma_g} \leq z\right) - \Phi(z)\right| \leq \frac{ (1 + \sqrt{2}) (m-1) \sigma_h}{\sqrt{ m (n - m + 1)}\sigma_g} + \frac{6.1 \bE_{H} [|g|^3]}{n^{1/2} \sigma_g^3 }.
\end{equation}
Our \emph{key observation}  is that the ratio of moments
\begin{equation} \label{Achilles_heel}
\frac{\sigma_h}{\sigma_g}
\end{equation}
figuring in the bound \eqref{BEustat} could be extremely large for the  kernel $h$ in \eqref{polynomial_kernel} when $H$ approaches being singular, so much so  that  the   factor $\sqrt{ m (n - m + 1)}$   in the denominator could fail to counter-balance it even with a  large $n$. When this happens, the  vacuous B-E bound in \eqref{BEustat} is actually informing  that the distribution of $\frac{\sqrt{n}U_n}{m \sigma_g}$, and hence that of $T_f$ via the first-order approximation in \eqref{Wf_approx_by_Un},  could be very poorly approximated by $\cN(0, 1)$. 
In what follows, we  shall use  Isserlis' theorem \citep{isserlis1918formula, wick1950}  on mixed normal central moments   to explain this phenomenon.


\begin{theorem}[Isserlis' theorem] \label{thm:isserlis}

For $r \in \mathbb{N}$, let  $Y_1, \dots, Y_{2r}$ be  jointly normal variables that are all centered. Then 
\begin{equation} \label{isserlis_expansion_generic}
\bE[Y_1 \cdots Y_{2r}] = \sum_{\cJ} \prod_{\ell = 1}^r \bE[Y_{u_\ell} Y_{v_\ell}],
\end{equation}
where the summation is over all  partitions $\cJ = \big\{\{u_1, v_1\}, \dots, \{u_r, v_r\}\big\}$ of $[2r] = \{1, \dots, 2r\}$ into disjoint pairs $\{u_\ell, v_\ell\}\in [2r]^2$; in  particular, there are $\ (2r)!/ (2^r r!)$ distinct pairings of $[2r]$. 
\end{theorem}

This result can be found in \citet[Theorem 1.28]{jason_gauss_hilbert_space_1997}, and we will call a summand in \eqref{isserlis_expansion_generic}   an \emph{Isserlis expansion term}. (Also note that the form  in \eqref{V_structure} can be derived from \thmref{isserlis}.) It is the easiest to illustrate our claimed phenomenon about the ratio in \eqref{Achilles_heel} with a prototypical example: the tetrad. In what follows, suppose we are to test the tetrad equality constraint  
\begin{equation} \label{specific_tetrad}
f(\Theta) = \theta_{uz}\theta_{vw} - \theta_{uw}\theta_{vz} = 0 \text{ for } u \neq v \neq w \neq z,
\end{equation}
 which has  
a corresponding degree-two kernel
\begin{equation} \label{Hun_for_tetrad}
h({\bf x}_1,{\bf x}_2) =
\frac{1}{2} \Bigg[
\hat{\theta}_{uz}(\bx_1) \hat{\theta}_{vw}(\bx_2)  -  \hat{\theta}_{uw}(\bx_1) \hat{\theta}_{vz}(\bx_2)+ 
\hat{\theta}_{vw}(\bx_1)  \hat{\theta}_{uz}(\bx_2)  -  \hat{\theta}_{vz}(\bx_1) \hat{\theta}_{uw}(\bx_2)
  \Bigg]
\end{equation} 
  from  \eqref{polynomial_kernel}. 
Moreover,  for simplicity, we further assume
 the covariance matrix $\Theta$ has a specific one-factor structure such that
\begin{equation} \label{equicorrelation_assump}
\theta_{uv}  =  \begin{cases} 
 1& \text{if } u = v;\\
 \rho   & \text{if } u \neq v,
  \end{cases}
\end{equation}
for a constant $\rho \in (0, 1)$, which makes the gradient of \eqref{specific_tetrad} take the form
\begin{equation} \label{tetrad_gradient}
\nabla f (\Theta)
= (\theta_{vw}, \theta_{uz}, - \theta_{vz}, -\theta_{uw})^T = \rho (1, 1, - 1, - 1)^T;
\end{equation}
compare \eqref{one_factor_model}.  The following observations can be made for this example:

\subsubsection{First observation for the tetrad example \eqref{specific_tetrad}: $\sigma_g^2 \rightarrow 0$ as $|\nabla f(\Theta)| \rightarrow 0$} \label{sec:1st_observation}
This is just a simple consequence of the expression for $\sigma_g^2$ in  \eqref{sigma_g2_equal_wald_variance}, the form of the tetrad gradient \eqref{tetrad_gradient} and the structure of $V(\Theta)$ in \eqref{V_structure}, as $|\nabla f (\Theta)| \rightarrow 0$ is equivalent to $\rho \rightarrow 0$. 

 \subsubsection{Second observation for the tetrad example \eqref{specific_tetrad}: $\sigma_h^2 \rightarrow 1$ as $|\nabla f(\Theta)| \rightarrow 0$} \label{sec:2nd_observation}
Upon expanding $h^2({\bf X}_1, \bX_2)$, the variance quantity $\sigma_h^2$ can be seen as
\begin{multline} \label{sigmah_sq_expand}
\sigma_h^2 = \bE_H[h(\bX_1, \bX_2)^2] = \frac{1}{4} \times \Bigg\{ \bE \big[\hat{\theta}_{uz}(\bX_1)^2 \hat{\theta}_{vw}(\bX_2)^2 \big] +  \bE[\hat{\theta}_{uw}(\bX_1)^2 \hat{\theta}_{vz}(\bX_2)^2 ]+ \\
\bE[ \hat{\theta}_{vw}(\bX_1)^2 \hat{\theta}_{uz}(\bX_2)^2] 
+ \bE[ \hat{\theta}_{vz}(\bX_1)^2 \hat{\theta}_{uw}(\bX_2)^2 ] 
+ \underbrace{\cdots\cdots\cdots }_{\text{ $2 \cdot {4 \choose 2}$  expected cross terms}}\Bigg\},
 \end{multline}
where each expectation explicitly shown in the braces of \eqref{sigmah_sq_expand} is the square of one of the four terms
\begin{equation}\label{four_random_terms}
\hat{\theta}_{uz}(\bX_1) \hat{\theta}_{vw}(\bX_2), 
\quad  \hat{\theta}_{uw}(\bX_1) \hat{\theta}_{vz}(\bX_2),
 \quad  \hat{\theta}_{vw}(\bX_1) \hat{\theta}_{uz}(\bX_2)
  \text{ and }  \hat{\theta}_{vz}(\bX_1) \hat{\theta}_{uw}(\bX_2) 
\end{equation}
that come from plugging the data $(\bX_1, \bX_2)$ into the four additive terms in the bracket of \eqref{Hun_for_tetrad}, and the terms hidden in ``$\cdots \cdots \cdots$" are the expectations of the cross terms from multiplying any two terms  in \eqref{four_random_terms}. In particular, with \thmref{isserlis},  it is straight-forward to compute that
\begin{multline} \label{form_of_isserlis_expansion_term}
 \text{ each expectation  explicitly shown in \eqref{sigmah_sq_expand} must be evaluated as}\\
1 +  \underbrace{ \Big\{\substack{ \text{ $\frac{8!}{2^4 4!} - 1$ finitely many   additive terms,   each equal to  either $0$, $\rho^2$ or $\rho^4$} } \Big\}}_{\text{ summing to $4(\rho^4 + \rho^2)$}}.
\end{multline}
For instance, the first expectation $\bE[\hat{\theta}_{uz}(\bX_1)^2 \hat{\theta}_{vw}(\bX_2)^2]$ in \eqref{sigmah_sq_expand}
 can be computed with Isserlis' theorem as
\begin{equation} \label{show_exact_isserlis_term}
\bE[\hat{\theta}_{uz}(\bX_1)^2 \hat{\theta}_{vw}(\bX_2)^2]  = \bE[X_{1u}^2X_{1z}^2 X_{2v}^2X_{2w}^2]  =  \underbrace{ \bE[ X_{1u}^2  ] \bE[X_{1z}^2] \bE[ X_{2v}^2 ] \bE[X_{2w}^2]}_{ = 1\cdot 1 \cdot 1 \cdot 1 = 1} + \cdots
\end{equation}
under \eqref{equicorrelation_assump}, where ${\bf X}_i = (X_{i1}, \dots, X_{ip})^T$ for each $i \in [n]$;
we have only  displayed the Isserlis expansion term that takes the value $1$ and corresponds to the pairwise partition where each of  the  sample indices (i.e.  $1$ or $2$) \text{ and }  variable indices (i.e. $u$, $v$, $w$ or $z$)  \emph{exactly pairs with itself}. The rest of the Isserlis expansion terms not displayed,  for which such self-pairings are not simultaneously satisfied,  must either be  $0$, $\rho^2$ or $\rho^4$ under the structure of $\Theta$ assumed in \eqref{equicorrelation_assump}, as described in \eqref{form_of_isserlis_expansion_term};  examples of such terms are
\begin{multline*}
 \bE[X_{1u}X_{1z}] \bE[ X_{1u}X_{1z}] \bE[ X_{2v}^2 ] \bE[X_{2w}^2]  =  \rho \cdot \rho \cdot 1 \cdot 1 = \rho^2, \\
 \bE[X_{1u}X_{1z}] \bE[ X_{1u}X_{1z}] \bE[ X_{2v} X_{2w} ] \bE[ X_{2v}  X_{2w}]  =  \rho \cdot \rho \cdot \rho  \cdot \rho = \rho^4 \text{ and }\\
 \bE[X_{1u} X_{2v}]  \bE[X_{1u} X_{2v}] \bE[X_{1z}^2] \bE[X_{2w}^2] = 0 \cdot 0 \cdot 1 \cdot 1 =  0,
 \end{multline*}
which result from different  configurations of the index pairings and can be shown as summing to $4(\rho^4 + \rho^2)$ by simple counting. At the same time, one can also  show that
\begin{multline} \label{cannot_be_bigger_than_rho}
\text{ each of the $2 \cdot {4 \choose 2}$ cross terms not explicitly displayed in \eqref{sigmah_sq_expand} must be} \\
\text{ a sum of
$\frac{8!}{2^4 4!}$ many terms, each with magnitude no larger than $\rho$};
\end{multline}
this is because simultaneous self-pairings of all sample and variable indices as in the first expansion term (with value $1$)  in \eqref{show_exact_isserlis_term} cannot occur. 
For instance,  taking the expectation of the cross term from multiplying the first two terms in \eqref{four_random_terms} and applying \thmref{isserlis} gives
\begin{align*}
\bE[\hat{\theta}_{uz}(\bX_1) \hat{\theta}_{vw}(\bX_2)\hat{\theta}_{uw}(\bX_1) \hat{\theta}_{vz}(\bX_2) ]
&=\bE [X_{1u}^2 (X_{1w} X_{1z} )X_{2v}^2(X_{2w}X_{2z})] \\
&= \underbrace{\bE [X_{1u}^2] \bE[ X_{1w} X_{1z} ] \bE[X_{2v}^2] \bE[X_{2w}X_{2z}]}_{ = 1\cdot \rho \cdot 1 \cdot \rho = \rho^2} + \cdots
\end{align*}
for which all the Isserlis expansion terms cannot form exact self-pairings of all the indices to give the value $1$. In light of both  facts in  \eqref{form_of_isserlis_expansion_term} and \eqref{cannot_be_bigger_than_rho}, one can then conclude from  \eqref{sigmah_sq_expand} that $\sigma_h^2 \rightarrow 1$ as $\rho \rightarrow  0$ (which is equivalent to $|\nabla f(\Theta)| \rightarrow 0$).

The observations in  \secsref{1st_observation} and \secssref{2nd_observation} have essentially shown the following: for the tetrad, the ratio $\sigma_h/\sigma_g$ in  \eqref{Achilles_heel} can be arbitrarily large as the hypothesis $H$ approaches being singular.  However, there is nothing special about this example which we have  picked only for its simplicity. The above phenomenon can also be observed for other  complicated polynomials, and can be similarly explained when applying Isserlis' theorem to evaluate $\sigma_h^2$ and $\sigma_g^2$; in short, ``simultaneous self-pairing" of indices prevents $\sigma_h$ from tending to zero even when the gradient $\nabla f (\Theta)$ (and hence $\sigma_g$) approaches zero.

\subsection{Incomplete U-statistic should strike a fine  balance} \label{sec:strike_balance}
While \secref{predicament}  demonstrated how the normal distributional approximation accuracy of  $T_f$ can be  plagued by the largeness of the ratio in \eqref{Achilles_heel} under near singularity, we now explore why $U_{n, N}'$ is  a  good candidate to circumvent this issue.

Interestingly,  we shall first consider the following average of independent terms:
\begin{equation}\label{Snm}
S_{\lfloor n/m\rfloor} = \frac{1}{\lfloor n/m \rfloor}\sum_{i=1}^{\lfloor n/m \rfloor} h ({\bf X}_{(i-1)m +1}, \dots, {\bf X}_{im}),
\end{equation}
where $h$ is the same kernel in \eqref{polynomial_kernel}, and  the summands are constructed with non-overlapping batches of samples  and are hence independent. By the classical central limit theorem, it must be the case that, under $H$, 
\begin{equation} \label{indep_sum_convg}
\frac{\sqrt{\lfloor n/m\rfloor}S_{\lfloor n/m\rfloor} }{\sigma_h} \longrightarrow_d \cN(0,1) \text{ whether }\sigma_g \text{ equals } 0 \text{ or not};
\end{equation}
recall $\sigma_h^2 >0$. Moreover,  the  classical  B-E bound  \citep{berry1941accuracy, esseen1942} says that  for an absolute constant $C >0$, 
\begin{equation} \label{BEindep}
\sup_{z \in \bR}\left| P_{H}\left(\frac{\sqrt{\lfloor n/m\rfloor}S_{\lfloor n/m\rfloor}}{\sigma_h}  \leq z\right) - \Phi(z)\right| \leq \frac{C \bE_{H}[|h|^3]}{\sqrt{\lfloor n/m\rfloor} \sigma_h^3}.
\end{equation}
At this point, it is appropriate to introduce the following  result about the \emph{hypercontractivity} of a polynomial in jointly normal random variables, which refers to the fact that the lower norms of such a variable can control its  higher norms up to an absolute constant. 
\begin{lemma}[Hypercontractivity of polynomials in jointly normal random variables] \label{lem:hypercontraction}
Let $\bG =
 (g_1, \dots, g_r)^T \sim \cN_r(0, \Sigma)$ be an $r$-variate centered normal random vector, and $Q(\bG) = Q(g_1, \dots, g_r)$ be a polynomial of degree $k$ in the variables $g_1, \dots, g_r$. Then for $1 < q_1 \leq q_2 < \infty$ and an absolute constant $C(k) >0$ depending only on $k$, 
\[
\|Q({\bG})  \|_{q_2} \leq C(k) \bigg(\frac{q_2 - 1}{q_1 -1}\bigg)^{k/2} \| Q({\bG}) \|_{q_1}. 
\]
\end{lemma}
\noindent \emph{Remark}:  \lemref{hypercontraction} is a direct consequence  of the Khinchine-Kahane inequality for Rademacher chaos \citep[Theorem 3.2.5, p.115]{de2012decoupling}   and  will  be used repeatedly in the sequel; we also note that the original version stated in \citet[Theorem 3.2.10, p.118]{de2012decoupling} assumes $\Sigma$ to be the identity matrix,  but our version here is an immediate consequence because one can write $\bG = \Sigma^{1/2}\bZ$ for a standard normal vector $\bZ$, and it is easy to see that $Q(\bG) = Q(\Sigma^{1/2}\bZ)$ is also a degree-$k$ polynomial in $\bZ$. 

We can now see why $S_{\lfloor n/m\rfloor}$ is  well approximated by a normal distribution even if $\sigma_g$ is very close to zero: Since  $h(\bX_1, \dots, \bX_m)$ is a polynomial of degree $2m$ in jointly Gaussian random variables  by  the construction in \eqref{breve_H}-\eqref{polynomial_kernel}, its hypercontractivity  from  \lemref{hypercontraction} implies that, for an absolute constant $C(m) > 0$,  
\begin{equation}  \label{hypercontractivity_of_h}
\frac{\bE_{H}[|h|^3]}{\sigma_h^3}  \leq C(m).
\end{equation} 
Thus, we can further write the B-E bound in \eqref{BEindep} as
\begin{equation} \label{singular_agnostic_speed_Snm}
\sup_{z \in \bR}\left| P_{H}\left(\frac{\sqrt{\lfloor n/m\rfloor}S_{\lfloor n/m\rfloor}}{\sigma_h}  \leq z\right) - \Phi(z)\right| \leq \frac{C(m) }{\sqrt{n}};
\end{equation}
 unlike the $U_n$ in \eqref{complete_U_stat}, which ultimately drives the asymptotic distribution of $T_f$ and whose normal approximation accuracy  in \eqref{BEustat} could be plagued by the large ratio $\sigma_h/\sigma_g$ when $H$ is nearly singular, the normal approximation of $S_{\lfloor n/m\rfloor}$ is robust no matter what  $\sigma_g$ is. As a side note, since the projection kernel $g({\bX})$ is always  a degree-2 polynomial in Gaussian random variables as per its structure in \eqref{g_fn_structure}, \lemref{hypercontraction} also suggests that the bound  
 \begin{equation} \label{hypercontractivity_of_g}
\frac{\bE_{H} [|g|^3]}{ \sigma_g^3 } \leq C
 \end{equation}
 for an absolute constant $C >0$; unlike the ratio $\frac{\sigma_h}{\sigma_g}$, $\frac{E_{H} [|g|^3] }{\sigma_g^3}$ is not the source of poor normal approximation for $U_n$ despite appearing in the U-statistic B-E bound
 \eqref{BEustat}. The moral  here is that  $S_{\lfloor n/m\rfloor}$ benefits from being a sum of the \emph{independent} hypercontractive terms $\{ h ({\bf X}_{(i-1)m +1}, \dots, {\bf X}_{im})\}_{i =1, \dots, \lfloor n/m \rfloor}$ and 
 has good normal approximation property as suggested in \eqref{singular_agnostic_speed_Snm}. On the contrary, the complete U-statistic $U_n$ in \eqref{complete_U_stat} is a sum of similarly  hypercontractive terms $\{h({\bf X}_{i_1}, \dots, {\bf X}_{i_m})\}_{(i_1, \dots, i_m) \in I_{n, m}}$, but the intricate dependence among them cripples the normal approximation accuracy by injecting the ratio $\frac{\sigma_h}{\sigma_g}$ in the B-E bound \eqref{BEustat}; in fact, an examination of the proof behind \eqref{BEustat}  \citep[Section 10.2.1]{chen2010normal} precisely suggests that $\frac{\sigma_h}{\sigma_g}$ stems from such dependence. 
 
However, despite being  (pleasantly)  
singularity-agnostic with respect to   the convergent speed, 
$S_{\lfloor n/m\rfloor}$ pays the price of essentially dividing down the sample size by  $m$, while $U_n$, which drives the distribution of $T_f$, is aggressive in terms of power by summing over all the ${n \choose m}$ sample indices in $I_{n, m}$;  see \citet[Theorem 5.1.4, p.176]{serfling1980} for an optimality result about  complete U-statistics regarding the latter aspect. Intuitively, one should expect, by summing  over only a random  sample of indices in $I_{n,m}$ as in \eqref{icu_def} with an appropriately chosen computational budget $N$,  $U_{n, N}'$ shall \emph{strike a balanced middle ground} between $S_{\lfloor n/m\rfloor}$ and $U_n$, i.e., summing over more terms than $S_{\lfloor n/m\rfloor}$  on average to be more power efficient,  but  carrying less dependence between the summands than $U_n$ to maintain the normal approximation accuracy under near-singular situations.

\section{A singularity-agnostic Berry-Esseen bound for $U_{n, N}'$}  \label{sec:main_results}

We now state our main result: a  Berry-Esseen bound for the normal convergence in \eqref{janson_convg} that truly mathematizes  our intuition at the end of \secref{strike_balance} about $U_{n, N}'$. We call the bound ``singularity-agnostic" because it captures the fact that the convergent speed of  \eqref{janson_convg} is not affected by the underlying singularity status of $H$,  when the computational budget $N$ is appropriately chosen; recall that singularity can also be characterized by the zeroness of $\sigma_g$ as per \eqref{singular_condition_as_Ustat}.

\begin{theorem}[``Singularity-agnostic" Berry-Esseen bound for $U'_{n, N}$]\label{thm:main}
Let  $U_{n ,N}'$ be constructed as in \eqref{icu_def} with the kernel $h$ in  \eqref{polynomial_kernel} and  data from \eqref{iid_data_sample}; assume \eqref{non_singular_assumption_Theta}, $\log(n) \leq N$, $2 \leq m$ and $2m < n$ hold. Then one has the generic bound \\
\begin{equation} \label{generic_bdd_main}
 \sup_{z \in \bR}\Bigg| P_{H}\Big(\frac{\sqrt{n} U_{n, N}'}{\sigma} \leq z\Big) - \Phi(z) \Bigg| \leq \frac{C}{\sqrt{N}} \Bigg\{\frac{\sqrt{1 - 3p_{n, N} + 3p_{n, N}^2}}{1 - p_{n, N}} + \log n \Bigg\}
  +  \frac{C(m)}{\sqrt{n}} + \mathfrak{R},
\end{equation}
for an error term $\mathfrak{R} \geq 0$ with the property  that:
\begin{enumerate}
\item (Regular  case) If $\sigma_g \neq 0$,  then
\[
\mathfrak{R} \leq  C(m) \Bigg\{   \sum_{r=2}^m\frac{ \sqrt{N}(\log (2n^r +1))^{3r} }{\sqrt{ (1 - p_{n, N})n^{r} }} \Bigg\} + C \Bigg( \frac{p_{n, N}}{1- p_{n, N}} + \frac{1}{\sqrt{n}}\Bigg).
 \]

\item (Singular case) If $\sigma_g =  0$, then
\[
\mathfrak{R} \leq  \frac{C(m) \sqrt{N}(\log (2n^m +1 ))^{3m}}{n\sqrt{1 - p_{n, N}}} +  C  \frac{p_{n, N}}{1- p_{n, N}}.
 \]

\end{enumerate}
In particular, with the additional assumptions that  $n \leq N \leq mn$ and $2m+1 < n$,  we have the overall bound 
\begin{equation} \label{sing_agn_BE_bdd}
 \sup_{z \in \bR}\Bigg| P_{H}\Bigg(\frac{\sqrt{n} U_{n, N}'}{\sigma} \leq z\Bigg) - \Phi(z) \Bigg| \leq \frac{C(m) \log (2n^m +1 ))^{3m}}{\sqrt{n}},
\end{equation}
 regardless of the value of $\sigma_g$. 
\end{theorem}


 The  bound  
\eqref{sing_agn_BE_bdd}  is arguably our most important statement, 
as the  $C(m)$ in \eqref{sing_agn_BE_bdd} is a constant that \emph{only} depends on $m$. Just like  the B-E bound for $S_{\lfloor n/m\rfloor}$ in \eqref{singular_agnostic_speed_Snm}, it informs that the convergent speed of $U_{n, N}'$ to normality won't be severely slowed down by the singularity status of $H$, as opposed to its complete counterpart $U_n$ whose slow convergence is suggested by the largeness of the ratio in \eqref{Achilles_heel} under near-singular situations. At the same time, when $n \leq N \leq mn$,  $U_{n, N}'$  is summing over more terms on average than $n/m$,  achieving more efficient use of the data than $S_{\lfloor n/m\rfloor}$. 
  In particular, $N = 2n$ is a  ``balanced" choice for the computational budget advocated in \citet{sturma2022testing}, which is also adopted for the small simulation study behind \figref{warped_size}.  We also comment on three other noteworthy aspects:

\begin{enumerate} [(a)]

\item As $N$ approaches ${n \choose m}$, the ``success probability" $p_{n, N}$ of the  Bernoulli samplers tends to $1$; since $1  - p_{n, N}$ figures in the denominators,   the right hand side of the generic bound in \eqref{generic_bdd_main} will blow up. This is  natural:  the incomplete $U_{n, N}'$ should behave very much like  $U_n$ when $N$ is close to ${n \choose m}$, and that the bound becoming vacuous simply indicates that the normal approximation accuracy could be very low, much like the large ratio in \eqref{Achilles_heel} advises that the convergence in \eqref{complete_Ustat_convergence} could be slow  under near-singular situations. 

\item \thmref{main}$(i)$ is most comparable to  a B-E bound for one-dimensional incomplete U-statistic proven in \citet{peng2022rates}. Precisely, under the same assumptions and $\sigma_g \neq 0$, \citet[Theorem 5]{peng2022rates} suggests that  
\begin{align}
&\sup_{z \in \mathbb{R}} \left| P_{H}\left( \frac{ \sqrt{n}U_{n, N}' }{\sigma} \leq z \right) - \Phi(z)\right| \notag \\
&\leq C \Bigg\{\frac{\bE_{H} [|h|^3]}{\sqrt{N} \sigma_h^3} + \left(\frac{m}{n}\right)^{1/3} + \frac{\bE_{H}[|g|^3]}{\sqrt{n}  \sigma_g^3} + \left[ \frac{m}{n} \left( \frac{\sigma_h^2}{ \sigma_g^2m} - 1\right) \right]^{1/2} \Bigg\} \notag \\
&\leq C(m) \Bigg\{ \frac{1}{\sqrt{N}} +\frac{1}{n^{1/3}} + \left[ \frac{1}{n} \left( \frac{(\sigma_h/\sigma_g)^2}{ m} - 1\right) \right]^{1/2} \Bigg\} \label{peng_bdd}, 
\end{align}
where the last inequality is an immediate consequence of the hypercontractivity in  \eqref{hypercontractivity_of_h} and \eqref{hypercontractivity_of_g}. However, due to the existence of the ratio $\sigma_h/\sigma_g$ which can be extremely large as described in \secref{predicament} for a near-singular $H$, \eqref{peng_bdd} is not refined enough to capture the singularity-agnostic nature of $U_{n, N}'$  as in \eqref{sing_agn_BE_bdd}. Moreover, compared to the rate $n^{-1/3}$ in \eqref{peng_bdd}, our bound in \eqref{sing_agn_BE_bdd} has the rate $n^{-1/2}$ up to some logarithm factors, which is  near-optimal.
\item Our paper really concerns a non-trivial kernel $h$ with degree $m \geq 2$, and  \thmref{main} is also proven for this  case. If $m =1$, then $\sigma_h = \sigma_g$; from a singularity perspective, one can simply use $S_{\lfloor n/1\rfloor} = n^{-1}\sum_{i=1}^n h(\bX_i)$, which has the singularity-agnostic B-E bound \eqref{singular_agnostic_speed_Snm}.
\end{enumerate}


The proof of \thmref{main} involves  analyzing the tail behavior of the \emph{Hoeffding projections} of  $U_n$, which is one of our main technical innovations.   In what follows, if $P_1, \dots, P_m$ are $m$ distributions on $\mathbb{R}^p$ and $\Psi: (\bR^p)^m \longrightarrow \bR$ is a function with $m$ arguments in $\bR^p$, we use the notation $P_1 \times \dots \times P_m \Psi = \int \Psi({\bf x}_1, \dots, {\bf x}_m) d(P_1 \times \dots \times P_m)$.  Let $\delta_{\bf x}$ denote the Dirac measure at ${\bf x} \in \bR^p$. For each  $r \in [m]$ and a  \emph{symmetric} kernel $\Psi: (\bR^p)^m \longrightarrow \bR$,  let $\pi_r (\Psi): (\bR^p)^r \longrightarrow \bR$ be the Hoeffding projection kernel 
\begin{equation} \label{projection_fns}
\pi_r (\Psi)({\bf x}_1, \dots, {\bf x}_r) = (\delta_{{\bf x}_1} - P_{H}) \times \cdots \times (\delta_{{\bf x}_r} - P_{H})  \times P_{\Hf}^{m-r} \Psi,
\end{equation}
which is also symmetric in ${\bf x}_1, \dots, {\bf x}_r$ and is said to be \emph{canonical} (or \emph{completely degenerate}), in the sense that, for any $i = 1, \dots, r$
\begin{equation} \label{canonical_property_hoeffding_proj}
\int  \pi_r (\Psi) ({\bf x}_1, \dots, {\bf x}_r) dP_{\Hf}({\bf x}_i) = 0 \text{ almost surely}.
\end{equation}
In particular, for our interested kernel $h$ defined in \eqref{polynomial_kernel},  $\pi_1(h) = g$; moreover, 
if
\[
d \equiv \max \Bigg\{1 \leq r \leq m: \int h({\bf x}_1, \dots, {\bf x}_m) dP_{\Hf}^{m - r +1} = 0 \text{ almost surely}\Bigg\},
\]
 i.e. $h$ is $P_H$-degenerate of order $d -1$, 
 then 
we have
the  \emph{Hoeffding decomposition}
\begin{equation} \label{H_decomp}
U_n = \sum_{r=d}^m {m \choose r} U_n^{(r)},
\end{equation}
where 
\begin{equation} \label{H_proj_u_stat}
U_n^{(r)} \equiv {n \choose r}^{-1} \sum_{(i_1, \dots, i_r) \in I_{n, r}}  \pi_r (h)({\bf X}_{i_1}, \dots, {\bf X}_{i_r})
\end{equation}
 is each a canonical U-statistic of degree $r$ known as a \emph{Hoeffding projection}; 
we refer  to  \citet[p.138]{de2012decoupling} for  these standard results.  The following  inequalities, which stem from a Bernstein inequality proven  for the Hoeffding projections in  \secref{Bernstein_pf}, are crucial for establishing \thmref{main}.

\begin{lemma}[``Consequences" of Bernstein inequalities for the $U_n^{(r)}$'s] \label{lem:Bernstein_ineq_h_proj}
Let $U_n^{(r)}$ be constructed as in \eqref{H_proj_u_stat} with the kernel $h$ in  \eqref{polynomial_kernel} and  data from \eqref{iid_data_sample}. Assume \eqref{non_singular_assumption_Theta} and $2m < n$. There exist absolute constants $C_1(m), C_2(m), C_3(m), c_1(m),  c_2(m) >0$ such that, for all $r = d, \dots, m$,
\begin{enumerate}
 
 \item $P_{H}\Big(\Big|\sqrt{n} U_n^{(r)}\Big| >     C_1(m)  \sigma_h n^{(1-r)/2}  (\log (2n^r +1 ))^{3r}\Big)  \leq \frac{C_2(m)}{n} $  and
  

\item $ P_{H}\Big( |\sqrt{n} U_n^{(r)}|> z\Big) \leq      C_3(m) \exp\Big(-  \frac{c_1(m) z^2 n^{r-1}}{ \sigma_h^2 }   \Big) $ for $0 \leq  z \leq c_2(m)   \sigma_h   n^{\frac{1-r}{2}}$

 \end{enumerate}

\end{lemma}
 \lemref{Bernstein_ineq_h_proj}$(ii)$ is particularly remarkable, as it captures the sub-Gaussian behavior of the tail of $U_n^{(r)}$ for smaller values of $z$, just like the classical Bernstein's inequality for a sum of independent sub-exponential variables \citep[Proposition 3.1.8]{GN2016}. 
Like existing works  on  B-E bounds for incomplete U-statistics \citep{chen2019randomized, SongChenKato2019, peng2022rates, sturma2022testing}, we will rewrite $U_{n, N}'$ as
\begin{equation} \label{U_prime_as_U_star}
U_{n, N}'  = \frac{N}{\hat{N}} W_n,
\end{equation}
where 
\begin{equation} \label{Wn_decomp}
W_n \equiv U_n + \sqrt{1 - p_{n,N}}  B_n,
\end{equation}
for 
\begin{equation}\label{Bn_def}
B_n \equiv  \frac{1}{N} \sum_{(i_1, \dots, i_m) \in I_{n, m}} \frac{Z_{(i_1, \dots, i_m)} - p_{n,N}}{\sqrt{1 - p_{n,N}}} h({\bf X}_{i_1}, \dots, {\bf X}_{i_m});
\end{equation}
essentially, up to the factor $\sqrt{1 - p_{n,N}}$, $W_n$ is a sum of the complete U-statistic $U_n$ and $B_n$, and it is a simple matter to check that \eqref{U_prime_as_U_star}-\eqref{Bn_def} hold.  With these preparations, we are now primed to prove \thmref{main}. 
\begin{proof}[Proof of \thmref{main}]
It is a direct consequence of \eqref{U_prime_as_U_star} that
\begin{equation} \label{CK_decomposition_rescale}
\frac{\sqrt{n}U_{n, N}' }{\sigma} = \Bigg(\frac{N}{\hat{N}} - 1\Bigg)\frac{\sqrt{n} W_n}{\sigma} +\frac{\sqrt{n} W_n}{\sigma},
\end{equation}
which allows us to simplify the problem slightly: 
By Bernstein's inequality for bounded random variables  \citep[Lemma 2.2.9, p.102]{vaart1996weak}, 
\begin{align} 
P \Bigg( \Bigg|\frac{\hat{N}}{N}-  1\Bigg| > \frac{\sqrt{2\log n}}{\sqrt{N}}\Bigg) 
&\leq  2 \exp \Bigg(- \frac{\log (n)}{1- p_{n, N} +\frac{\sqrt{2}}{3} \max(p_{n, V}, 1 - p_{n, V}) \sqrt{\frac{\log(n)}{N}} }   \Bigg) \notag\\
 &\leq 2 n^{-1/2},  \label{Nhat_over_N_tail_bdd}
\end{align}
where the last inequality uses the assumption $\log(n) \leq N$. 
On the other hand, since $\bar{\Phi}(z) \leq  \min \big(\frac{1}{2}, \frac{1}{z \sqrt{2 \pi}} \big) e^{-z^2/2}$ for any $z > 0$ \citep[p.16]{chen2010normal},
\begin{equation} \label{scaled_Wn_tail_bdd}
P_H\Bigg( \Bigg|\frac{\sqrt{n} W_n}{\sigma}\Bigg| > \sqrt{2 \log n} \Bigg) \leq    \sup_{z \in \bR}\Bigg| P_H\Bigg( \Bigg|\frac{\sqrt{n} W_n}{\sigma}\Bigg| > z \Bigg)  - \bar{\Phi}(z)\Bigg| + \frac{1}{2 n}.
\end{equation}
Since  
$
|\Phi(z) - \Phi(z + t)| \leq \phi(0) |t| \leq \frac{|t|}{\sqrt{2 \pi}}
$ for any $z, t \in \bR$, combining the two tail inequalities \eqref{Nhat_over_N_tail_bdd}  and \eqref{scaled_Wn_tail_bdd}, from \eqref{CK_decomposition_rescale}, 
 we further have 
\begin{equation} \label{general_BE_bdd_for_rescaled_icu_stat}
\Bigg|P_{H}\Bigg(\frac{\sqrt{n}U_{n, N}' }{\sigma} \leq z \Bigg)- \Phi(z) \Bigg| \leq C \Bigg\{\frac{1}{\sqrt{n}} + \frac{\log n}{\sqrt{N}} +   \sup_{z \in \bR} \Bigg| P_{H}\Bigg(\frac{\sqrt{n} W_n}{\sigma} \leq z\Bigg) - \Phi(z) \Bigg| \Bigg\}.
\end{equation}
So for proving  \thmref{main}, it suffices to prove a B-E bound of the same form for 
\begin{equation*} \label{LHS_BE_bdd_Wn}
\sup_{z \in \bR}\Bigg|P_{H}\Bigg(\frac{\sqrt{n} W_n}{\sigma} \leq z\Bigg) - \Phi(z) \Bigg|
\end{equation*}
 on the right hand side of \eqref{general_BE_bdd_for_rescaled_icu_stat}, which will utilize the decomposition of $W_n$ into a sum between $U_n$ and $B_n$ in \eqref{Wn_decomp}.  Note that, conditional on the observed data $\{{\bf X}_i\}_{ i \in [n]}$, $B_n$ is a sum of independent terms driven by the randomness of the Bernoulli flips $Z_{\bf i}$, so it is expected to follow some form of central limit theorem. In particular, the following B-E bound for $B_n$  proven  in \appref{BE_bdd_Bn_pf_sec} will be useful.

\begin{lemma}[B-E bound for $B_n$] \label{lem:BE_bdd_Bn} Under the same assumptions as \thmref{main}, we have  a B-E bound of the form
\begin{equation*} \label{BE_Bn_eta_one_third}
\sup_{z' \in \bR}\Bigg| P_H\Bigg(  \frac{ \sqrt{N} B_n }{\sigma_h } \leq  z'\Bigg) -  \Phi(z') \Bigg| \leq  \frac{C}{\sqrt{N}} \frac{\sqrt{1 - 3p_{n, N} + 3p_{n, N}^2}}{1 - p_{n, N}} 
  +  \frac{C(m)}{\sqrt{n}}.
\end{equation*}
\end{lemma}

From this point on, we will separate the proof  into the regular and singular cases corresponding to $\sigma_g \neq 0$ and $\sigma_g = 0$. 
 We will focus on proving \thmref{main}$(i)$ here, where another technical  innovation lies due to  the need to handle the delicate near-singular situations with a \emph{nonuniform} Berry-Esseen bound for $U_n$. The proof of \thmref{main}$(ii)$  is deferred to \appref{pf_degenerate_case}.

\subsection{Proof of \thmref{main}$(i)$} \label{sec:pf_main_part_one}

To bound $\sup_{z \in \bR} |\PH(\frac{\sqrt{n} W_n}{\sigma} \leq z) - \Phi(z) |$ in \eqref{general_BE_bdd_for_rescaled_icu_stat} when $\sigma_g >0$, it suffices to consider $z \geq 0$ only, or else one can replace the kernel $h(\cdot)$ with $-h(\cdot)$. Let $Z_1$ and $Z_2$ be two independent $\cN(0,   1)$ random variables that are also independent of the data $\{{\bf X}_i\}_{i \in [n]}$, and use $p = p_{n, N}$ and $\alpha= \alpha_n$ to simplify notation. From \eqref{Wn_decomp}, we write the chain of inequalities
\begin{align*}
&\PH\Bigg(\frac{\sqrt{n} W_n}{\sigma} \leq z \Bigg) \\
& =  \bEH\Bigg[ P\left( \frac{\sqrt{N} B_n }{\sigma_h}\leq 
\frac{z \sigma}{\sigma_h\sqrt{ \alpha (1- p)}} 
-  \sqrt{\frac{N}{1 - p}} \frac{U_n}{\sigma_h} \Bigg| \{{\bf X}_i\}_{i \in [n]} \right) \Bigg]\\
&\leq P_H\left( Z_1 \leq \frac{z \sigma}{\sigma_h\sqrt{\alpha (1- p)}}- \sqrt{\frac{N}{1 - p}} \frac{U_n}{\sigma_h}\right) + \sup_{z' \in \bR} \Big| P(  \sqrt{N} B_n \sigma_h^{-1} \leq  z' ) -  \Phi(z')\Big|\\
&= \bE \Bigg[ P_H\Bigg( \frac{\sqrt{n} U_n }{m \sigma_g} \leq \frac{z \sigma}{m \sigma_g} - \frac{ \sigma_h\sqrt{ \alpha (1- p) } }{m \sigma_g} Z_1 \Bigg| Z_1 \Bigg) \Bigg]+ \sup_{z' \in \bR}| P(  \sqrt{N} B_n \sigma_h^{-1} \leq  z' ) -  \Phi(z')|\\
&\leq  P\Big(Z_2 \leq  \frac{z \sigma}{m \sigma_g} - \frac{ \sigma_h\sqrt{ \alpha (1- p) } }{m \sigma_g} Z_1\Big) + \sup_{z' \in \bR}| P(  \sqrt{N} B_n \sigma_h^{-1} \leq  z' ) -  \Phi(z')| + \varepsilon_1\\
&= \Phi \Bigg(\frac{\sigma z}{\sqrt{m^2  \sigma_g^2+  \alpha (1 - p) \sigma_h^2}}  \Bigg) + \sup_{z' \in \bR}| P(  \sqrt{N} B_n \sigma_h^{-1} \leq  z' ) -  \Phi(z')| + \varepsilon_1 \\
&\leq \Phi(z) +\sup_{z' \in \bR}| P(  \sqrt{N} B_n \sigma_h^{-1} \leq  z' ) -  \Phi(z')| + \varepsilon_1 + \varepsilon_2,
\end{align*}
where
\begin{equation} \label{varep1_def}
\varepsilon_1 \equiv \bE\Bigg[  \Big| P_H\Big(\frac{\sqrt{n} U_n }{m\sigma_g} \leq  \frac{z \sigma}{m \sigma_g} - \frac{ \sigma_h\sqrt{ \alpha (1- p) } }{m \sigma_g} Z_1 \Big| Z_1 \Big) - \Phi\Big( \frac{z \sigma}{m \sigma_g} - \frac{ \sigma_h\sqrt{ \alpha (1- p) } }{m \sigma_g} Z_1\Big)  \Big| \Bigg],
\end{equation}
and
\[
\varepsilon_2 \equiv \Bigg| \Phi\Bigg(\frac{\sigma z}{\sqrt{m^2  \sigma_g^2+  \alpha (1 - p) \sigma_h^2}}  \Bigg) - \Phi(z)\Bigg|;
\]
the last equality above uses that $m \sigma_g Z_2 + \sigma_h \sqrt{\alpha (1- p)} Z_1$ is distributed as a $\cN(0, m^2  \sigma_g^2+  \alpha (1 - p) \sigma_h^2)$ random variable. By an analogous and reverse argument, 
\[
P\Bigg(\frac{\sqrt{n} W_n}{\sigma} \leq z \Bigg)  \geq \Phi(z)  - \sup_{z' \in \bR}| P(  \sqrt{N} B_n \sigma_h^{-1} \leq  z' ) -  \Phi(z')| - \varepsilon_1 - \varepsilon_2,
\]
which together with the previous inequality implies 
\begin{equation} \label{BE_bdd_for_Wn_non_degen}
\Bigg|P\Bigg(\frac{\sqrt{n} W_n}{\sigma} \leq z \Bigg)  - \Phi(z) \Bigg| \leq \sup_{z' \in \bR}| P(  \sqrt{N} B_n \sigma_h^{-1} \leq  z' ) -  \Phi(z')|+ \varepsilon_1 + \varepsilon_2.
\end{equation}
Combining \eqref{general_BE_bdd_for_rescaled_icu_stat}, \eqref{BE_bdd_for_Wn_non_degen}, and \lemref{BE_bdd_Bn}, to finish proving  \thmref{main}$(i)$, it suffices to let $\mathfrak{R} = \varepsilon_1 + \varepsilon_2$ and form a bound for it by  bounding $\varepsilon_1$ and $\varepsilon_2$ separately:

\subsubsection{Bound on $\varepsilon_1$} 
A critical ingredient to develop our bound for $\varepsilon_1$ is the following \emph{nonuniform} B-E bound for $U_n$ with  terms that decay exponentially in the magnitude of the argument in the distribution function.

\begin{lemma} [A nonuniform B-E bound for $U_n$] \label{lem:nonunif_BE}
Let $U_n$ be constructed as in \eqref{complete_U_stat} with the kernel $h$ in  \eqref{polynomial_kernel} and  data from \eqref{iid_data_sample};  assume  $2m<n$. For any $x \in \bR$,  there exist constants $C_1, C_2 >0$ such that
\begin{multline*}
\Bigg|\PH\Bigg( \frac{\sqrt{n}}{m \sigma_g} U_n  \leq x \Bigg) -\Phi(x)\Bigg| \leq C_1 \frac{\bE_{H}[|g|^3]}{\sqrt{n}\sigma_g^3} + \\
\PH\left(\Bigg|\frac{\sqrt{n}}{ m \sigma_g} \sum_{r=2}^m {m \choose r} U_n^{(r)}\Bigg| >  \frac{|x| +1}{3} \right) +   C_2 e^{-|x|/6} \Bigg( 
\frac{ (\sqrt{2}+2)(m-1) \sigma_h}{ \sqrt{m (n-m+1)}\sigma_g} \Bigg).
\end{multline*}
\end{lemma}

The proof is  in \appref{nonunif_BE_pf_sec}. Note that the first term $C_1\bEH[|g|^3]/(\sqrt{n}\sigma_g^3)$ in the bound has no dependence on $x$ so \lemref{nonunif_BE} is, strictly speaking, ``partially nonuniform". With more effort, we could have  made this term dependent on $x$ as well, which is not necessarily for our purpose; the  crucial part  is that the last term containing the ratio $\sigma_h/\sigma_g$ from \eqref{Achilles_heel} decays in $|x|$  \emph{exponentially}, which provides a pathway for eliminating this ratio in our final B-E bound for $U'_{n, N}$. 

Now, we define
$
\mathfrak{z}_{z,Z_1} \equiv \frac{z \sigma}{m \sigma_g} - \frac{ \sigma_h\sqrt{ \alpha (1- p) } }{m \sigma_g} Z_1
$
 , with the understanding that $\mathfrak{z}_{z,Z_1}$ depends on both $z$ and $Z_1$. Moreover, it is easy to see that
 \begin{equation} \label{frakz_normality}
 \mathfrak{z}_{z,Z_1} \sim \cN \Big(\frac{z \sigma}{m \sigma_g},  \tau^2\Big)
 \end{equation}
 with randomness driven by $Z_1$, where  $\tau \equiv \frac{\sqrt{ \alpha (1 - p)} \sigma_h}{m \sigma_g}$. By 
treating $Z_1$  as fixed and apply \lemref{nonunif_BE} to the absolute difference inside the expectation operator in the definition of $\varepsilon_1$ in  \eqref{varep1_def}, we get
\begin{multline} \label{after_apply_nonunifBE}
\varepsilon_1 \leq \frac{C_1\bE_{H}[|g|^3]}{\sqrt{n}\sigma_g^3} + \frac{C_2 (\sqrt{2}+2)(m-1) \sigma_h}{ \sqrt{m (n-m+1)}\sigma_g} \bE_{Z_1}\Big[  e^{-|\frakz_{z, Z_1}|/6}  \Big] + \\
\bE_{Z_1}\Big[ P_H\Big(\Big|\frac{\sqrt{n}}{ m \sigma_g} \sum_{r=2}^m {m \choose r} U_n^{(r)}\Big| >  \frac{|\frakz_{z, Z_1} |+1}{3} \Big| Z_1 \Big) \Big] 
,
\end{multline}
where the subscript in $\bE_{Z_1}[\cdot]$ emphasizes that the expectation is driven by the random variable $Z_1$.   We now have to bound the  two expectations of the type $\bE_{Z_1}[\cdot]$  on the right hand side of \eqref{after_apply_nonunifBE}. For the first one, we note from \eqref{frakz_normality} that
\[
\bE_{Z_1}\Big[  e^{-|\frakz_{z, Z_1}|/6}  \Big] = \int_{\bR} e^{-|x|/6} \phi\Bigg( \frac{x - \frac{z \sigma}{m \sigma_g}}{\tau}\Bigg) dx;
\]
since $e^{-|x|/6}$ is symmetrically decreasing as $|x| \longrightarrow \infty$, intuitively, $\bE_{Z_1}\Big[  e^{-|\frakz_{z, Z_1}|/6}  \Big]$ should attain the largest value when the highest mass from the $\cN (\frac{z \sigma}{m \sigma_g},  \tau^2)$ density can be placed at $x = 0$, which happens when $z = 0$. Indeed, that $\bE_{Z_1}[  e^{-|\frakz_{z, Z_1}|/6} ] \leq \bE_{Z_1}[  e^{-|\frakz_{0, Z_1}|/6} ] = \frac{2}{ \sqrt{ 2 \pi } \tau}\int_{0}^\infty  e^{-x/6} \exp\Big( - \frac{x^2}{2 \tau^2}\Big)dx$ is implied by the \emph{rearrangement inequality} \citep[Theorem 3.4, p.82]{LL2001}, so we can proceed using ``completing the square" as
\begin{align}
\bE_{Z_1}\Bigg[  e^{-|\frakz_{z, Z_1}|/6} \Bigg] 
&\leq 
 \frac{2}{ \sqrt{ 2 \pi } \tau}\int_{0}^\infty  e^{-x/6} \exp\Big( - \frac{x^2}{2 \tau^2}\Big)dx \notag\\
&=\exp\left( \frac{\tau^2 }{72}\right) \frac{2}{\sqrt{2 \pi} \tau} \int_0^\infty \exp\left( -   \frac{(x  - (-  6^{-1}\tau^2))^2}{2 \cdot \tau^2} \right) dx \notag\\
&= 2 \exp\left( \frac{\tau^2 }{72}\right) \barPhi(   \tau/6) \notag\\
&\leq 2 \exp\left( \frac{\tau^2 }{72}\right) \min\left(\frac{1}{2} , \frac{6}{ \tau \sqrt{2 \pi}}  \right) \exp\left( \frac{- \tau^2}{72}\right) \notag\\
&= 2 \min\left(\frac{1}{2} , \frac{6}{ \tau \sqrt{2 \pi}}  \right) 
\label{upper_bdd_nonunif_exp_term},
\end{align}
where the last inequality uses the well-known $\barPhi(x) \leq \min(\frac{1}{2}, \frac{1}{x \sqrt{2 \pi}})e^{-x^2/2}$ for $x >0$ \citep[p.16]{chen2010normal}.
 For the second  $\bE_{Z_1}[\cdot]$ expectation in  \eqref{after_apply_nonunifBE},  we first apply a union bound to the tail probability   and  obtain 
\begin{multline} \label{union_bdd}
\PH\Big(\Big|\frac{\sqrt{n}}{ m \sigma_g} \sum_{r=2}^m {m \choose r} U_n^{(r)}\Big| >  \frac{|\frakz_{z, Z_1} |+1}{3}  \Big| Z_1 \Big)  
\leq \\
 \sum_{r=2}^m \PH\Big(|\sqrt{n} U_n^{(r)}| >  \frac{m \sigma_g(|\frakz_{z, Z_1}| +1)}{3{m \choose r}  (m-1)}  \Big| Z_1 \Big);
\end{multline}
each  summand on the right hand side can be bounded using \lemref{Bernstein_ineq_h_proj} as
\begin{multline} \label{each_summand_sub_Gauss_tail_prob}
\PH\Big(|\sqrt{n} U_n^{(r)}| >  \frac{m \sigma_g(|\frakz_{z, Z_1}| +1)}{3{m \choose r}  (m-1)} \Big| Z_1 \Big) \\
 \leq
  \begin{cases} 
 C(m) \exp \Big( - \frac{c(m) \sigma_g^2\frakz_{z, Z_1}^2 n^{r-1}}{\sigma_h^2}\Big)& \text{if } |\frakz_{z, Z_1}| +1 \leq  \frakc_m  n^{\frac{1-r}{2}} \frac{\sigma_h}{\sigma_g};\\
   \frac{C(m)}{n}   & \text{if }  |\frakz_{z, Z_1}| +1 >  \frak{C}_m   n^\frac{1- r}{2}(\log (2n^r +1))^{3r} \frac{\sigma_h}{\sigma_g},
  \end{cases}
\end{multline}
for some special constants $\frakc_m, \frak{C}_m> 0$ depending only on $m$, whose values will \emph{not} change in their ensuing reoccurrences. Combining \eqref{union_bdd} and \eqref{each_summand_sub_Gauss_tail_prob} and take the outer expectation $\bE_{Z_1}[\cdot]$, we get
\begin{align}
&\bE_{Z_1}\Big[P_H\Big(\Big|\frac{\sqrt{n}}{ m \sigma_g} \sum_{r=2}^m {m \choose r} U_n^{(r)}\Big| >  \frac{|\frakz_{z, Z_1} |+1}{3}  \Big| Z_1 \Big)  \Big]  \notag \\
&\leq \frac{C_1(m)}{n} + C_2(m) \sum_{r=2}^m \Bigg\{\bE_{Z_1}\Big[ \exp \Big( - \frac{c(m) \sigma_g^2\frakz_{0, Z_1}^2 n^{r-1}}{\sigma_h^2}\Big) \Big] \notag \\
& \quad + \bE_{Z_1} \Big[   I\Big(  \frakc_m  n^{\frac{1-r}{2}} \frac{\sigma_h}{\sigma_g} -1 < |\frakz_{z, Z_1} | \leq  \frak{C}_m    n^\frac{1-r}{2}(\log (2n^r +1))^{3r}\frac{\sigma_h}{\sigma_g} -1\Big)\Big] 
\Bigg\}. \label{overarch_expect_Z1_bdd}
\end{align}
Now we bound the two $\bE_{Z_1}[\cdot]$ expectations in \eqref{overarch_expect_Z1_bdd}. First, from \eqref{frakz_normality} we have
\begin{align}
&\bE_{Z_1}\Big[ \exp \Big( - \frac{c(m) \sigma_g^2\frakz_{0, Z_1}^2 n^{r-1}}{\sigma_h^2}\Big) \Big] \notag\\
&= \frac{m \sigma_g}{ \sqrt{ 2 \pi \alpha (1 - p) \sigma_h}  }\int_{\bR} 
 \exp \Bigg\{ - \Bigg[ \frac{ \sigma_g^2  }{\sigma_h^2}  \Bigg( c(m)  n^{r-1} + \frac{m^2  }{2 \alpha (1-p)}\Bigg)\Bigg] x^2 \Bigg \}dx \notag\\
&= \frac{m }{\sqrt{2 \alpha (1 - p) c(m) n^{r-1} + m^2}} \notag \\
&\leq 
\frac{C(m)\sqrt{N}}{\sqrt{  \big(1- p\big)  n^r}}. \label{first_expect_Z1_bdd}
\end{align}
Second, from the normal distribution of $\frakz_{z, Z_1}$ in \eqref{frakz_normality} again,
\begin{align}
& \bE_{Z_1} \Big[   I\Big(  \frakc_m  n^{\frac{1-r}{2}} \frac{\sigma_h}{\sigma_g} -1 < |\frakz_{z, Z_1} | \leq  \frak{C}_m    n^\frac{1-r}{2}(\log (2n^r +1))^{3r}\frac{\sigma_h}{\sigma_g} -1\Big)\Big] \notag  \\
&\leq   \frac{n^{\frac{1-r}{2}}\sigma_h}{\sqrt{2 \pi } \tau \sigma_g} (\frak{C}_m (\log (2n^r +1))^{3r} - \frakc_m) \notag\\
&\leq   \frac{C(m) \sqrt{N}[(\log (2n^r +1))^{3r} - 1]}{\sqrt{ (1 - p)n^{r} }} \label{second_expect_Z1_bdd}
\end{align}
Now, gathering  \eqref{overarch_expect_Z1_bdd}, \eqref{first_expect_Z1_bdd} and \eqref{second_expect_Z1_bdd}, we obtain
\begin{multline} \label{upper_bdd_expect_Z1_tail_prob}
\bE_{Z_1}\Bigg[P_H\Bigg(\Big|\frac{\sqrt{n}}{ m \sigma_g} \sum_{r=2}^m {m \choose r} U_n^{(r)}\Big| >  \frac{|\frakz_{z, Z_1} |+1}{3}  \Bigg| Z_1 \Bigg)  \Bigg] \leq \\
 C(m) \Bigg\{ \frac{1}{n} +  \sum_{r=2}^m\frac{ \sqrt{N}(\log (2n^r +1))^{3r} }{\sqrt{ (1 - p)n^{r} }} \Bigg\}.
\end{multline}
Finally, applying both the bounds \eqref{upper_bdd_nonunif_exp_term} and \eqref{upper_bdd_expect_Z1_tail_prob} to \eqref{after_apply_nonunifBE}, we get
\begin{align} 
\varepsilon_1 &\leq \frac{\bE[|g|^3]}{\sqrt{n}\sigma_g^3} +  
 C(m) \Bigg\{   \sum_{r=2}^m\frac{ \sqrt{N}(\log (2n^r +1))^{3r} }{\sqrt{ (1 - p)n^{r} }} \Bigg\} \notag\\
 &\leq   \frac{C}{\sqrt{n}} +  
 C(m) \Bigg\{   \sum_{r=2}^m\frac{ \sqrt{N}(\log (2n^r +1))^{3r} }{\sqrt{ (1 - p)n^{r} }} \Bigg\}  \label{varep1_ultimate_bdd},
\end{align}
where the last inequality uses the hypercontractivity of $g({\bf X})$ from \eqref{hypercontractivity_of_g}.

 \subsubsection{Bound on $\varepsilon_2$} To simplify notation we let $\tilde{\sigma} = \sqrt{m^2  \sigma_g^2+  \alpha (1 - p) \sigma_h^2}$. By Taylor's expansion around $z$, that $\sigma \geq \tilde{\sigma}$, and the assumption that $z \geq 0$, we get 
\begin{align} 
\varepsilon_2 &\leq     \phi(z) z \Bigg(\frac{\sigma  }{\tilde{\sigma}}  - 1 \Bigg) \notag\\
&\leq  C  \Bigg(\frac{\sigma  - \tilde{\sigma} }{\tilde{\sigma}}   \Bigg) \text{ by $\sup_{z \geq 0} z \phi(z) \leq C$}\notag \\
&\leq  C  \Bigg(\frac{n {n \choose m}^{-1} \sigma_h^2 }{\tilde{\sigma}^2} \Bigg) \text{ by multiplying  $\frac{\sigma + \tilde{\sigma}}{\tilde{\sigma}}$ to the previous expression}\notag \\
&\leq C \Bigg(\frac{n {n \choose m}^{-1} \sigma_h^2 }{ \alpha (1 - p) \sigma_h^2 } \Bigg) =  \frac{Cp}{1 - p}
\label{varep2_ultimate_bdd}
\end{align}

Finally, collecting  \eqref{varep1_ultimate_bdd} and \eqref{varep2_ultimate_bdd}, we get 
\begin{equation*} \label{varepsilon_bdd}
\mathfrak{R} = \varepsilon_1  + \varepsilon_2  \leq  
 C(m) \Bigg\{   \sum_{r=2}^m\frac{ \sqrt{N}(\log (2n^r +1))^{3r} }{\sqrt{ (1 - p_{n, N})n^{r} }} \Bigg\} + C \Bigg(  \frac{p_{n, N}}{1- p_{n, N}} + \frac{1}{\sqrt{n}}\Bigg),
\end{equation*}
as claimed in \thmref{main}$(i)$.

\subsection{Proof for the overall bound in 
\eqref{sing_agn_BE_bdd} } \label{sec:overall_main_bdd}
It is a simple consequence of combining parts $(i)$ and $(ii)$ because
\begin{align*}
p_{n, N} = N {n \choose m}^{-1} &\leq  \frac{ m!}{(n-1) \cdots (n - m+1)}\cdot m \quad \text{ (by $N \leq mn$)}\\
&\leq  \frac{m(m-1)\cdots 2}{(2m+1) \dots (m+3)} \cdot m \quad \text{ (by $2m+1 < n$)} \\
&< 1.
\end{align*}

\end{proof}

\section{Bernstein's inequalities for the Hoeffding projections} \label{sec:Bernstein_pf}
In this section we will establish a Bernstein inequality for  the Hoeffding projections $U_n^{(r)}$ from \eqref{H_proj_u_stat} that implies \lemref{Bernstein_ineq_h_proj} as a corollary. Before that, however, we shall present a slightly more general result  that concerns the moments of generalized decoupled U-statistics that are canonical. 
Throughout this section, calligraphic constants such as $\cC_r$, $\cK_r \dots$ denote special constants whose value \emph{do not} vary from places to places (as opposed to the generic constants in normal font such as $C(r)$ defined in \secref{outline_notation}), where their subscripts   specify exclusive dependence on another quantity such as $r$.
For a  measurable space $(M, \cM)$ and $r \in \bN$, 
let $\{Y_i^{(l)}\}_{1 \leq i \leq n, 1 \leq l \leq r}$ be independent  $M$-valued random variables, and $\kappa: M^r \longrightarrow \bR$ be a real-valued measurable function not necessarily symmetric in its $r$ arguments. 
 Moreover,   
let 
\begin{equation} \label{kappa_kernels_def}
\kappa_{\bf i} = \kappa(Y_{i_1}^{(1)}, \dots, Y_{i_r}^{(r)}) \text{ for each } {\bf i} \equiv (i_1, \dots, i_r) \in [n]^r,
\end{equation}
 which is assumed to have the \emph{canonical property} that
\begin{equation} \label{kappa_canonical_prop}
\bE_l[\kappa_{\bf i}] = 0 \text{ for each } l \in [r],
\end{equation}
where  $\bE_l[\cdot]$ denotes integration with respect to the variables $\{Y_i^{(l)}\}_{1 \leq i \leq n}$. For a non-empty $I \subset [r]$, we similarly let $\bE_I[\cdot]$ denote integration with respect to the variables $\{Y_i^{(l)}\}_{1 \leq i \leq n, l \in I}$; moreover, we let $|I|$ be the cardinality of $I$, $I^c = [r] \backslash I$ be the complement of $I$ in $[r]$, ${\bf i}_I = (i_l)_{l \in I}$ for any ${\bf i} \in [n]^r$,  and  $\cP_I$ be the collection of all partitions $\cJ = \{J_1, \dots, J_k\}$ of $I$ into non-empty  disjoint subsets, where we also define $|\cJ|  = k$ for a $\cJ$ as above. By convention, we also let $\cP_{\emptyset} = \{\emptyset\}$ with $|\emptyset| = 0$.

\begin{definition}[$\|\cdot\|_{\cJ}$-norm for a partition $\cJ \in \cP_I$ ] \label{def:opnorm}
For $r \in \bN$ and a non-empty set $I \subset [r]$, consider $\cJ = \{J_1, \dots, J_k\} \in \cP_I$. For the array  $(\kappa_{\bf i})_{{\bf i} \in [n]^r}$ defined in \eqref{kappa_kernels_def} with the canonical property \eqref{kappa_canonical_prop} and any fixed value of ${\bf i}_{I^c}$, define
\begin{multline*}
\| (\kappa_{\bf i})_{{\bf i}_I }\|_{\cJ} = \sup \Bigg\{
\Bigg| \bE_I\Bigg[  \sum_{{\bf i}_I}
\kappa  (Y_{i_1}^{(1)}, \dots, Y_{i_r}^{(r)})
\prod_{j=1}^{|\cJ|}  f_{{\bf i}_{J_j}}^{(j)}  \Big( (Y^{(l)}_{i_l})_{l \in J_j}\Big)
\Bigg]
\Bigg|
: \\
\bE 
\Bigg[
\sum_{{\bf i}_{J_j}} \Big|  f_{{\bf i}_{J_j}}^{(j)}  \Big( (Y^{(l)}_{i_l})_{l \in J_j}\Big)
 \Big|^2
\Bigg] \leq 1 \text{ for } j = 1, \dots, |\cJ|
\Bigg\},
\end{multline*}
where $\sum_{{\bf i}_{J_j}}(\cdot)$ means summing over all ${\bf i}_{J_j} \in [n]^{|J_j|}$, i.e. it is a sum of $n^{|J_j|}$ many terms; $\sum_{{\bf i}_I}(\cdot)$ is a summation similarly defined.
Moreover, by convention, we let $\|(\kappa_{\bf i})_{{\bf i}_\emptyset}\|_\emptyset = |\kappa_{\bf i}|$.
\end{definition}
\defref{opnorm} comes from \citet[Definition 3, p.9]{adamczak2006moment}, and note that,  $\| (\kappa_{\bf i})_{{\bf i}_I }\|_{\cJ}$ is a function in the variables 
$\{Y_i^{(l)}\}_{i \in [n], l \in I^c}$, so it is random unless $I = [r]$. 
A useful result regarding the \emph{canonical generalized decoupled U-statistic}
\[
\sum_{{\bf i} \in [n]^r} \kappa_{\bf i}
\] is the following sharp moment inequality, which is proved in  \appref{pf_moment_ineq_decouple_u_stat} by further developing the moment inequalities for  such statistics in the fundamental works of  \citet{gine2000exponential, adamczak2006moment} with an inductive argument.

\begin{theorem}[Sharp moment inequality for canonical generalized decoupled U-statistics] \label{thm:sharp_ustat_moment}
For the array  $(\kappa_{\bf i})_{{\bf i} \in [n]^r}$ defined in \eqref{kappa_kernels_def} with the canonical property \eqref{kappa_canonical_prop}, there exists an absolute constant $\cK_r$ depending only on $r$ such that, for any $p \geq 2$,
\begin{equation*}
\bE\Bigg[\Big|\sum_{{\bf i} \in [n]^r} \kappa_{\bf i}\Big|^p \Bigg] 
\leq \cK_r^p \sum_{I \subset [r]} \sum_{\cJ \in \cP_I} p^{p(|I^c|+ |\cJ|/2)}
\bE_{I^c}\Bigg[ \max_{{\bf i}_{I^c}}  \Big(\bE_I \big[\sum_{{\bf i}_I } \kappa_{\bf i}^2\big]\Big)^{p/2}\Bigg].
\end{equation*}

\end{theorem}

We are now ready to prove a Bernstein's inequality for $U_n^{(r)}$ that implies \lemref{Bernstein_ineq_h_proj}. To facilitate the application of  \thmref{sharp_ustat_moment} to that end, we let $\{\bX_1^{(k)}, \dots, \bX_n^{(k)}\}_{k =1}^r$ be $r$ independent copies of our data in \eqref{iid_data_sample}, and for each ${\bf i} \equiv (i_1, \dots, i_r) \in [n]^r$,  we  define  the  random ``decoupled" quantity
\begin{equation} \label{frakh_def}
h_{\bf i}^{(r)}  \equiv 
  \begin{cases} 
  \frac{\pi_r(h)(\bX_{i_1}^{(1)}, \dots, \bX_{i_r}^{(r)})}{r!}  & \text{if } i_1 \neq \dots \neq i_r; \\
   0     & \text{if otherwise},
  \end{cases}
\end{equation}
with the projection kernel  $\pi_r(h)$.
Moreover, by letting $\bE_{l, H}[\cdot]$ denote the expectation with respect to $\{\bX^{(l)}_i\}_{1 \leq i \leq n}$ for $l \in [r]$ and under $H$, each $h_{\bf i}^{(r)}$ in \eqref{frakh_def} inherits the canonical property
\begin{equation} \label{canonical_prop_of_frakh}
\bE_{l, H} [h^{(r)}_{\bf i}] = 0
\end{equation}
from $\pi_r(h)$; recall \eqref{canonical_property_hoeffding_proj}. We now present our  Bernstein's inequality for $U_n^{(r)}$; it is stated in terms of the sum 
\begin{equation} \label{alt_form_for_nchooser_Unr}
{n \choose r}  U_n^{(r)} = \sum_{(i_1, \dots, i_r) \in I_{n, r}}  \pi_r (h)({\bf X}_{i_1}, \dots, {\bf X}_{i_r}),
\end{equation}
 which is very close in form to the ``decoupled" sum $\sum_{{\bf i} \in [n]^r} h^{(r)}_{\bf i}$, except the latter has the $r$ independent  copies of the original data in \eqref{iid_data_sample} as input.

\begin{theorem}[Bernstein inequality for $U_n^{(r)}$] \label{thm:Bernstein_ineq}
Suppose $2m < n$. For $r \in [m]$, consider $U_n^{(r)}$ from \eqref{H_proj_u_stat} with our kernel $h$ defined in \eqref{polynomial_kernel}. Then, for all $x \geq 0$, there are constants $C(r), \cC_r, \cK_r > 0$ such that 
\begin{multline}  \label{most_general_bernstein_ineq}
P_H\Bigg(\Big|  {n \choose r}  U_n^{(r)}\Big| > x\Bigg) \leq  \\
C(r) \exp\Bigg(- \min_{\substack{I \subset [r]\\ \cJ \in \cP_I}} \Bigg( \frac{x}{\cC_r \cdot \cK_r \cdot e \cdot (\log ( 2 n^{|I^c|} + 1) )^{|I^c|} 
n^{|I|/2} \sigma_h  }  \Bigg)^{\frac{1}{2|I^c|+ |\cJ|/2}} 
\Bigg);
\end{multline}
note that $\cK_r$ is  the constant from \thmref{sharp_ustat_moment}.


\end{theorem}
We will first prove  \lemref{Bernstein_ineq_h_proj} as  a corollary, then  \thmref{Bernstein_ineq}:

\begin{proof} [Proof of \lemref{Bernstein_ineq_h_proj} as a corollary of \thmref{Bernstein_ineq}]

It suffices to show that, for $r \in [m]$,  if $x' = \cC_r \cdot \cK_r \cdot e \cdot (\log (2 n^r +1))^r n^{r/2} \sigma_h \cdot  (\log (n+ 2))^{2r}$, then there exists constants $C_1(r), C_2(r), c_1(r), c_2(r) >0$ such that
  \begin{equation}  \label{non_sub_Gauss_tail_general}
P_H\Bigg(\Bigg|  {n \choose r}  U_n^{(r)}\Bigg| > x'\Bigg) \leq    \frac{C_1(r)}{n +2}
\end{equation}
and
 \begin{equation} \label{sub_Gauss_tail_general}
 P_H\Bigg(\Bigg|  {n \choose r}  U_n^{(r)}\Bigg|  > x\Bigg) \leq   
    C_2(r) \exp\Bigg(-   \frac{c_1(r)x^2}{ n^r \sigma_h^2 } \Bigg) 
 \text{ for  } 0 \leq x \leq c_2(r)   \sigma_h  \cdot n^{r/2}.
 \end{equation}
Note that \eqref{non_sub_Gauss_tail_general} and \eqref{sub_Gauss_tail_general} readily implies \lemref{Bernstein_ineq_h_proj}$(i)$ and $(ii)$ respectively, as the constants depending on $r$ can be changed to the ``big" and ``small" constants depending on $m$ in \lemref{Bernstein_ineq_h_proj}, given that $r \in [m]$.

Define
\begin{equation} \label{MIJ_def_2}
M_I \equiv \bigg(\log ( 2 n^{|I^c|} + 1) \bigg)^{|I^c|} 
n^{|I|/2}  \sigma_h \text{ for any } I \subset [r].
\end{equation}
In particular, since $\frac{x'}{\cC_r \cdot \cK_r \cdot e \cdot M_I} \geq 1$ for all $I$ (because $\log (n+2) > 1$ even for $n=1$),  it must be that
\begin{equation} \label{larger_statement_in_x_prime}
\Bigg( \frac{x'}{\cC_r \cdot \cK_r \cdot e \cdot M_I }  \Bigg)^{\frac{1}{2|I^c|+ |\cJ|/2}}   \geq \Bigg( \frac{x'}{\cC_r \cdot \cK_r \cdot e \cdot M_I }  \Bigg)^{\frac{1}{2r}}.
\end{equation}
Hence, \thmref{Bernstein_ineq} applied with the threshold $x'$ together with \eqref{larger_statement_in_x_prime} implies
\[
 P_H\Bigg(\Bigg|  {n \choose r}  U_n^{(r)}\Bigg| > x'\Bigg) \leq \exp (- \log (n+2)),
\]
which is \eqref{non_sub_Gauss_tail_general}.
Moreover, if  
\begin{equation} \label{sub_Gaussian_condition}
\bigg([r], \{[r]\}\bigg) = \argmin_{I \subset [r], \cJ \in \cP_I} \Bigg( \frac{x}{\cC_r \cdot \cK_r \cdot e \cdot M_I }  \Bigg)^{\frac{1}{2|I^c|+ |\cJ|/2}},
\end{equation}
the right hand side of the bound in  \thmref{Bernstein_ineq}  will take the sub-Gaussian form
  \[
    C(r) \exp\Bigg(-  \Bigg( \frac{x}{\cC_r \cdot \cK_r \cdot e \cdot n^{r/2} \sigma_h }  \Bigg)^2 \Bigg)
  \]
  since $M_{[r]} = n^{r/2}\sigma_h$, $|[r]^c| = 0$ and $|\{[r]\}| = 1$. Hence, we can establish  \eqref{sub_Gauss_tail_general} by  the following claim  proven in \appref{remaining_pf_Bernstein_ineq_h_proj} about a sufficient condition for \eqref{sub_Gaussian_condition} to hold. 
  
  \begin{claim} \label{clm:claim2}
  There exists an absolute constant $c(r) > 0$ such that, if
  \[
x \leq c(r)  \sigma_h  \cdot n^{r/2},
  \]
  then the condition in \eqref{sub_Gaussian_condition} is true.  
  \end{claim}
\end{proof}

\begin{proof} [Proof of \thmref{Bernstein_ineq}]For any $I \subset [r]$,  we will use $\bE_{I, H}[\cdot]$ to denote expectation taken with respect to  $\{\bX^{(l)}_i\}_{1 \leq i \leq n, l \in I}$  under $H$. 
By \thmref{sharp_ustat_moment} and the canonical property in \eqref{canonical_prop_of_frakh}, for  the decoupled quantities defined in \eqref{frakh_def} and $p \geq 2$,
\begin{align} \label{moment_bdd_transparent}
\bEH\bigg[\Big|\sum_{{\bf i} \in [n]^r} h^{(r)}_{\bf i}\Big|^p \bigg] 
&\leq \cK_r^p \sum_{I \subset [r]} \sum_{\cJ \in \cP_I} p^{p(|I^c|+ |\cJ|/2)} \bE_{I^c, H}\bigg[\bigg(\max_{{\bf i}_{I^c}} \bE_{I, H} \bigg[\sum_{{\bf i}_I } (h^{(r)}_{\bf i})^2\bigg]  \bigg)^{p/2} \bigg].
\end{align}
 First, we make  the following claim whose proof can be found in \appref{remain_pf_bernstein_ineq}:
\begin{claim}\label{clm:claim1} There exists an absolute constant $C(r) >0$ depending only on $r$ such that,  for any $I \subset [r]$ and $p \geq 2$, 
\begin{equation*} 
\bigg\|\max_{{\bf i}_{I^c}} \bE_{I, H} \bigg[\sum_{{\bf i}_I } (h^{(r)}_{\bf i})^2\bigg]  \bigg\|_{p/2, H} \leq 
C(r) p^{2 |I^c|} \bigg(\log ( 2 n^{|I^c|} + 1) \bigg)^{2|I^c|}  n^{|I|} \sigma_h^2.
\end{equation*}
\end{claim}
In consideration of the form in \eqref{alt_form_for_nchooser_Unr}, 
by the standard \emph{decoupling inequality} for U-statistics \citep[Theorem 3.1.1, p.99]{de2012decoupling}, \eqref{moment_bdd_transparent} and \clmref{claim1}, there exists a special constant $\cC_r > 0$ depending only on $r$ such that \begin{align}
&\bEH\bigg[\Big|{n \choose r}  U_n^{(r)}  \Big|^p \bigg]  \notag \\
&\leq (\cC_r \cK_r)^p \sum_{I \subset [r]} \sum_{\cJ \in \cP_I} p^{p(2|I^c|+ |\cJ|/2)}  M_I^p \text{ for any } p \geq 2, \label{k_moment_bdd_in_MIJ}
\end{align}
where we have defined
\begin{equation} \label{MIJ_def}
M_I \equiv \bigg(\log ( 2 n^{|I^c|} + 1) \bigg)^{|I^c|} 
n^{|I|/2}  \sigma_h \text{ for any } I \subset [r].
\end{equation}
Moreover, define
\[
(I_p, \cJ_p) = \argmax_{I \subset [r], \cJ \in \cP_I} p^{p(2|I^c|+ |\cJ|/2)}  M_I^p \text{ for }  p \geq 2.
\]
and 
\begin{equation} \label{px_def}
p_x = \min_{I \subset [r], \cJ \in \cP_I} \Bigg( \frac{x}{\cC_r \cdot \cK_r \cdot e \cdot M_I }  \Bigg)^{\frac{1}{2|I^c|+ |\cJ|/2}} \text{ for any } x \geq 0,
\end{equation}
where we remark that $p_x$ could be a number less than $2$.
By Markov's inequality and \eqref{k_moment_bdd_in_MIJ}, we get that
\begin{multline} \label{tail_ineq_for_px_geq_2}
\PH\Bigg(\Big|{n \choose r}  U_n^{(r)} \Big| > x\Bigg) \leq \frac{C(r)  \Big(\cC_r \cK_r p_x^{(2|I_{p_x}^c|+ |\cJ_{p_x}|/2)}  M_{I_{p_x}}\Big)^{p_x}}{x^{p_x}}
\leq C(r)  e^{- p_x} \\
 \text{ for any } x >0 \text{ such that } p_x \geq 2,
\end{multline}
where the first inequality comes from how $(\cI_{p_x}, \cJ_{p_x})$ is defined, and the second inequality uses the fact that 
\[
p_x \leq \bigg( \frac{x}{\cC_r \cdot \cK_r \cdot e \cdot M_{I_{p_x}} }  \bigg)^{\frac{1}{2|I_{p_x}^c|+ |\cJ_{p_x}|/2}}
\]
 by the definition in \eqref{px_def}. Since $\PH(|\sum_{{\bf i} \in I_{n, r}} \pi_r(h)({\bf X}_{\bf i})| > x)$ must be bounded by $e^2 e^{-p_x}$ if $p_x < 2$, one can further deduce from \eqref{tail_ineq_for_px_geq_2} that, for a constant $C(r) > 0$, 
\begin{equation*} \label{overall_tail_bdd}
\PH\Bigg(\Big|{n \choose r}  U_n^{(r)} \Big| > x\Bigg)  \leq  C(r)  e^{- p_x} 
 \text{ for all } x\geq 0,
\end{equation*}
which is exactly the bound stated by \thmref{Bernstein_ineq}.

\end{proof}

We shall make one remark before ending this section. In the proof of \thmref{Bernstein_ineq} above,  we have effectively exploited the polynomial structure of the projection kernel $\pi_r(h)$  to prove \clmref{claim1}  in \appref{remain_pf_bernstein_ineq}. 
Generally, developing such Bernstein inequalities for U-statistics with more general unbounded kernels is a problem of current interest; see \citet{chakrabortty2018tail} and \citet{bakhshizadeh2023exponential} for the related discussion.

\section{Discussion} \label{sec:discussion}

Under the Gaussian setting of the present paper, via an investigation into the information encoded by a key ratio of moments appearing in the standard Berry-Esseen bound for complete U-statistics as in \eqref{BEustat},  we have obtained a solid understanding of why the classical asymptotics of the Wald test cannot be useful near singularities. With the same understanding, our main result in  \secref{main_results} also pinpoints why our proposed incomplete U-statistic construction is a natural alternative for circumventing the issue by trading off just enough statistical power. 
We conclude by discussing several  open issues that are  expected to lead to much future research.

\subsection{Relaxation of Gaussian assumptions}

The standard Gaussian setting as in \eqref{iid_data_sample} serves as a good starting point for looking into the behavior of the Wald test in face of (near-)singularities, and the proofs of our theorems have suggested that the hypercontractive estimate (or Khintchine-type inequality) in \lemref{hypercontraction} plays a key role in explaining the singularity-agnostic behavior of $S_{\lfloor n/m \rfloor}$ and $U_{n,N}'$ with respect to the weak convergent speed. However, we have only adopted the Gaussian assumptions more for its neatness rather than absolute necessity, as similar Khintchine-type inequalities for log-concave distributions can also be found in the literature; see \citet{carbery2001distributional} and the references therein.

\subsection{Berry-Esseen bound for  one-dimensional Studentized incomplete U-statistics.}
While our B-E bound in \thmref{main} captures that the version of $U_{n, N}'$ normalized by the theoretical limiting variance $\sigma^2$ is singularity-agnostic with respect to both the limiting distribution and convergent speed, it would be ideal to be able to establish the same for the practical version normalized by the data-driven ``divide-and-conquer" \emph{Studentizer} $\hat{\sigma}^2$  recommended in \citet[Section 2.3]{sturma2022testing} and  \citet[Appendix A]{chen2019randomized}. Note that establishing B-E bounds for Studentized statistics is a highly non-trivial task that is the subject matter of self-normalized limit theory \citep{SN_limit_theory}, and it was not until  recently that optimal-rate B-E bounds  for Studentized \emph{complete} U-statistics of any kernel degree appeared in works by one of the present authors \citet{leung2023another, leung2023nonuniform}, under minimalistic moment assumptions. We note that the uniform B-E bound for Studentized U-statistics in \citet[Theorem 3.1]{leung2023another} also very much parallels the bound for standardized U-statistics in \eqref{BEustat}, in that an informative moment ratio like $\sigma_h/\sigma_g$ also shows up.

\subsection{Extension to multivariate or high-dimensional settings}

As has been mentioned  at the end of \secref{bkground}, even more challenging is to develop similar singularity-agnostic B-E bounds in the spirit of \thmref{main}, in the multivariate or high-dimensional settings where multiple polynomial constraints are handled at once by an incomplete U-statistic with a vector-valued kernel; after all, a statistical model such as the factor analysis is described by a multitude of such constraints like the vanishing tetrads in \eqref{tetrad_forms}. This endeavor is particularly well-motivated, as \citet{dufour2017wald} have shown that the Wald statistic for multiple polynomial constraints may  diverge in distribution, and hence a singularity-agnostic alternative is even more sought after with this in mind.

\section*{Acknowledgement}

We thank Professor Mathias Drton for some helpful discussion during the early stages of conceiving this project. The project is partially funded by the European Research Council (ERC) under the European Union's Horizon 2020 research and innovation programme (grant agreement No 883818). Nils Sturma acknowledges support by the Munich Data Science Institute at Technical University of Munich via the Linde/MDSI PhD Fellowship program.

\appendix

\section{Review on Sub-Weibull random variables} \label{app:subweibull_review_sec} 
We will invoke the theory of sub-Weibull random variables for the proofs in \appsref{additional_pf_main_sec} and \appssref{sec_claim_pf}, which can be found in a series of recent works  such as \citet{kuchibhotla2022moving, zhang2022sharper, wong2020lasso, vladimirova2020sub}. We  summarize the required results and refer the reader to these references for details. 

\begin{definition}[Sub-Weibull$(\alpha)$ random variable and its properties] For a given $\alpha > 0$, let $ \psi_{\alpha} (\cdot)$ be the function
\[
 \psi_{\alpha} (x)= \exp(x^\alpha) -1 \text{ for } x \geq 0. 
\]
For a real-valued random variable $Y$, consider its $\psi_{\alpha}$-Orlicz norm defined by
\[
\|Y\|_{\psi_\alpha} \equiv \inf \{\eta >0 : \bE[  \psi_{\alpha} (|Z|/\eta)] \leq 1\}; 
\]
then $Y$ is said to be sub-Weibull of order $\alpha$, or sub-Weibull$(\alpha)$,  if $\|Y\|_{\psi_\alpha} < \infty$.  
\end{definition}

The above definition comes from \citet[Definition 2.2]{kuchibhotla2022moving}.
Now we consider the following properties of sub-Weibull random variables:
\begin{lemma} [Relation with $p$-norms] \label{lem:weibull_p_norm_relation}
If $Y$ is sub-Weibull$(\alpha)$, then there exists absolute  constants $C(\alpha), c(\alpha) >0$ depending only on $\alpha$ such that
\begin{equation*} \label{equiv_subweibull_norms}
c(\alpha)  \sup_{ p \geq 1} p^{-1/\alpha} \|Y\|_p \leq \|Y\|_{\psi_\alpha} \leq C(\alpha) \sup_{ p \geq 1} p^{-1/\alpha} \|Y\|_p;
\end{equation*}
\end{lemma}

\begin{lemma} [Bernstein inequality for  independent sub-Weibull$(\alpha)$ random variables  for $\alpha \leq 1$]\label{lem:bernstein_sum_weibull}
If $Y_1, \dots, Y_n$ are independent and identically distributed mean-0 sub-Weibull$(\alpha)$ random variables with $\alpha \leq 1$, for some absolute constants $\cA_\alpha$ and $\cC_\alpha$ depending only on $\alpha$ and an absolute constant $\gamma$ (whose value is $\approx 1.78$),  one has the following Bernstein-type inequality that for $z \geq 0$, 
\begin{align*}
&P\Big( |\sum_{i=1}^n Y_i| > z \Big) \\
&\leq 2 \exp \Bigg\{- \Bigg(  
\frac{z^\alpha}{ ( 4e \cC_\alpha \gamma^{2/\alpha} \cA_\alpha \max_{1 \leq i \leq  n}\|Y_i\|_{ \psi_\alpha} )^{\alpha}}
 \wedge
  \frac{z^2}{ 16 e^2 \cC_\alpha^2 \sum_{i=1}^n \|Y_i\|^2_{\psi_\alpha}}\Bigg) \Bigg\} \\
   &=  
 \begin{cases} 
   2 \exp \Big(-   \frac{z^\alpha}{ ( 4e\cC_\alpha  \gamma^{2/\alpha} \cA_\alpha \|Y_1\|_{ \psi_\alpha} )^{\alpha}}  \Big)
 & \text{if } 
 z >  
 4 e \cC_\alpha \Big( \frac{n}{ \gamma^2 \cA_\alpha^\alpha   } \Big)^{\frac{1}{2 - \alpha}} \|Y_1\|_{\psi_\alpha}
  ;\\
  2 \exp \Big( -   \frac{z^2}{ 16 e^2 \cC_\alpha^2 n \|Y_1\|^2_{\psi_\alpha}}\Big)       & 
  \text{if } 
  z \leq
4 e \cC_\alpha \Big( \frac{n}{ \gamma^2 \cA_\alpha^\alpha   } \Big)^{\frac{1}{2 - \alpha}} \|Y_1\|_{\psi_\alpha}.
  \end{cases}
\end{align*}
In particular, for a sufficiently small $z$, the tail probability exhibits sub-Gaussian decay in $z$.
\end{lemma}

\begin{lemma}[Maximal inequality for sub-Weibull random variables]\label{lem:max_ineq}
For $\alpha >0$, let $Y_1, \dots, Y_k$ be $k$ sub-Weibull$(\alpha)$ random variables, not necessarily independent. For an absolute constant $C(\alpha) >0$ depending only on $\alpha$, 
\[
\|\max_{1 \leq i \leq k} Y_i\|_{\psi_\alpha} \leq C(\alpha) \psi_\alpha^{-1} (2k) \max_{1 \leq i \leq k} \|Y_i\|_{\psi_\alpha}
\]
\end{lemma}

 \lemsref{weibull_p_norm_relation} and \lemssref{bernstein_sum_weibull} respectively come from  \citet[Lemma 5]{wong2020lasso} and  \citet[Theorem 1$(c)$]{zhang2022sharper}. The maximal inequality in \lemref{max_ineq} can be shown by  standard arguments in the literature regarding Orlicz norms:

\begin{proof}[Proof of \lemref{max_ineq}]
This is  \citet[Lemma 2.2.2]{vaart1996weak}.  Strictly speaking, \citet[Lemma 2.2.2]{vaart1996weak} is only applicable to the case when $\alpha \geq 1$, in which case $\psi_\alpha$ is a convex function and the factor $\psi_\alpha^{-1} (2k)$ can be further simplified to   $\psi_\alpha^{-1} (k)$. However, when $\alpha < 1$, $\psi_\alpha$ is not convex, and one has to remain with the factor $\psi_\alpha^{-1} (2k)$ precisely because convexity was  used in the proof of  \citet[Lemma 2.2.2]{vaart1996weak} to simplify it to $\psi_\alpha^{-1} (k)$
\end{proof}

\begin{lemma} [Sub-Weibull-ness of Gaussian polynomial] \label{lem:GaussPolyIsSubWeibull}
Let $\bG =
 (g_1, \dots, g_r)^T \sim \cN_r(0, \Sigma)$ be an $r$-variate centered normal random vector, and $Q(\bG) = Q(g_1, \dots, g_r)$ be a polynomial of degree $k$ in the variables $g_1, \dots, g_r$.
 Then $Q(\bG)$ is sub-Weibull of order $2/k$.
\end{lemma}
\begin{proof} [Proof of \lemref{GaussPolyIsSubWeibull}]
This is an easy consequence of
\begin{itemize}
\item the fact that a normal random variable is sub-Weibull of order $2$,
\item
 \citet[Proposition D.2]{kuchibhotla2022moving} on the sub-Weibull-ness of a product of possibly dependent sub-Weibull random variables, and 
 \item that a sum of sub-Weibull random variables each of order no more than $\alpha$ is still a sub-Weibull random variable of order $\alpha$.   
 \end{itemize}

\end{proof}

\section{Additional material for \secref{tale}} \label{app:misc_sec} 

This section collects a few miscellaneous facts mentioned in \secref{tale}. 

\subsection{About the variances}\label{app:misc_facts}
\begin{lemma}[Miscellaneous facts] \label{lem:facts} 
These facts are true for the two variances $\sigma_g^2$ and $\sigma_h^2$ defined in \secref{tale}:
\begin{enumerate}
\item $\sigma_g^2 = \frac{  (\nabla f (\Theta))^T V(\Theta) \nabla f (\Theta)}{m^2}$.
\item $\sigma_h > 0$ under the assumption \eqref{non_singular_assumption_Theta}.
\end{enumerate}

\end{lemma}
\begin{proof} [Proof of \lemref{facts}$(i)$]
Let's index the gradient vector $\nabla f(\Theta) = ({\nabla f(\Theta)}_{uv})_{ 1\leq u \leq v \leq p}$  by the upper diagonal coordinates in $\Theta$. For a given pair $1 \leq u \leq v \leq p$ and a given set of $r$ pairs $\{(u_1, v_1), \dots, (u_r, v_r)\}$ such that $1 \leq u_i \leq v_i \leq p$ for each $i \in [r]$, define 
\[
\frake^{uv}_{u_1v_1, \dots, u_r v_r} =   \big|\{(u_i, v_i): (u_i, v_i) = (u, v)\}\big|,
\]
which counts the number of pairs $(u_i, v_i)$ that are same as $(u, v)$, where we  remark that the collection $\{(u_1, v_1), \dots, (u_r, v_r)\}$ is a multiset because a given $(u_i, v_i)$ may appear in it more than once. Then, it is easy to see from \eqref{form_of_f} that
\begin{equation} \label{f_gradient_form_uv}
{\nabla f(\Theta)}_{uv} = \\
  \sum_{r=1}^{m}    \sum_{ \substack{ 1 \leq u_i \leq v_i \leq p \\  \text{ for } i \in [r]} }
 \frake^{uv}_{u_1v_1, \dots, u_r v_r}  a_{u_1v_1, \dots, u_r v_r} \theta_{uv}^{\frake^{uv}_{u_1v_1, \dots, u_r v_r}  -1 }\prod_{\substack{i \in [r]: \\(u_i, v_i) \neq (u, v) }} \theta_{u_iv_i},
\end{equation}
At the same time, from \eqref{g_fn_structure}, we  rewrite $g(\bX)$, with the data $\bX$ plugged in, as
\begin{align} 
&g(\bX)   \notag \\
&=  a_0 +
  \frac{1}{m}\sum_{r=1}^{m}   \;\ \sum_{ \substack{ 1 \leq u_i \leq v_i \leq p \\  \text{ for } i \in [r]} } \;\ \sum_{j =1}^r  a_{u_1 v_1, \dots, u_r v_r }   \hat{\theta}_{u_j v_j} (\bX) \prod_{i \in [r] : i \neq j}\theta_{u_i v_i} \notag\\
    &= a_0 +
  \frac{1}{m}\sum_{r=1}^{m}  \;\ \sum_{ \substack{ 1 \leq u_i \leq v_i \leq p \\  \text{ for } i \in [r]} } \;\ \sum_{j =1}^r     \Bigg\{a_{u_1 v_1, \dots, u_r v_r }   \theta_{u_jv_j}^{\frake^{u_jv_j}_{u_1v_1, \dots, u_r v_r}  -1 } \prod_{\substack{i \in [r] : \\ (u_i, v_i) \neq (u_j, v_j)}}\theta_{u_i v_i} \Bigg\}  \hat{\theta}_{u_j v_j} (\bX) \notag\\
  &= a_0 + 
  \frac{1}{m}\sum_{r=1}^{m}   \sum_{ \substack{ 1 \leq u_i \leq v_i \leq p \\  \text{ for } i \in [r]} } 
  \notag \\
  & 
  \qquad
  \sum_{\substack{1\leq u \leq v \leq p: \\(u, v) = (u_i, v_i)\\ \text{ for some } i \in [r]}}
 \Bigg\{ \frake^{uv}_{u_1v_1, \dots, u_r v_r} a_{u_1 v_1, \dots, u_r v_r }   \theta_{uv}^{\frake^{uv}_{u_1v_1, \dots, u_r v_r}  -1 } \prod_{\substack{i \in [r] : \\ (u_i, v_i) \neq (u, v)}}\theta_{u_i v_i} \Bigg\}   \hat{\theta}_{u v}(\bX), \label{last_form_for_gX}
\end{align}
where \eqref{last_form_for_gX} comes from the fact that the innermost sum 
\begin{equation} \label{innermost_sum_of_gX}
\sum_{j =1}^r     \Bigg\{a_{u_1 v_1, \dots, u_r v_r }   \theta_{u_jv_j}^{\frake^{u_jv_j}_{u_1v_1, \dots, u_r v_r}  -1 } \prod_{\substack{i \in [r] : \\ (u_i, v_i) \neq (u_j, v_j)}}\theta_{u_i v_i} \Bigg\}  \hat{\theta}_{u_j v_j} (\bX)
\end{equation}
can be re-expressed as
\[
  \sum_{\substack{1\leq u \leq v \leq p: \\(u, v) = (u_i, v_i)\\ \text{ for some } i \in [r]}}
 \Bigg\{ \frake^{uv}_{u_1v_1, \dots, u_r v_r} a_{u_1 v_1, \dots, u_r v_r }   \theta_{uv}^{\frake^{uv}_{u_1v_1, \dots, u_r v_r}  -1 } \prod_{\substack{i \in [r] : \\ (u_i, v_i) \neq (u, v)}}\theta_{u_i v_i} \Bigg\}   \hat{\theta}_{u v} (\bx),
\]
because in \eqref{innermost_sum_of_gX} the index $j$  essentially ``sweeps" through the $r$ pairs in the multiset $\{(u_1, v_1), \dots, (u_r, v_r)\}$. It is now easy to see that \lemref{facts}$(i)$ is true by  applying the  bilinearity  of the covariance operator to $\text{cov} (g(\bX), g(\bX))$ with the expression in \eqref{last_form_for_gX},  and compare it to $(\nabla f (\Theta))^T V(\Theta) \nabla f (\Theta)$ using the expression in \eqref{f_gradient_form_uv}.
\end{proof}

\begin{proof} [Proof of \lemref{facts}$(ii)$]
Since $f$ is not a constant polynomial, considered as a function in $\bR^{pm}$ the kernel $h$ in \eqref{polynomial_kernel} is also a non-constant polynomial. Hence, the \emph{algebraic set}
\[
\{(\bx_1, \dots, \bx_m) : h(\bx_1, \dots, \bx_m)  = 0\}
\]
must be a proper subset of $\bR^{pm}$ \citep[Proposition 5, Chapter 1, p.3]{CLO2015}, which implies it has a dimension strictly less than $p m$  \citep[Corollary 3.4.5]{RR1991} and is hence of measure zero in $\bR^{pm}$. Given that $(\bX_1, \dots, \bX_m)$ has a non-degenerate joint distribution under \eqref{non_singular_assumption_Theta}, it must be  that  $\var_{H}[h(\bX_1, \dots, \bX_m)] > 0$ since $P_{H}( h(\bX_1, \dots, \bX_m) = 0) = 0$ and $\bE_{H}[h] = 0$.
\end{proof}

We  now give more discussion about our assumption in \eqref{non_singular_assumption_Theta}. As briefly touched upon in the main text, it readily follows from a result of \citet[Proposition 8.2]{eatonmultistat} that the invertibility of $\Theta$ implies that the matrix $V(\Theta)$ is also invertible; in fact, an implication in the opposite direction is also true.  As such, under \eqref{non_singular_assumption_Theta}, the singularity status of $H$ can be defined in terms of the zero'ness of $\nabla f(\Theta)$ because the limiting variance 
\begin{equation}\label{limiting_var_of_wald}
(\nabla f(\Theta))^T V(\Theta) \nabla f(\Theta)
\end{equation}
of $\sqrt{n}f(\hat{\Theta})$
is zero if and only if $\nabla f(\Theta) = 0$. If we relax \eqref{non_singular_assumption_Theta} so that $\Theta$ can become non-invertible, then the singularity of $H$  has to be defined directly in terms of the zeroness of  \eqref{limiting_var_of_wald}, taking into account the eigenspace corresponding to the zero eigenvalues of $V(\Theta)$; moreover, certain polynomials $f$ may also have to be excluded from consideration, because their corresponding kernel $h$ in \eqref{polynomial_kernel} may end up having zero variance $\sigma_h^2$ under $H$.  The Berry-Esseen-type theory developed in this work would not be affected, but we have opted to assume \eqref{non_singular_assumption_Theta}  to streamline the exposition.

%

\subsection{The leading term of $T_f$ in \eqref{wald_stat_standardized} is $U_n$.} \label{app:leading_term_of_Wf}
In light of \lemref{facts}$(ii)$, it suffices to show that $f(\hat{\Theta}) = U_n + o_p(n^{-1/2})$.
By simple enumerations, 
one can first rewrite  $U_n$ as
\begin{equation} \label{U_n_explicit_form}
U_n = a_0 + \sum_{r=1}^m  \sum_{  1 \leq u_i \leq v_i \leq p \text{ for } i \in [r] }a_{u_1v_1, \dots, u_r v_r}  U_{u_1v_1, \dots, u_r v_r},
\end{equation}
where each 
\begin{equation}\label{each_u_stat_uv1tor}
U_{u_1v_1, \dots, u_r v_r} =  {n \choose r}^{-1} \sum_{(i_1, \dots, i_r) \in I_{n, r}} h_{u_1v_1, \dots, u_r v_r} (\bX_{i_1}, \dots, \bX_{i_r})
\end{equation}
is a U-statistic of degree $r$ with the symmetric kernel $h_{u_1v_1, \dots, u_r v_r} : (\bR^p)^r \rightarrow \bR$ defined by
\[
h_{u_1v_1, \dots, u_r v_r}  (\bx_1, \dots \bx_r)= \frac{1}{r!} \sum_{\pi \in \cS_r}  \hat{\theta}_{u_1 v_1}({\bf x}_{\pi_1}) \dots \hat{\theta}_{u_r v_r } ({\bf x}_{\pi_r}),
\]
for $\cS_r$ being the set of all permutations of $(1, \dots, r)$.
Moreover, letting  $\hat{\theta}_{uv} = n^{-1} \sum_{i=1}^n \hat{\theta}_{uv}(\bX_i)$, one can write $f(\hat{\Theta})$ as 
\begin{equation}  \label{fhatTheta_form}
f(\hat{\Theta}) = a_0 + \sum_{r=1}^m  \sum_{ 1 \leq u_i \leq v_i \leq p \text{ for } i \in [r] }a_{u_1 v_1, \dots, u_r   v_r}  \hattheta_{u_1 v_1} \cdots \hattheta_{u_r v_r}.
\end{equation}
Comparing \eqref{U_n_explicit_form} and \eqref{fhatTheta_form}, it further suffices to show that 
\begin{equation} \label{leading_term_boils_down}
\hattheta_{u_1 v_1} \cdots \hattheta_{u_r v_r} = U_{u_1v_1, \dots, u_r v_r} + o_p(n^{-1/2}). 
\end{equation}

Further analyzing $\hattheta_{u_1 v_1} \cdots \hattheta_{u_r v_r}$ gives
\begin{align*}
&\hattheta_{u_1 v_1} \cdots \hattheta_{u_r v_r}   \\
&= \frac{\sum_{i_1 =1}^n \cdots \sum_{i_r = 1}^n \hattheta_{u_1 v_1}(\bX_{i_1}) \cdots \hattheta_{u_r v_r}(\bX_{i_r}) }{n^r} \\
&= 
 \frac{ \sum_{1 \leq i_1 \neq \dots \neq i_r \leq n} \hattheta_{u_1 v_1}(\bX_{i_1}) \cdots \hattheta_{u_r v_r}(\bX_{i_r}) }{n^r} + 
 \frac{\{ \dots \text{ a sum of $n^r - {n \choose r} r! $ many terms} \dots \}}{n^r} \\
 &= \frac{r!\sum_{(i_1, \dots, i_r) \in I_{n, r}} h_{u_1v_1, \dots, u_r v_r}  (\bX_{i_1}, \dots \bX_{i_r})}{n^r} +   \frac{\{ \dots \text{ a sum of $O(n^{r-1})$ many terms} \dots \}}{n^r}. 
\end{align*}
Compare the last expression with \eqref{each_u_stat_uv1tor}, one can readily see that 
\begin{align*}
\hattheta_{u_1 v_1} \cdots \hattheta_{u_r v_r} - U_{u_1v_1, \dots, u_r v_r}&= \frac{n (n-1) \cdots (n-r +1) - n^r}{n^r} U_{u_1v_1, \dots, u_r v_r} + O_p(n^{-1}) \\
&= O(n^{-1}) U_{u_1v_1, \dots, u_r v_r} + O_p(n^{-1}),
\end{align*}
which implies \eqref{leading_term_boils_down} because $U_{u_1v_1, \dots, u_r v_r}  = O_p(1)$.

\section{Additional proofs for \secref{main_results}} \label{app:additional_pf_main_sec}

\subsection{Proof of \lemref{BE_bdd_Bn}, the B-E bound for $B_n$} \label{app:BE_bdd_Bn_pf_sec}

To simplify notation,  for any  ${\bf i} =  (i_1, \dots, i_m) \in I_{n,m}$, we will use  the shorthand ${\bf X}_{\bf i} = ({\bf X}_{i_1}, \dots, {\bf X}_{i_m})$, $Z_{\bf i} = Z_{(i_1, \dots, i_m)}$ and $h({\bf X}_{\bf i}) = h({\bf X}_{i_1}, \dots, {\bf X}_{i_m})$. Moreover, we will let $p = p_{n, N}$. It suffices to show that for any $\eta \in (3/7, 1/2)$, 
\begin{multline} \label{showing_goal_for_Bn_BE_bdd}
\sup_{z' \in \bR}\Bigg| P_H\Bigg(  \frac{ \sqrt{N} B_n }{\sigma_h } \leq  z'\Bigg) -  \Phi(z') \Bigg| \leq 
  \frac{C_1}{\sqrt{N}} \frac{\sqrt{1 - 3p_{n, N} + 3p_{n, N}^2}}{1 - p_{n, N}} 
  \Bigg( \frac{\sqrt{1 +n^{- \eta}}}{1 - n^{ -\eta}} \Bigg)  \\
+ \frac{C_2 \sqrt{m}}{\sqrt{n} (1 - n^{-\eta})} + 6 \exp \Big(  -  \frac{  n^{1 -2 \eta}}{C_3(m) }\Big)  + \frac{ C_4(m) }{ n^{2(1- \eta)}}.
\end{multline}

 We shall first define   two new U-statistics, 
\[
U_{h^2} =  {n \choose m}^{-1} \sum_{{\bf i} \in I_{n,m}} h^2 ({\bf X}_{\bf i})  \text{ and } U_{h^4} = {n \choose m}^{-1} \sum_{{\bf i} \in I_{n,m}} h^4({\bf X}_{\bf i}),
\]
with the non-negative symmetric kernel $h^2$ and $h^4$ derived from $h$. 
Since
\begin{multline*}
\sup_{z' \in \bR}\Bigg| \PH\Bigg(  \frac{ \sqrt{N} B_n }{\sigma_h } \leq  z'\Bigg) -  \Phi(z') \Bigg| \\
 =  \sup_{z'\in \bR}  \Bigg| \bEH\Bigg[ \bE\Bigg[ I \Bigg(  \frac{ \sqrt{N} B_n }{\sigma_h } \leq  z'\Bigg) -  \Phi(z')  \Bigg| \{{\bf X}_{\bf i}\}_{{\bf i} \in I_{n, m}} \Bigg] \Bigg]\Bigg| \\
\leq   \sup_{z'\in \bR}  \bEH \Bigg[ \Bigg| P\Bigg(  \frac{ \sqrt{N} B_n }{\sigma_h } \leq   z'  \Bigg| \{{\bf X}_{\bf i}\}_{{\bf i} \in I_{n, m}}\Bigg) -  \Phi(z') \Bigg|  \Bigg],
\end{multline*} 
using   the triangular inequality  we  can further get
\begin{align}
\sup_{z' \in \bR}\Bigg| \PH\Bigg(  \frac{ \sqrt{N} B_n }{\sigma_h } \leq  z'\Bigg) -  \Phi(z') \Bigg|  \leq E_1 + E_2, \label{ep1_bdd_by_E1_E2}
\end{align}
 where 
\[
E_1 \equiv \sup_{z' \in \bR} \bEH 
\Bigg[ \Bigg|P\Bigg( \frac{  \sqrt{N} B_n}{\sqrt{U_{h^2}}}   \leq  \frac{\sigma_h z'}{\sqrt{U_{h^2} }} \Bigg| \{{\bf X}_{\bf i}\}_{{\bf i} \in I_{n, m}} \Bigg) - \Phi \Bigg( \frac{\sigma_h z'}{\sqrt{ U_{h^2}}}\Bigg) \Bigg| \Bigg] 
\]
and
\[
E_2 \equiv \sup_{z' \in \bR} \bEH \Bigg[ \Bigg|  \Phi \Bigg( \frac{\sigma_h z'}{\sqrt{ U_{h^2}}}\Bigg) -  \Phi (  z' ) \Bigg|   \Bigg].
\]
Hence, we will bound $E_1$ and $E_2$ separately:

\subsubsection{Bounding $E_1$} \label{app:E1bdd}
We use $\|\cdot\|_{\psi_\alpha, H}$ to denote a sub-Weibull norm defined with expectation taken under $H$.
 By defining $\tau_h \equiv \bE_H[h^4]$, we write the Hoeffding decomposition \citep[p.137]{de2012decoupling} of $U_{h^2}$ and $U_{h^4}$,
\begin{equation} \label{H_decomp_h2_h3}
U_{h^2} - \sigma_h^2= \sum_{r=1}^m {m \choose r} U_{h^2}^{(r)} \text{ and } \quad
  U_{h^4} - \tau_h= \sum_{r=1}^m {m \choose r} U_{h^4}^{(r)},
\end{equation}
where for each $r$, the projection U-statistics look like
\begin{multline*}
U_{h^2}^{(r)} \equiv {n \choose r}^{-1} \sum_{(i_1, \dots, i_r) \in I_{n, r}}  \pi_r (h^2) \Big({\bf X}_{i_1}, \dots, {\bf X}_{i_r}\Big) \text{ and } \\
U_{h^4}^{(r)} \equiv {n \choose r}^{-1} \sum_{(i_1, \dots, i_r) \in I_{n, r}}  \pi_r (h^4) \Big({\bf X}_{i_1}, \dots, {\bf X}_{i_r}\Big);
\end{multline*}
recall the generic Hoeffding projection kernels  defined in \eqref{projection_fns}. We will form  tail probability estimates for $U_{h^2}$ and $U_{h^4}$ via their projections and shall first consider the case when $r = 1$;
 in particular,  we first write
\[
U_{h^2}^{(1)} = \frac{1}{n} \sum_{i=1}^n \pi_1 (h^2) (\bX_i) , 
\]
which is an average of the independent and identically distributed 
\begin{equation} \label{1st_proj_h_sq}
\pi_1 (h^2) (\bX_i) =  \bEH[h^2({\bf X}_i, {\bf X}_1', \dots,  {\bf X}_{m-1}' )| {\bf X}_i]  - \sigma_h^2, \quad  i =1, \dots n,
\end{equation}
where ${\bf X}_1', \dots,  {\bf X}_{m-1}'$ are independent copies of ${\bf X}_1$. From \eqref{1st_proj_h_sq}, it is easy to see that
$
\|\pi_1 (h^2)\|_{1, H} \leq 2 \sigma_h^2
$, which implies the tail probability bound
\begin{equation} \label{first_tail_prob_bdd_for_U_h21}
\PH\Bigg( |U_{h^2}^{(1)}| > \frac{ \sigma_h^2}{m^2 n^\eta}\Bigg) \leq \PH\Bigg(\Big|\sum_{i=1}^n \pi_1 (h^2) (\bX_i)\Big| > \frac{ n^{1 - \eta}\|\pi_1 (h^2)\|_{1, H}}{2 m^2 }\Bigg). 
\end{equation}
From \eqref{1st_proj_h_sq} again and the structure of $h$ in \eqref{polynomial_kernel}, one can also see that each $\pi_1 (h^2) (\bX_i)$
is a  Gaussian polynomial of degree $4$ and hence sub-Weibull$(1/2)$ by \lemref{GaussPolyIsSubWeibull}; hence, given that $n^{1 - \eta} < n^{\frac{1}{2 - 1/2}} = n^{2/3}$ as $\eta > 1/3$, as well as the fact that 
$
c \|\pi_1 (h^2) \|_{1, H} \leq \|\pi_1 (h^2) \|_{\psi_{1/2}, H}
$
 for a small enough absolute constant $c$ (depending on the sub-Weibull order $1/2$) via \lemref{weibull_p_norm_relation},  one can further apply the  sub-Gaussian part of the Bernstein's inequality for sub-Weibull random variables in \lemref{bernstein_sum_weibull} to the right hand side of \eqref{first_tail_prob_bdd_for_U_h21}, and further get that for $C (m) >0$, 
\begin{equation} \label{tail_prob_bdd_for_U_h21}
\PH\Bigg( |U_{h^2}^{(1)}| > \frac{ \sigma_h^2}{m^2 n^\eta}\Bigg) 
\leq 2 \exp \Big(  -  \frac{   n^{1 -2 \eta}}{C (m)  }\Big).
\end{equation}
Similarly, using that each $\pi_1 (h^4)(\bX_i)$  is a Gaussian polynomial of degree $8$ and hence sub-Weibull$(1/4)$ by \lemref{GaussPolyIsSubWeibull}, as well as that $n^{1 - \eta} < n^{\frac{1}{2 - 1/4}} = n^{4/7}$ (as $\eta > 3/7$) and 
$
\|\pi_1 (h^2) \|_{1, H} \leq C\|\pi_1 (h^2) \|_{\psi_{1/2}, H}
$, one can derive in a completly analogous manner that
\begin{equation} \label{tail_prob_bdd_for_U_h41}
P_H\Bigg( |U_{h^4}^{(1)}| > \frac{\tau_h}{m^2 n^\eta}\Bigg)\leq 2 \exp \Big(  -  \frac{  n^{1 - 2\eta} }{C(m) }\Big).
\end{equation}
Now, using the Hoeffding decomposition for $U_{h^2}$ in \eqref{H_decomp_h2_h3}  and the tail bound for $U_{h^2}^{(1)}$ in \eqref{tail_prob_bdd_for_U_h21},  by the union bound we get that

\begin{align}
P_H\Big(|U_{h^2} - \sigma_h^2| >  \frac{ \sigma_h^2}{ n^\eta}\Big) 
&\leq  \sum_{r = 1}^m P_H \Bigg( |U_{h^2}^{(r)}| > {m \choose r}^{-1}\frac{ \sigma_h^2}{m n^\eta} \Bigg)  \notag\\
 &\leq 2 \exp \Big(  -  \frac{ n^{1 -2 \eta}}{C(m)  }\Big)  +  \sum_{r = 2}^ m \frac{ \Big(m n^{\eta} {m \choose r}\Big)^2\bE_H \Big[ (U_{h^2}^{(r)})^2 \Big]}{\sigma_h^4} \notag\\
 &\leq 2 \exp \Big(  -  \frac{ n^{1 -2 \eta}}{C_1(m)}\Big)  + C_2(m)\sum_{r = 2}^ m \frac{ \tau_h }{ n^{r - 2\eta}\sigma_h^4} \label{tail_prob_bdd_for_U_h2},
\end{align}
where the last inequality comes from standard second-moment bound for mean 0 U-statistics \citep[Theorem 2.1.3, p.72]{korolyuk2013theory}. Similarly, using the Hoeffding decomposition for $U_{h^4}$ in \eqref{H_decomp_h2_h3}  and the tail bound for $U_{h^4}^{(1)}$ in \eqref{tail_prob_bdd_for_U_h41},
\begin{equation}
P_H\Big(|U_{h^4} - \tau_h| >  \frac{ \tau_h}{ n^\eta}\Big) \leq 2 \exp \Big(  -  \frac{  n^{1 - 2\eta} }{C_1(m) }\Big) + C_2(m) \sum_{r = 2}^ m \frac{ \bE_H[h^8] }{ n^{r - 2\eta}\tau_h^2}. \label{tail_prob_bdd_for_U_h4}
\end{equation}

Define the events
$
\cE_1  \equiv \Bigg\{|U_{h^2} - \sigma_h^2| >  \frac{ \sigma_h^2}{ n^\eta}\Bigg\}$ and $
\cE_2  \equiv \Bigg\{|U_{h^4} - \tau_h| >  \frac{ \tau_h}{ n^\eta}\Bigg\}
$.
First, \emph{conditional on} the observed data $\{{\bf X}_i \}_{i \in [n]}$,
\begin{align*}
\var\Bigg(\sqrt{N} B_n \Bigg| \{{\bf X}_i\}_{i \in [n]}\Bigg)
 &= 
\var\Bigg( {n \choose m}^{-1/2}  \sum_{{\bf i} \in I_{n,m}   } \frac{(Z_{\bf i} - p)h({\bf X}_{\bf i})}{\sqrt{p(1 - p)}}\Bigg| \{{\bf X}_i\}_{ i \in [n]}\Bigg) \\
& = 
 {n \choose m}^{-1} \sum_{{\bf i} \in I_{n,m}} h^2 ({\bf X}_{\bf i}) \equiv  U_{h^2};
\end{align*}
hence, 
by the  classical B-E bound \citep[Theorem 3.6, p.54]{chen2010normal}, for any $z' \in \bR$, we have
\begin{align}
&\Bigg| P\Bigg(\frac{\sqrt{N} B_n}{ \sqrt{U_{h^2}} } \leq \frac{\sigma_h z'}{ \sqrt{ U_{h^2} }}\Bigg| \{{\bf X}_i \}_{ i \in [n]} \Bigg)- 
\Phi \Bigg( \frac{ \sigma_h z'}{\sqrt{U_{h^2}}}\Bigg) \Bigg| \notag \\
&\leq C \sum_{\boldi \in I_{n, m}} \bE \Bigg[ \frac{ |Z_\boldi - p|^3 |h(\bX_\boldi)|^3}{ (U_{h^2}{n \choose m}  p (1- p))^{3/2} } \bigg|   \{\bX_i\}_{ i \in [n]} \Bigg]  \notag \\
&\leq C \sum_{\boldi \in I_{n, m}}
\sqrt{ \frac{\bE[  |Z_\boldi - p|^4 ]  h^4(\bX_\boldi)}{ \Big(U_{h^2}{n \choose m}  p (1- p)\Big)^2 }}
\sqrt{ \frac{\bE[  |Z_\boldi - p|^2 ]  h^2(\bX_\boldi)}{ U_{h^2}{n \choose m}  p (1- p) }} \text{ by Cauchy's inequality }  \notag \\
&\leq C
\sqrt{\sum_{\boldi \in I_{n, m}} \frac{\bE[  |Z_\boldi - p|^4 ]  h^4(\bX_\boldi)}{ \Big(U_{h^2}{n \choose m}  p (1- p)\Big)^2 }}
\underbrace{\sqrt{\sum_{\boldi \in I_{n, m}} \frac{\bE[  |Z_\boldi - p|^2 ]  h^2(\bX_\boldi)}{ U_{h^2}{n \choose m}  p (1- p) }} }_{= 1}\text{ by Cauchy's inequality } \notag \\
&= C \sqrt{ \frac{ (1 - 3p + 3p^2) U_{h^4} }{U_{h^2}^2 {n \choose m} p(1- p)}}  = C \frac{\sqrt{1 - 3p +3 p^2}}{\sqrt{N (1- p)}}
\frac{\sqrt{U_{h^4}}}{U_{h^2}} \label{classical_BE_bdd_Bn_conditional}.
\end{align}
Next, using $\Omega$ to denote the underlying probability space, by the definition of $E_1$ and  \eqref{classical_BE_bdd_Bn_conditional},  we then write 
\begin{align}
E_1 &\leq C \bEH \Bigg[ \frac{\sqrt{1 - 3p +3 p^2}}{\sqrt{N (1- p)}}
\frac{\sqrt{U_{h^4}}}{U_{h^2}} I( \Omega \backslash (\cE_1 \cup \cE_2))
 \Bigg] + \PH(\cE_1 \cup \cE_2) \notag\\
 &\leq  \frac{C_1}{\sqrt{N}} \frac{\sqrt{1 - 3p + 3p^2}}{1 - p} 
  \Bigg( \frac{\sqrt{1 +n^{- \eta}}}{1 - n^{ -\eta}} \Bigg)  
  + 4 \exp \Big( - \frac{n^{1- 2 \eta}}{C_2(m)}\Big) + \frac{C_3 (m)}{n^{2 (1 - \eta)}}, \label{E1_ineq}
\end{align}
where the last inequality  uses \eqref{tail_prob_bdd_for_U_h2} and \eqref{tail_prob_bdd_for_U_h4}, as well as the hypercontractive property that the ratios $\tau_h/\sigma_h^4$ and $\bEH[h^8]/\tau_h^2$ can both be bounded an absolute constant (\lemref{hypercontraction}).

\subsubsection{Bounding $E_2$} \label{app:E2bdd}
Let $\cE_1$  be defined as in \appref{E1bdd}. Using Taylor's expansion around $z'$,  we get that 
\begin{align}
&\Bigg|  \Phi \Bigg( \frac{\sigma_h z'}{\sqrt{ U_{h^2}}}\Bigg) -  \Phi( z') \Bigg| \cdot I(\Omega \backslash \cE_1) \notag\\
&\leq \Bigg\{  \sup_{z' \in \bR} |z'| \phi \Bigg(\frac{|z'|}{\sqrt{1 +   n^{-\eta}}}\Bigg) \Bigg\}
\Bigg|\frac{\sqrt{U_{h^2}} - \sigma_h}{\sqrt{U_{h^2}}}\Bigg| \cdot I(\Omega \backslash \cE_1) \notag\\
&\leq C \Bigg|\frac{\sqrt{U_{h^2}} - \sigma_h}{\sqrt{U_{h^2}}}\Bigg| \cdot I(\Omega \backslash \cE_1) \text{ by $ \sup_{z' \in \bR} |z'| \phi \Bigg(\frac{|z'|}{\sqrt{1 +   n^{-\eta}}}\Bigg) \leq C$ } \notag \\
&\leq C \Bigg|\frac{\sqrt{U_{h^2}} - \sigma_h}{\sqrt{U_{h^2}}}\Bigg| \frac{\sqrt{U_{h_2}} + \sigma_h}{\sqrt{U_{h_2}}} \cdot I(\Omega \backslash \cE_1)  \text{ by  $U_{h^2} \geq 0$ almost surely } \notag \\
&= C \frac{|U_{h^2} - \sigma_h^2|}{U_{h^2}}  \cdot I(\Omega \backslash \cE_1) \leq \frac{C |U_{h^2} - \sigma_h^2|}{\sigma_h^2(1 - n^{ -\eta})} \text{ by the definition of } \cE_1. \label{E2_ineq_primer}
\end{align}
From the inequality in  \eqref{E2_ineq_primer}, using the definition of $E_2$ and \eqref{tail_prob_bdd_for_U_h2},   we have 
\begin{align} 
E_2 &\leq C  \frac{ \bEH[  |U_{h^2} - \sigma_h^2|  ]}{\sigma_h^2(1 - n^{ -\eta})}  + 
2 \exp \Big(  -  \frac{ n^{1 -2 \eta}}{C_1(m)  }\Big)  +
 C_2(m)\sum_{r = 2}^ m \frac{ \tau_h }{ n^{r - 2\eta}\sigma_h^4}  \notag\\
&\leq C \frac{ \sqrt{m\tau_h}}{\sqrt{n} \sigma_h^2(1 - n^{-\eta})} + 
2 \exp \Big(  -  \frac{ n^{1 -2 \eta}}{C_1(m) }\Big)  +
 C_2(m)\sum_{r = 2}^ m \frac{ \tau_h }{ n^{r - 2\eta}\sigma_h^4},  \notag
\end{align}
where the last inequality uses the bound
\[
\bE_H[|U_{h^2} - \sigma^2_h|] \leq \sqrt{\text{var}_H(U_{h^2})} \leq  \sqrt{\frac{m}{n}\bEH[ (h^2 - \sigma_h^2)^2]} \leq \sqrt{\frac{m \tau_h}{n}};
\]
see \citet[Lemma 5.2.1.A$(i)$, p.183]{serfling1980}. Finally, using the fact that the ratio $\frac{\tau_h}{\sigma_h^4}$ can be bounded by an absolute constant by \lemref{hypercontraction}, we further get
\begin{equation}
E_2 \leq  \frac{C\sqrt{m}}{\sqrt{n} (1 - n^{-\eta})} + 2 \exp \Big(  -  \frac{  n^{1 -2 \eta}}{C_1(m) }\Big)  + \frac{ C_2(m) }{ n^{2(1- \eta)}}.
\label{E2_ineq}
\end{equation}

\subsubsection{Overall bound on  $\sup_{z' \in \bR}\Big| \PH\Big(  \frac{ \sqrt{N} B_n }{\sigma_h } \leq  z'\Big) -  \Phi(z') \Big| $} 
 Combining \eqref{ep1_bdd_by_E1_E2}, \eqref{E1_ineq} and \eqref{E2_ineq}, we have shown \eqref{showing_goal_for_Bn_BE_bdd}.

\subsection{Proof of \lemref{nonunif_BE}, the nonuniform B-E bound} \label{app:nonunif_BE_pf_sec}
We will first state some useful technical results, and the proof will follow after.

\begin{lemma}[Exponential randomized concentration inequality  for a sum of  censored random variables] \label{lem:modified_RCI_bdd}
Let $\xi_1, \dots, \xi_n$ be independent random variables  with mean zero and finite second moments, and for each $i =1, \dots, n$, define 
\[
\bar{\xi}_i = \xi_i I(|\xi_i| \leq 1) + 1 I(\xi_i >1) - 1  I(\xi_i <-1),
\]
an upper-and-lower censored version of $\xi_i$; moreover, let $S =\sum_{i=1}^n \xi_i $ and $\bar{S} = \sum_{i=1}^n \bar{\xi}_i$ be their corresponding sums, and $\Delta_1$ and $\Delta_2$ be two random variables on the same probability space. Let $\delta = \frac{\beta_2 + \beta_3}{4}$ for $\beta_2 \equiv \sum_{i=1}^n \bE[\xi_i^2 I(|\xi_i|>1)]$ and $\beta_3 \equiv \sum_{i=1}^n \bE[\xi_i^3 I(|\xi_i|\leq1)]$, and assume $
\sum_{i=1}^n \bE[\xi_i^2] = 1$. Then it is true that
\begin{align*}
&\bE[e^{ \frac{\barS}{2}} I(\Delta_1 \leq \barS \leq \Delta_2)] \leq 
\left( \mathbb{E}\left[e^{ \bar{S}}\right] \right)^{\frac{1}{2}}\exp\left( - \frac{1}{256 \cdot\delta^2}\right) \\
&+ 8e^{\frac{\delta}{2}} \Biggl\{ 2\sum_{i =1}^n \mathbb{E} [ |\barxi_i|e^{ \frac{\barS^{(i)} }{2} } (|\Delta_1 - \Delta_1^{(i)}| + |\Delta_2 - \Delta_2^{(i)}|)] \\
& + 
\mathbb{E}[|\barS|e^{ \frac{\bar{S}}{2}}(|\Delta_2 - \Delta_1| + 2 \delta)]\\
& +   \sum_{i=1}^n  \Big|\bE[ \barxi_i]\Big| \bE[e^{\frac{ \barS^{(i)}}{2} }(|\Delta_2^{(i)} - \Delta_1^{(i)}| + 2 \delta)]\Biggr\} ,
\end{align*}
where $\Delta_1^{(i)}$ and $\Delta_2^{(i)}$ are any random variables on the same probability space such that $\xi_i$ and $(\Delta_1^{(i)}, \Delta_2^{(i)},  S^{(i)}, \barS^{(i)})$ are independent, where $S^{(i)} \equiv S - \xi_i$ and $\barS^{(i)} \equiv \barS - \barxi_i$. 
 \end{lemma}
 
 \begin{lemma}[Bound on expectation of $\bar{\xi}_i$] \label{lem:exp_xi_bi_bdd} Let $\bar{\xi}_i = \xi_i I(|\xi_i| \leq 1) + 1 I(\xi_i >1) - 1  I(\xi_i <-1)$ with $\bE[\xi_i] = 0$. Then
\[
  \left|\bE[ \barxi_i]\right| \leq   \bE[ \xi_i^2I(|\xi_i| > 1) ]  \leq \bE[|\xi_i|^3] \wedge\bE[ \xi_i^2 ]
\]
\end{lemma}

 \begin{lemma} [Bennett's inequality for a sum of censored random variables]  \label{lem:bennett_censored} 
Let $\xi_1, \dots, \xi_n$ be independent random variables with $\bE[\xi_i] = 0$ for all $i = 1, \dots, n$ and $\sum_{i=1}^n \bE[\xi_i^2] \leq 1$, and define $\barxi_i = \xi_i I(|\xi_i| \leq 1) + 1 I(\xi_i >1) - 1  I(\xi_i <-1)$. 
For any $t >0$ and $\barS =  \sum_{i=1}^n \barxi_i$, we have
\[
\bE[e^{t \barS}] \leq  \exp\left( e^{2t}/4 - 1/4 + t/2\right).
\]
\end{lemma}

\begin{lemma} \label{lem:special_bound}
Let $\xi_1, \dots, \xi_n$ be independent mean zero random variables satisfying $\sum_{i=1}^n \bE[\xi_i^2] = 1$ and let $S = \sum_{i  = 1} \xi_i$. Then for $z \geq 0$ and $p \geq 2$, 
\begin{multline*}
P\Big(S \geq z, \max_{1  \leq i \leq n} |\xi_i| > 1\Big) \leq\\
 2 \sum_{1 \leq i \leq n} P\Bigg( |\xi_i| > \frac{z}{2p} \Bigg) + e^p( 1 + z^2/(4p) )^{-p} \sum_{i=1}^n \bE[\xi_i^2 I(|\xi_i|>1)]
\end{multline*}

\end{lemma}

 \lemsref{modified_RCI_bdd}, \lemssref{exp_xi_bi_bdd} and \lemssref{bennett_censored} are small technical results for proving U-statistic B-E bounds that can be found  in \citet[Appendix A]{leung2023nonuniform}. \lemref{special_bound} comes from  \citep[Lemma  8.3, p.238]{chen2010normal}; 
while it was originally stated for any $z \geq 2$ and the tail probability being bounded  is $P\Big(S \geq z, \max_{1  \leq i \leq n} \xi_i > 1\Big)$ instead, the current version still stands with an essentially identical proof.

\begin{proof}[Proof of \lemref{nonunif_BE}] 
With the uniform B-E bound in \eqref{BEustat}, we can without loss of generality, we assume $|x| \geq 1$; we can actually further assume $x \geq 1$, because one can alternatively consider $-h$ for negative value of $x$.
With the U-statistics constructed with the Hoeffding projected kernels in \eqref{H_proj_u_stat}, we first let 
\begin{multline}\label{xi_Delta_def}
 \xi_i \equiv \frac{g({\bf X}_i)}{\sqrt{n}\sigma_g} \text{ and } S \equiv \sum_{i=1}^n \xi_i, \text{ as well as } \\
 \Delta \equiv {n -1 \choose m -1}^{-1} \sum_{1 \leq i_1 < \dots < i_m \leq n} \frac{\bar{h}_m(X_{i_1}, X_{i_2}, \dots, X_{i_m}) }{\sqrt{n}\sigma_g} = \frac{\sqrt{n}}{ m \sigma_g} \sum_{r=2}^m {m \choose r} U_n^{(r)},
\end{multline}
where
$
\bar{h}_m(x_1 \dots, x_k) \equiv h(x_1 \dots, x_m) - \sum_{i=1}^m g(x_i)$,
so  by the Hoeffding decomposition in \eqref{H_decomp} we can write
\[
 \frac{\sqrt{n}}{m \sigma_g} U_n =  S+ \Delta.
\]
 Moreover, for each $i =1, \dots, n$, we let
\[
\Delta^{(i)} =  {n -1 \choose m -1}^{-1} \sum_{\substack{1 \leq i_1 < \dots < i_m \leq n\\ i_l \neq i \text{ for } l = 1, \dots, m}} \frac{\bar{h}_m(X_{i_1}, X_{i_2}, \dots, X_{i_m}) }{\sqrt{n} \sigma_g}
\]
be the ``leave-one-out" version of $\Delta$ that omits all terms involving ${\bf X}_i$. 
 In addition, we define the censored versions of  some of the above variables:
 \[
\barxi_i =  \xi_i I(|\xi_i| \leq 1) + 1 I(\xi_i >1) - 1  I(\xi_i <-1), \qquad \barS = \sum_{i=1}^n  \barxi_i ,
 \]
\begin{equation*} \label{Delta_censored}
\bar{\Delta}_x = \Delta I\left(|\Delta| \leq  \frac{x +1}{3} \right) +  \frac{x +1}{3} I\left(\Delta >   \frac{x +1}{3}\right) -  \frac{x +1}{3}I\left(\Delta < -  \frac{x +1}{3}\right)
\end{equation*} 
and
\begin{equation*} \label{Delta_censored_i}
\bar{\Delta}_x^{(i)} = \Delta^{(i)} I\left(| \Delta^{(i)}| \leq  \frac{x +1}{3} \right) +  \frac{x +1}{3} I\left( \Delta^{(i)} >   \frac{x +1}{3}\right) -  \frac{x +1}{3}I\left( \Delta^{(i)} < -  \frac{x +1}{3}\right).
\end{equation*} 
Since 
\[
 -  P_H(x - |\bar{\Delta}_x| \leq S \leq x  ) \leq P_H(S + \bar{\Delta}_x \leq  x) - P_H(S \leq x)  \leq P_H(x \leq S \leq x + |\bar{\Delta}_x|),
 \]
we can write
\begin{align}
&\Bigg|P_H\Big( \frac{\sqrt{n}}{m \sigma_g} U_n  \leq x \Big) -\Phi(x)\Bigg| \notag \\
 &\leq \Big|P_H(S + \bar{\Delta}_x \leq   x) - \Phi(x) \Big| + P_H\left(|\Delta| >  \frac{x +1}{3} \right) \notag \\
&\leq  \Big|P_H(S \leq   x) - \Phi(x) \Big| + P_H\left(|\Delta| >  \frac{x +1}{3} \right) +   P_H(x - |\bar{\Delta}_x| \leq S \leq x  )\vee P_H(x \leq S \leq x + |\bar{\Delta}_x|) \notag\\
&\leq \Big|P_H(S  \leq   x) - \Phi(z) \Big| + P_H\left(|\Delta| >  \frac{x +1}{3} \right) +  \notag \\
 &\hspace{2cm}P_H(S \geq x/3,  \max_{1 \leq i \leq n} |\xi_i| > 1)  + P_H(x - |\bar{\Delta}_x| \leq \barS \leq x  )\vee P_H(x \leq \barS \leq x + |\bar{\Delta}_x|) .
\label{chopping_Delta}
\end{align}
Next, we develop bounds for the various terms in \eqref{chopping_Delta}.
\subsubsection{Bound on $|P_H(S  \leq   x) - \Phi(x) | $} \label{sec:1st_little_bdd_nonunif} By the classical B-E bound for a sum of independent random variables \citep[Theorem 3.6, p.54]{chen2010normal}, we have
\[
|P_H(S  \leq   x) - \Phi(x) | \leq 9.4 \sum_{i=1}^n \bE[|\xi_i|^3]
\]

\subsubsection{Bound on $P_H(x - |\bar{\Delta}_x| \leq \barS \leq x  )\vee P_H(x \leq \barS \leq x + |\bar{\Delta}_x|)$} \label{sec:2nd_little_bdd_nonunif} 
We will first bound 
$
P_H(x - |\bar{\Delta}_x| \leq \barS \leq x  )$.
 Applying  \lemref{modified_RCI_bdd} along with the choice  $\Delta_1 =  x - |\bar{\Delta}_x|$, $\Delta_1^{(i)} =  x - |\bar{\Delta}^{(i)}_x|$, $\Delta_2 = \Delta_2^{(i)} = x$, by the fact that $x - |\bar{\Delta}_1| \geq x/3$, we get that
\begin{align}
&P_H(x - |\bar{\Delta}_x| \leq \barS \leq x  ) \leq 
e^{-x/6} \left( \bEH\left[e^{ \barS }\right] \right)^{1/2}\exp\left( - \frac{1}{ 16 (\sum_{i=1}^n \bEH[|\xi_i|^3])^2}\right) \notag\\
&+8 e^{ 1/8 - x/6} 
\Biggl\{  2\sum_{i =1}^n \bEH \Big[|\barxi_i | e^{\barS^{(i)}/2 }  \Big|\bar{\Delta}_x - \bar{\Delta}_x^{(i)}\Big|\Big]  \notag \\
& \hspace{3cm}+
\bEH\Big[|\barS| e^{ \barS/2 } \Big(|\bar{\Delta}_x| + \frac{\sum_{i=1}^n \bE[|\xi_i|^3]}{2} \Big)\Big]  \notag \\
& \hspace{4cm} + \sum_{i=1}^n \Big|\bEH[\bar{\xi}_i]\Big| \bEH\Bigg[ e^{ \barS^{(i)}/2}\Bigg(  |\bar{\Delta}_x^{(i)}|+ \frac{1}{2}\Bigg)\Bigg] 
 \Biggr\} \label{first_app_randomized_ineq},
\end{align}
where we have also used the fact that $\beta_2 + \beta_3 \leq \min (1 , \sum_{i=1}^n \bEH[|\xi_i|^3])$, for $\beta_2 \equiv \sum_{i=1}^n \bEH[\xi_i^2 I(|\xi_i|>1)]$ and $\beta_3 \equiv \sum_{i=1}^n \bEH[\xi_i^3 I(|\xi_i|\leq1)]$. 
Next, we  control the different terms figuring in \eqref{first_app_randomized_ineq}:
\begin{enumerate}
\item  
 $\bEH[e^{ \barS }] \leq C$  by Bennett's inequality (\lemref{bennett_censored});
 \item By Cauchy's inequality, 
\begin{align*}
&\bEH \Big[|\barxi_i | e^{\barS^{(i)}/2 }  \Big|\bar{\Delta}_x - \bar{\Delta}_x^{(i)}\Big|\Big] \\
 &\leq \| |\barxi_i | e^{\barS^{(i)}/2 }\|_{2,H} \|{\Delta} - {\Delta}^{(i)}\|_{2, H}\\
& = \|\barxi_i \|_{2, H} \|e^{\barS^{(i)}/2 }\|_{2, H} \|{\Delta} - {\Delta}^{(i)}\|_{2, H} \text{ by independence}\\
&\leq C   \|\barxi_i \|_{2, H} \|{\Delta} - {\Delta}^{(i)}\|_{2, H} \text{  by Bennett's inequality (\lemref{bennett_censored})} \\
&\leq  C   \|\xi_i \|_{2, H} \|{\Delta} - {\Delta}^{(i)}\|_{2, H} 
\end{align*}

\item Since $\barS/2 \leq e^{|\barS/2|} \leq e^{\barS/2} + e^{- \barS/2}$,
\begin{align*}
 &\bEH\Big[|\barS| e^{ \barS/2 } \Big(|\bar{\Delta}_x| + \frac{\sum_{i=1}^n \bEH[|\xi_i|^3]}{2} \Big)\Big]\\
  &\leq 2  \bEH\Big[ (1 + e^{\barS})\Big(|\bar{\Delta}_x| +  \frac{\sum_{i=1}^n \bEH[|\xi_i|^3]}{2} \Big)\Big] \\
 &\leq C (\|\Delta\|_{2, H} +  \sum_{i=1}^n \bEH[|\xi_i|^3] )  \text{  by Bennett's inequality (\lemref{bennett_censored})}
\end{align*}
\item Since $x \geq 1$, 
\begin{align*}
&\sum_{i=1}^n \Big|\bEH[\bar{\xi}_i]\Big| \bEH\Bigg[ e^{ \barS^{(i)}/2}\Bigg(  |\bar{\Delta}_x^{(i)}|+ \frac{1}{2}\Bigg)\Bigg] \\
&\leq  \sum_{i=1}^n |\bEH[\bar{\xi}_i]|  (\| \Delta^{(i)}\|_{2, H} + C) \text{  by Bennett's inequality (\lemref{bennett_censored})} \\
&\leq  \max_{1 \leq i \leq n}\| \Delta^{(i)}\|_{2, H} + C \sum_{i=1}^n \bEH[|\xi_i|^3] \text{ by \lemref{exp_xi_bi_bdd}}
\end{align*}
\end{enumerate}
With these bounds we can continue from \eqref{first_app_randomized_ineq} and get
\begin{equation*}
P_H(x - |\bar{\Delta}_x| \leq \barS \leq x  ) \leq C e^{-x/6} \Bigg(\sum_{i=1}^n \bEH[|\xi_i|^3] + \sum_{i=1}^n \|\xi_i\|_{2, H} \|\Delta  - \Delta^{(i)}\|_{2,H} + \|\Delta\|_{2, H} + \max_{1 \leq i \leq n} \|\Delta^{(i)}\|_{2, H}\Bigg).
\end{equation*}
With a similar and analogous argument,  one can also prove the same bound for $P(x \leq \barS \leq x +  |\bar{\Delta}_x|   )$, and hence we conclude
\begin{multline*}
P(x - |\bar{\Delta}_x| \leq \barS \leq x  ) \vee P(x \leq \barS \leq x +  |\bar{\Delta}_x|   ) \\
\leq C e^{-x/6} \Big(\sum_{i=1}^n \bEH[|\xi_i|^3] + \sum_{i=1}^n \|\xi_i\|_{2, H} \|\Delta  - \Delta^{(i)}\|_{2, H} + \|\Delta\|_{2, H} + \max_{1 \leq i \leq n} \|\Delta^{(i)}\|_{2, H}\Big).
\end{multline*}

\subsubsection{Bound on $P_H(S \geq x/3,  \max_{1 \leq i \leq n} |\xi_i| > 1)  $} \label{sec:3rd_little_bdd_nonunif} 
By \lemref{special_bound}, since $ \sum_{i=1}^n \bEH[\xi_i^2 I(|\xi_i|>1)]
 \leq \sum_{i=1}^n \bEH[\xi_i^3 ]$ and using the fact that $x \geq 1$, we have 
\[
P_H\Big(S \geq x/3,  \max_{1 \leq i \leq n} |\xi_i| > 1\Big)   \leq C \sum_{i=1}^n \bEH[|\xi_i|^3].
\]
Lastly, with the 3 bounds  in \secsref{1st_little_bdd_nonunif},  \secssref{2nd_little_bdd_nonunif} and \secssref{3rd_little_bdd_nonunif}, from \eqref{chopping_Delta} we get 
\begin{multline}\label{pre_lem_nonunif}
\Bigg|P_H\Big( \frac{\sqrt{n}}{m \sigma_g} U_n  \leq x \Big) -\Phi(x)\Bigg| \leq P_H\left(|\Delta| >  \frac{x +1}{3} \right) + C_1\sum_{i=1}^n \bE_H[|\xi_i|^3] + \\
\hspace{1.5cm} C_2 e^{-x/6} \Big( \sum_{i=1}^n \|\xi_i\|_{2, H} \|\Delta  - \Delta^{(i)}\|_{2, H} + \|\Delta\|_{2, H} + \max_{1 \leq i \leq n} \|\Delta^{(i)}\|_{2, H}\Big).
\end{multline}
By routine U-statistic calculations, one can show that
\[
 \|\Delta  - \Delta^{(i)}\|_{2, H} \leq \frac{\sqrt{2}(m-1) \sigma_h}{ \sqrt{nm(n-m+1)}\sigma_g}, 
 \quad \|\Delta\|_{2, H} \leq \frac{(m-1) \sigma_h}{ \sqrt{m (n-m+1)}\sigma_g}
\]
and
\[
\max_{1 \leq i \leq n} \|\Delta^{(i)}\|_{2, H} \leq 
\frac{\sqrt{n-1}}{\sqrt{n}}\frac{{n -2 \choose m -1}}{{n -1 \choose m -1}}\frac{ (m-1)\sigma_h}{ \sqrt{m(n-m)} \sigma_g} = 
 \frac{\sqrt{n-m}}{\sqrt{n-1}} \frac{ (m-1)\sigma_h}{ \sqrt{nm} \sigma_g}
\]
see  \citet[Lemma  10,1, p.262]{chen2010normal} for instance. Moreover, recall the definition of $\xi_i$ in \eqref{xi_Delta_def}, we have 
\[
\|\xi\|_{2, H} = n^{-1/2} \quad \text{ and } \quad \bEH[|\xi_i|^3] = \frac{\bEH[|g|^3]}{n^{3/2}\sigma_g^3}.
\]
Plugging all these bounds and equalities into \eqref{pre_lem_nonunif} we have proved \lemref{nonunif_BE}.
 
\end{proof}

\subsection{Proof of \thmref{main}$(ii)$, B-E bound for $U_{n, N}'$ when $\sigma_g = 0$} \label{app:pf_degenerate_case}

It suffices to assume $z \geq 0$, otherwise one can replace the kernel $h$ with $-h$. Moreover, note that $\sigma = \sqrt{\alpha_n} \sigma_h$ because $\sigma_g = 0$ by singularity.
 From \eqref{Wn_decomp} we can write the chain of  inequalities
\begin{align}
&P_H\Big(\frac{\sqrt{n} W_n}{\sigma} \leq z\Big) = P_H\Bigg(\frac{\sqrt{N} W_n}{  \sigma_h}  \leq z\Bigg) \notag \\
&\leq  P_H(\sqrt{N} |U_n| > \sigma_h t)+ P_H\Bigg(  \sqrt{N} B_n  \leq   \frac{\sigma_h (z + t)}{\sqrt{1 - p_{n, N}}} \Bigg) \notag\\
&\leq P_H(\sqrt{N} |U_n| > \sigma_h t) + \sup_{z' \in \bR}| P_H(  \sqrt{N} B_n \sigma_h^{-1} \leq  z' ) -  \Phi(z')| + \Phi\left(  \frac{z + t}{\sqrt{1 - p_{n, N}}} \right) \notag\\
 &\leq \sup_{z' \in \bR}| P_H(  \sqrt{N} B_n \sigma_h^{-1} \leq  z' ) -  \Phi(z')| +  \Phi\left(  \frac{z }{\sqrt{1 - p_{n, N}}} \right) + \epsilon_1(t)\notag\\
 &\leq  \sup_{z' \in \bR}| P_H(  \sqrt{N} B_n \sigma_h^{-1} \leq  z' ) -  \Phi(z')| + \Phi(z)+  \epsilon_1(t) + \epsilon_2  \text{ for any } t \geq 0, \label{1st_side_of_BE_bdd}
\end{align}
where the error terms are defined as 
\begin{align*}
&\epsilon_1(t) = P_H(\sqrt{N} |U_n| > \sigma_h t) + \frac{t}{\sqrt{(2\pi)(1 - p_{n, N})}},  \\
& \epsilon_2 = \Bigg|\Phi\left(  \frac{z }{\sqrt{1 - p_{n, N}}} \right) - \Phi(z)\Bigg|,  
\end{align*}
with the dependence on $t$ made explicit. By a reverse argument, one can analogously show that 
\begin{equation*}
P_H\Bigg(\frac{\sqrt{n} W_n}{ \sqrt{\alpha_n} \sigma_h}  \leq z \Bigg) 
\geq  \Phi (z) - \sup_{z' \in \bR}| P_H(  \sqrt{N} B_n \sigma_h^{-1} \leq  z' ) -  \Phi(z')|  - \epsilon_1(t) - \epsilon_2  ;
\end{equation*}
together with \eqref{1st_side_of_BE_bdd} it implies
\begin{equation} \label{BE_bdd_for_Wn_in_t}
\Bigg|P_H\Bigg(\frac{\sqrt{n} W_n}{ \sqrt{\alpha_n} \sigma_h}  \leq z\Bigg)  - \Phi(z) \Bigg| \leq \sup_{z' \in \bR}| P_H(  \sqrt{N} B_n \sigma_h^{-1} \leq  z' ) -  \Phi(z')|   + \epsilon_1(t) + \epsilon_2 \text{ for any } t \geq 0. 
\end{equation}
Combining \eqref{general_BE_bdd_for_rescaled_icu_stat},  \lemref{BE_bdd_Bn} and \eqref{BE_bdd_for_Wn_in_t}, to finish proving  \thmref{main}$(ii)$, it suffices to let $\mathfrak{R} = \epsilon_1(t_m) + \epsilon_2$ for
\[
t_m= C(m)\frac{\sqrt{N}}{n^{d/2}}(\log (2n^m +1 ))^{3m},
\]
and form a bound for $\mathfrak{R} $ by  bounding $\epsilon_1(t_m)$ and $\epsilon_2$ separately:
\subsubsection{Bound on $\epsilon_1(t_m)  $}   
By the  decomposition \eqref{H_decomp},  \lemref{Bernstein_ineq_h_proj}$(i)$ implies that
\begin{equation} \label{ep0t_bound}
P_H(\sqrt{N} |U_n| > \sigma_h t_m) 
\leq\sum_{r = d}^m  P_H\Bigg(\sqrt{n}|U_n^{(r)}|> \frac{\sqrt{n}\sigma_h t_m }{{m \choose r}(m-d+1)\sqrt{N}}\Bigg) 
\leq \frac{C_1(m)}{n}
\end{equation}
for a large enough constant $C(m)$ in the definition of $t_m$.
Hence, we have
\begin{equation} \label{bound_on_ep0_tm}
\epsilon_1 (t_m)  \leq  \frac{C_1(m)}{n}+  \frac{C_2(m) \sqrt{N}(\log (2n^m +1 ))^{3m}}{n^{d/2}\sqrt{1 - p_{n, N}}} \leq  \frac{C(m) \sqrt{N}(\log (2n^m +1 ))^{3m}}{n \sqrt{1 - p_{n, N}}},
\end{equation}
where we have used that $d \geq 2$ by degeneracy.
\subsubsection{Bound on $\epsilon_2$}
By  Taylor expansion around $z$ and the assumption that $z \geq 0$, 
\begin{equation} \label{ep2_ultimate_bdd}
\epsilon_2 \leq     \phi(z) z  \left( \frac{1- \sqrt{1 - p_{n, N}}}{\sqrt{1 - p_{n, N}}}\right) \leq C \frac{p_{n, N}}{1 - p_{n,N} + \sqrt{1 - p_{n, n}}} \leq C \frac{p_{n, N}}{1 - p_{n, N}},
\end{equation}
where we've used $\sup_{z \geq 0}z \phi(z) \leq C$  and multiplying both the numerator and denominator by $1 + \sqrt{1 - p_{n, N}}$ in the second last inequality.

Finally,  collecting \eqref{bound_on_ep0_tm} and \eqref{ep2_ultimate_bdd}, we get that 
\[
\mathfrak{R} = \epsilon_1(t_m) + \epsilon_2 \leq \frac{C(m) \sqrt{N}(\log (2n^m +1 ))^{3m}}{n \sqrt{1 - p_{n, N}}} +    \frac{Cp_{n, N}}{1- p_{n, N}}
\]
as claimed by \thmref{main}$(ii)$.

 \section{Additional proofs for \secref{Bernstein_pf}} \label{app:sec_claim_pf}
 
 \subsection{Proof of \thmref{sharp_ustat_moment}, sharp moment inequality for canonical generalized decoupled U-statistics} \label{app:pf_moment_ineq_decouple_u_stat}
Directly from \citet[Theorem 6, p.21]{adamczak2006moment}, we have the inequality 
 \[
 \bE\bigg[\Big|\sum_{{\bf i} \in [n]^r} \kappa_{\bf i}\Big|^p \bigg] \leq \cK_r^p \sum_{I \subset [r]} \sum_{\cJ \in \cP_I} p^{p(|I^c|+ |\cJ|/2)}
\bE_{I^c}\Big[ \max_{{\bf i}_{I^c}} \| (\kappa_{\bf i})_{{\bf i}_I }\|^p_{\cJ} \Big];
 \]
hence, to prove \thmref{sharp_ustat_moment}, it remains to show that for a given $I \subset [r]$ and  partition $\cJ \in \cP_I$, 
\begin{equation}\label{op_norm_bdd}  
\| (\kappa_{\bf i})_{{\bf i}_I }\|_{\cJ}\leq   \sqrt{ \bE_I \Bigg[\sum_{{\bf i}_I } \kappa_{\bf i}^2\Bigg]}.
\end{equation}
For $k = |\cJ|$,  let  $( f_{{\bf i}_{J_1}}^{(1)}  ( \cdot )  )_{{\bf i}_{J_1} \in [n]^{|J_1|}}, \dots,
( f_{{\bf i}_{J_k}}^{(k)}  ( \cdot )  )_{{\bf i}_{J_k} \in [n]^{|J_k|}}$ be sequences of functions such that 
\begin{equation}\label{Hibert_norm_less_than_1}
\bE 
\Bigg[
\sum_{{\bf i}_{J_j} \in [n]^{|J_j|}} \Big|  f_{{\bf i}_{J_j}}^{(j)}  \Big( (Y^{(l)}_{i_l})_{l \in J_j}\Big)
 \Big|^2
\Bigg] \leq 1 \text{ for each } j = 1, \dots, k.
\end{equation}
One can first rewrite
\begin{align}
 &\bE_I\Bigg[  \sum_{{\bf i}_I}
\kappa_{\bf i}
\prod_{j=1}^k  f_{{\bf i}_{J_j}}^{(j)}  \Big( (Y^{(l)}_{i_l})_{l \in J_j}\Big)
\Bigg] \notag\\
&= \bE_{J_1}  \Bigg[ \sum_{{\bf i}_{J_1} }   f_{{\bf i}_{J_1}}^{(1)}  \Big( Y^{(J_1)}_{{\bf i}_{J_1}}\Big)  \cdots \bE_{J_{k-1}} \Bigg[ \sum_{{\bf i}_{J_{k-1}} }   f_{{\bf i}_{J_{k-1}}}^{(k-1)}  \Big( Y^{(J_{k-1})}_{{\bf i}_{J_{k-1}}}\Big)  \bE_{J_k} \Bigg[ \sum_{{\bf i}_{J_k} }   f_{{\bf i}_{J_k}}^{(k)}  \Big( Y^{(J_k)}_{{\bf i}_{J_k}}\Big)
 \kappa_{\bf i} \Bigg] \Bigg] \cdots \Bigg] \notag \\
 &= \bE_{J_1}  \Bigg[ \sum_{{\bf i}_{J_1} }   f_{{\bf i}_{J_1}}^{(1)}  \Big( Y^{(J_1)}_{{\bf i}_{J_1}}\Big)   \cdots
\bE_{J_t \cup \cdots \cup J_k}
\Bigg[
\sum_{{\bf i}_{J_t \cup \cdots \cup J_k}}
\kappa_{\bf i}
\prod_{j=t}^k  f_{{\bf i}_{J_j}}^{(j)}  \Big( Y^{(J_j)}_{{\bf i}_{J_j}}\Big)
\Bigg] \cdots
 \Bigg] \text{ for any } 1 < t  \leq k \label{alternative_layered_sum}
\end{align}
Then, we will form bounds for  the layers of expected values taken with respect to $\bE_{J_1}, \dots, \bE_{J_k}$ in a backward inductive manner. First, by Cauchy's inequality and \eqref{Hibert_norm_less_than_1},
\[
 \Bigg| \bE_{J_k} \Bigg[\sum_{{\bf i}_{J_k} }   f_{{\bf i}_{J_k}}^{(k)}  \Big( Y^{(J_k)}_{{\bf i}_{J_k}}\Big)
 \kappa_{\bf i}   \Bigg] \Bigg| \leq  \sqrt{\bE_{J_k} [\sum_{{\bf i}_{J_k} } \kappa_{\bf i}^2]} \sqrt{ \bE_{J_k}\big[ \sum_{{\bf i}_{J_k} }  |f_{{\bf i}_{J_k}}^{(k)}  ( Y^{(J_k)}_{{\bf i}_{J_k}})|^2] } \leq \sqrt{ \bE_{J_k} [\sum_{{\bf i}_{J_k} } \kappa_{\bf i}^2]}.
 \]
Now, assume that for a given $1< t \leq k$, one has shown that
\begin{equation} \label{induction_hypothesis}
\Bigg|\bE_{J_t \cup \cdots \cup J_k}
\Bigg[
\sum_{{\bf i}_{J_t \cup \cdots \cup J_k}}
\kappa_{\bf i}
\prod_{j=t}^k  f_{{\bf i}_{J_j}}^{(j)}  \Big( 
Y^{(J_j)}_{{\bf i}_{J_j}}
\Big)
\Bigg]\Bigg|
 \leq 
 \sqrt{ \bE_{J_t \cup \cdots \cup J_k} \Bigg[\sum_{{\bf i}_{J_t \cup \cdots \cup J_k} } \kappa_{\bf i}^2\Bigg]}. 
\end{equation}
Since 
\begin{multline*}
\bE_{J_{t-1} \cup \cdots \cup J_k}
\Bigg[
\sum_{{\bf i}_{J_{t-1} \cup \cdots \cup J_k}}
\kappa_{\bf i}
\prod_{j=t-1}^k  f_{{\bf i}_{J_j}}^{(j)}  \Big(
Y^{(J_j)}_{{\bf i}_{J_j}}
 \Big)
\Bigg]\\
= 
\bE_{J_{t-1}}
\Bigg[ \sum_{{\bf i}_{J_{t-1}}} 
\Bigg\{
 f_{{\bf i}_{J_{t-1}}}^{(t-1)}  
\Big( Y^{(J_{t-1})}_{{\bf i}_{J_{t-1}}}\Big)
\bE_{J_t \cup \cdots \cup J_k}
\Bigg[
\sum_{{\bf i}_{J_t \cup \cdots \cup J_k}}
\kappa_{\bf i}
\prod_{j=t}^k 
 f_{{\bf i}_{J_j}}^{(j)}  
\Big( Y^{(J_j)}_{{\bf i}_{J_j}}\Big)
\Bigg] \Bigg\}
\Bigg],
\end{multline*}
by applying Cauchy's inequality with \eqref{induction_hypothesis} and \eqref{Hibert_norm_less_than_1}, we get
\begin{equation} \label{induction_proved}
\Bigg|\bE_{J_{t-1} \cup \cdots \cup J_k}
\Bigg[
\sum_{{\bf i}_{J_{t-1} \cup \cdots \cup J_k}}
\kappa_{\bf i}
\prod_{j=t-1}^k  f_{{\bf i}_{J_j}}^{(j)}  \Big(
Y^{(J_j)}_{{\bf i}_{J_j}}
 \Big)
\Bigg]\Bigg| 
 \leq
 \sqrt{ \bE_{J_{t-1} \cup \cdots \cup J_k} \Bigg[\sum_{{\bf i}_{J_{t-1} \cup \cdots \cup J_k} } \kappa_{\bf i}^2\Bigg]}.
\end{equation}
Since the induction is proved in \eqref{induction_proved}, by expression \eqref{alternative_layered_sum} one have shown 
\[
\Bigg|\bE_I\Bigg[  \sum_{{\bf i}_I}
\kappa_{\bf i}
\prod_{j=1}^k  f_{{\bf i}_{J_j}}^{(j)}  \Big( (Y^{(l)}_{i_l})_{l \in J_j}\Big)
\Bigg]  \Bigg| \leq  \sqrt{ \bE_I \Bigg[\sum_{{\bf i}_I } \kappa_{\bf i}^2\Bigg]},
\]
which implies
\eqref{op_norm_bdd} in light of \defref{opnorm}.

\subsection{Remaining proof for \lemref{Bernstein_ineq_h_proj}} \label{app:remaining_pf_Bernstein_ineq_h_proj}

It remains to prove \clmref{claim2}. In  light of how the $M_I$'s are defined in \eqref{MIJ_def_2}, \eqref{sub_Gaussian_condition} precisely happens when, for any pair $I \subset [r]$ and $\cJ \in \cP_I$ such that $(I, \cJ) \neq ([r], \{[r]\})$, 
\begin{align}
x &\leq \cC_r  \cK_r e \Bigg[ \frac{ M_{[r]}^2}{M_{I}^{1/(2 |I^c| + |\cJ|/2)}} \Bigg]^{\frac{1}{2 - (2 |I^c| + |\cJ|/2)^{-1}}}\notag\\
 &= \cC_r \cK_r  e  \cdot  \sigma_h  \cdot \frac{n^{ \frac{ (2 r + \frac{1}{2}) |I^c|  + \frac{r}{2} |\cJ| - \frac{r}{2}}{ 4 |I^c| + |\cJ| -1 } }}{ \Big(\log (2 n^{|I^c|} +1) \Big)^\frac{|I^c|}{4 |I^c| + |\cJ| - 1} } \label{x_less_than_quantity},
  \end{align}
  where we have used the fact that 
  \[
 \Big(n^{r - \frac{|I|}{2(2 |I^c| + |\cJ|/2)}} \Big)^{\frac{1}{2 - (2 |I^c| + |\cJ|/2)^{-1}}}
= n^{ \frac{ (2 r + \frac{1}{2}) |I^c|  + \frac{r}{2} |\cJ| - \frac{r}{2}}{ 4 |I^c| + |\cJ| -1 } }.
\]
In particular, the right hand side on \eqref{x_less_than_quantity} is always equal to
\begin{equation} \label{max_x_bdd_when_I_is_1_to_m}
\cC_r \cdot \cK_r \cdot e \cdot  \sigma_h  \cdot n^{r/2} \text{ when } I = [r],
\end{equation}
because $|I^c| = 0$ in this case. 
We claim that the right hand side on \eqref{x_less_than_quantity} can be no smaller than
\begin{equation} \label{max_x_bdd_when_I_is_not_1_to_m}
\cC_r \cdot \cK_r \cdot e \cdot   \sigma_h  \cdot \frac{n^{\frac{(r+1)^2}{2(r+2)}}}{(\log (2n^r +1))^{\frac{r}{4r - 1}}} \text{ for any } I \text{ other than } [r];
\end{equation}
if this claim is true, because $n^{\frac{(r+1)^2}{2(r+2)}} / n^{\frac{r}{2}} = n^{\frac{1}{2(r+2)}}$ for all $r \geq 1$, considering \eqref{x_less_than_quantity}  and the upper limit for $x$ in \eqref{max_x_bdd_when_I_is_1_to_m} when $I = [r]$, we will have proven  \clmref{claim2}. Hence, it suffices to establish the claimed lower bound  in \eqref{max_x_bdd_when_I_is_not_1_to_m} for the right hand side of \eqref{x_less_than_quantity}, when $I \neq [r]$. 

First we define the functions 
\[
\mathtt{f}( |I^c| , |\cJ| ) =  \frac{ (2 r + \frac{1}{2}) |I^c|  + \frac{r}{2} |\cJ| - \frac{r}{2}}{ 4 |I^c| + |\cJ| -1 } ,
\]
and 
\[
\mathtt{g}(|I^c|, |\cJ|) = \frac{|I^c|}{4 |I^c| + |\cJ| - 1},
\]
so that the general expression in \eqref{x_less_than_quantity} can be written as 
\begin{equation} \label{x_less_than_quantity_exp}
\cC_r \cK_r  e  \cdot \sigma_h \cdot \frac{n^{\mathtt{f}( |I^c| , |\cJ| )}}{(\log (2 n^{|I^c|} +1))^{\mathtt{g}(|I^c|, |\cJ|)}}.
\end{equation}
Next we will respectively find the \emph{minimum attainable value} of $\mathtt{f}( |I^c| , |\cJ| )$ and the \emph{maximum attainable value} of $\mathtt{g}( |I^c| , |\cJ| )$ for $I \neq [r]$ and $\cJ \in \cP_I$:

\begin{enumerate}
\item Minimum attainable value of $\mathtt{f}( |I^c| , |\cJ| )$: By taking partial derivatives, we have 
\[
\frac{\partial \mathtt{f}}{\partial |I^c|} = \frac{\frac{1}{2}( |\cJ|- 1) }{( 4 |I^c| + |\cJ| -1 )^2} \geq 0
\]
and
\[
\frac{\partial \mathtt{f} }{\partial |\cJ|} = \frac{ - \frac{1}{2} |I^c| }{( 4 |I^c| + |\cJ| -1 )^2} \leq 0.
\]
Hence, for $\mathtt{f}$ to attain its minimum value, $|I^c|$ should be as small as possible at $|I^c| = 1$ and $|\cJ|$ should be as large as possible at $|\cJ| = r-1$, which gives 
\[
\mathtt{f}(1, r-1) = \frac{2 r + \frac{1}{2} + \frac{r}{2}(r-1) - \frac{r}{2}}{ 4 + r-2} = \frac{\frac{r^2}{2} + r  + \frac{1}{2}}{2 +r}
=  \frac{(r+1)^2}{2(r+2)}.
\]

\item Maximum attainable value of $\mathtt{g}( |I^c| , |\cJ| )$: For any non-empty $I^c \neq [r]$ (or equivalently, $I \neq \emptyset$), the largest $\mathtt{g}( |I^c| , |\cJ| )$ can be is $\frac{|I^c|}{4 |I^c|} = \frac{1}{4}$ since the smallest $|\cJ|$ can be is $1$. When $I^c = [r]$, then $|\cJ| = |\emptyset|$ must be $0$, in which case we have 
\[
\mathtt{g}(r, 0) =  \frac{r}{4r - 1},
\]
which is larger than $1/4$ and hence the maximum attainable value of $\mathtt{g}( |I^c| , |\cJ| )$.
\end{enumerate}
In light of the expression in \eqref{x_less_than_quantity_exp} and considering the minimum and maximum attainable values for $\mathtt{f}$  and $\mathtt{g}$ respectively when $I \neq [r]$ and $\cJ \in \cP_I$ above, we have established the lower bound in \eqref{max_x_bdd_when_I_is_not_1_to_m}.

\subsection{Remaining proof for \thmref{Bernstein_ineq}} \label{app:remain_pf_bernstein_ineq}
It remains to prove \clmsref{claim1}, which shall use the following lemmas: 

\begin{lemma} [{Sub-Weibull order for $\bE_{I, H} [ (h^{(r)}_{\bf i})^2 ]$}] \label{lem:subWeibull_our_kernel}
For a given $I \subset [r]$ and ${\bf i} = (i_1, \dots, i_r)$ such that $i_1 \neq \dots \neq i_r$, as a variable $\bE_{I, H} [ (h^{(r)}_{\bf i})^2 ]$ is
\begin{enumerate}
\item   a  polynomial with maximum possible degree $4(r - |I|)$ in jointly normal random variables, and 
 \item is sub-Weibull of maximum order $\alpha= \frac{1}{2(r- |I|)}$.
\end{enumerate}
\end{lemma}
\begin{proof} [Proof of \lemref{subWeibull_our_kernel}]
For $(i)$, in light of the definition of a projection kernel in \eqref{projection_fns} and the structure of $h$ defined by \eqref{breve_H} and  \eqref{polynomial_kernel}, $\pi_r (h)$ must be a polynomial in the input  $\bx_1, \dots, \bx_r$ whose possible highest degree terms stemming from $\delta_{{\bf x}_1}  \times \cdots \times \delta_{{\bf x}_r}   \times P_{\Hf}^{m-r} h$ are of the form
\[
 \hat{\theta}_{u_1 v_1}({\bf x}_1) \dots \hat{\theta}_{u_r v_r } ({\bf x}_r)
\]
when the coefficients are ignored.  Hence, $(h^{(r)}_{\bf i})^2$ is also a polynomial in the decoupled data $\bX_{i_1}^{(1)}, \dots, \bX_{i_r}^{(r)}$, with possible highest degree terms (modulo coefficient) of the form 
\[
 \hat{\theta}_{u_1 v_1}(\bX_{i_1}^{(1)}) \dots \hat{\theta}_{u_r v_r } (\bX_{i_r}^{(r)})  \hat{\theta}_{u_{r+1}v_{r+1}}(\bX_{i_1}^{(1)}) \dots \hat{\theta}_{u_{2r} v_{2r} } (\bX_{i_r}^{(r)}),
\]
which implies that  $\bE_{I, H} [ (h^{(r)}_{\bf i})^2 ]$  is also a polynomial  with possible highest degree terms (modulo coefficient) of the form
\[
\prod_{l \in I^c}  \hat{\theta}_{u_l v_l}(\bX_{i_l}^{(l)})   \hat{\theta}_{u_{l+r}v_{l+r}}(\bX_{i_l}^{(l)});
\]
the latter form implies $(i)$.  $(ii)$ is a consequence of $(i)$ in light of \lemref{GaussPolyIsSubWeibull}.
\end{proof}
\begin{lemma} [Bound on the second moment of $h^{(r)}_{[r]}$ under $H$] \label{lem:2nd_moment_bdd_hrr}
Under $H$, for a constant $C(r) > 0$,  the random quantity $h^{(r)}_{[r]} = \frac{\pi_r(h)(\bX_1^{(1)}, \dots, \bX_r^{(r)})}{r!} $ defined in \eqref{frakh_def} satisfy the  second moment bound: 
$
\bEH[(h^{(r)}_{[r]})^2] \leq C(r) \sigma_h^2 .
$
\end{lemma}
\begin{proof} [Proof of \lemref{2nd_moment_bdd_hrr}]
For each $k \in [m]$, let $\zeta_k \equiv \bEH[g_k^2]$ for $g_k(\bx_1, \dots, \bx_k) \equiv \bEH[h(\bx_1, \dots, \bx_k, \bX_{k+1}, \dots, \bX_m)]$. 
By classical U-statistic theory \citep[p.31-32]{korolyuk2013theory}, it is known that 
\[
\bEH\Big[\Big(\pi_r(h)(\bX_1, \dots, \bX_r)\Big)^2\Big]  = \zeta_r - {r \choose 1} \zeta_{r-1} + {r \choose 2} \zeta_{r-2}  + \cdots +(-1)^{r-1} {r \choose r-1} \zeta_1,
\]
as well as 
$
0 \leq \zeta_1 \leq \zeta_2 \leq \cdots \leq \zeta_m = \sigma_h^2$; apparently, these two facts imply $\bEH[(h^{(r)}_{[r]})^2] \leq C(r) \sigma_h^2$.

\end{proof}

We now begin to prove \clmref{claim1}. Note that the claim  is trivial if $I = [r]$, because in that case, $|I^c| = r - |I| = 0$ and using the definition in \eqref{frakh_def} as well as \lemref{2nd_moment_bdd_hrr},
\begin{multline*}
\bigg\|\max_{{\bf i}_{I^c}} \bE_{I, H} \bigg[\sum_{{\bf i}_I } (h^{(r)}_{\bf i})^2\bigg]  \bigg\|_{p/2, H} 
= \bigg\| \bEH \bigg[\sum_{{\bf i} } (h_{\bf i}^{(r)})^2\bigg]  \bigg\|_{p/2, H}  
= \bEH \bigg[\sum_{{\bf i} } (h^{(r)}_{\bf i})^2\bigg] \\
\leq n^r \bEH[(h^{(r)}_{[r]})^2] \leq C(r)  n^r \sigma_h^2 .
\end{multline*} 
For any  $I \subsetneq [r]$, i.e. $I$ is a proper subset of $[r]$, we have 
\begin{align}
 &\bigg\|\max_{{\bf i}_{I^c}} \bE_{I, H} \bigg[ \sum_{{\bf i}_I } (h^{(r)}_{\bf i})^2 \bigg]  \bigg\|_{p/2, H}  \notag\\
  &\underset{(i)}{\leq} C(r) p^{2 (r - |I|)}  \bigg\|\max_{{\bf i}_{I^c}} \bE_{I, H} \bigg[\sum_{{\bf i}_I } (h^{(r)}_{\bf i})^2\bigg]  \Bigg\|_{\psi_{\frac{1}{2(r- |I|)}}, H} \notag\\
 &\underset{(ii)}{\leq} C(r) p^{2 (r - |I|)} \psi_{\frac{1}{2(r- |I|)}}^{-1} \Bigg( 2 n^{r - |I|} \Bigg) \max_{{\bf i}_{I^c}} \Bigg\| \bE_{I,  H} \bigg[\sum_{{\bf i}_I } (h^{(r)}_{\bf i})^2\bigg] \Bigg\|_{\psi_{\frac{1}{2(r- |I|)}}, H} \notag \\
 &\underset{(iii)}{=}  C(r) p^{2 (r - |I|)} \bigg(\log ( 2 n^{r - |I|} + 1) \bigg)^{2(r - |I|)} 
 \Bigg\| \bE_{I,  H} \bigg[\sum_{{\bf i}_I } (h^{(r)}_{\bf i})^2\bigg] \Bigg\|_{\psi_{\frac{1}{2(r- |I|)}}, H}, \label{idunno}
\end{align}
where $\|\cdot\|_{\psi_\alpha, H}$ means a sub-Weibull norm defined by an expectation under $H$; the inequalities are explained as follows:
\begin{enumerate}
\item  \lemsref{max_ineq} and \lemssref{subWeibull_our_kernel}$(ii)$ imply $\max_{{\bf i}_{I^c}} \bE_{I, H} \bigg[ \sum_{{\bf i}_I } (h^{(r)}_{\bf i})^2 \bigg]$ is sub-Weibull of order $\frac{1}{2(r- |I|)}$. The inequality then follows from the first inequality in \lemref{weibull_p_norm_relation} and that a constant depending on $\frac{1}{2(r- |I|)}$ can be bounded by a further constant depending only on $r$ because $0 \leq |I| < r$;
\item  The maximal inequality \lemref{max_ineq}; 
\item The explicit expression $\psi^{-1}_{\frac{1}{2 (r - |I|)}} (\cdot) = (\log(\cdot + 1) )^{2 (r - |I|)}$.
\end{enumerate}
 By the second inequality in \lemref{weibull_p_norm_relation},
\begin{align}
 & \Bigg\| \bE_{I,  H} \bigg[\sum_{{\bf i}_I } (h^{(r)}_{\bf i})^2\bigg] \Bigg\|_{\psi_{\frac{1}{2(r- |I|)}}, H} \notag \\
  &\leq C(r) \sup_{p \geq 1}  p^{- 2(r - |I|)}\Bigg\| \bE_{I,  H} \bigg[\sum_{{\bf i}_I } (h^{(r)}_{\bf i})^2\bigg] \Bigg\|_{p, H} \notag \\
  &\leq C(r) \Bigg\{  \sup_{p \geq 2} \Bigg( p^{- 2(r - |I|)}  \sum_{{\bf i}_I } \Bigg\| \bE_{I, H} \bigg[(h^{(r)}_{\bf i})^2\bigg] \Bigg\|_{p, H} \Bigg) +  
  \sup_{1 \leq p \leq 2} \Bigg(p^{-2(r - |I|)}\sum_{{\bf i}_I } \Bigg\| \bE_{I, H} \bigg[(h^{(r)}_{\bf i})^2\bigg] \Bigg\|_{p, H}  \Bigg)\Bigg\} \notag\\
  &\leq  C(r) \Bigg\{  \sup_{p \geq 2} \Bigg( p^{- 2(r - |I|)}  \sum_{{\bf i}_I } \Bigg\| \bE_{I, H} \bigg[(h^{(r)}_{\bf i})^2\bigg] \Bigg\|_{p, H} \Bigg) +  \sum_{{\bf i}_I }  \Bigg\| \bE_{I, H} \bigg[(h^{(r)}_{\bf i})^2\bigg] \Bigg\|_{2, H} \Bigg\} \label{first_bdd_on_subweibull_norm}. 
\end{align}
Now for any $p \geq 2$, by  \lemref{hypercontraction} and \lemref{subWeibull_our_kernel}$(i)$, 
\begin{align}
\Bigg\| \bE_{I, H} \bigg[(h^{(r)}_{\bf i})^2\bigg] \Bigg\|_{p, H}
& \leq C(r) \times \big(p - 1\big)^{2(r - |I|)} \times \bigg\|\bEIH \big[(h^{(r)}_{\bf i})^2\big]\bigg\|_{2, H} \notag \\
& = C(r) \times  \big(p - 1\big)^{2(r - |I|)} \times \sqrt{  \bEH[ (\bEIH [(h^{(r)}_{\bf i})^2])^2]} \notag\\
&\leq C(r) p^{2(r - |I|)}  \sqrt{\bEH[ (h^{(r)}_{\bf i})^4] }  \text{ by Jensen's inequality} \notag\\
& = C(r) p^{2(r - |I|)} \|h^{(r)}_{\bf i}\|_{4, H}^2 \notag\\
& \leq C(r) 3^{2r} p^{2(r- |I|)} \|h^{(r)}_{\bf i}\|_{2, H}^2 \notag\\
& \text{ by \lemref{hypercontraction} ($h^{(r)}_{\bf i}$ is a degree-$2r$ polynomial  in normal random variables)} \notag\\
&= C(r) p^{2(r- |I|)} \bEH \big[ (h^{(r)}_{\bf i})^2 \big]  \label{pgeq2}.
\end{align}
Moreover,
\begin{align}
\Bigg\| \bEIH \bigg[(h^{(r)}_{\bf i})^2 \bigg] \Bigg\|_{2, H}
& \leq  (\bEH[ (h^{(r)}_{\bf i})^4] )^{1/2} = \|h^{(r)}_{\bf i} \|_4^2  \text{ by Jensen's inequality } \notag\\
&  \ \leq C(r)\bE[(h^{(r)}_{\bf i})^2  ] \text{ by \lemref{hypercontraction}}  \label{1leqpleq2}.
\end{align}
Combining \eqref{first_bdd_on_subweibull_norm} with \eqref{pgeq2} and \eqref{1leqpleq2}, we get from \lemref{2nd_moment_bdd_hrr} that, 
\begin{equation} \label{second_bdd_on_subweibull_norm}
\Bigg\| \bE_{I,  H} \bigg[\sum_{{\bf i}_I } (h^{(r)}_{\bf i})^2\bigg] \Bigg\|_{\psi_{\frac{1}{2(r- |I|)}}, H} 
 \leq C(r) \sum_{{\bf i}_I} \bEH \big[ (h^{(r)}_{\bf i})^2 \big] \leq C(r) n^{|I|} \sigma_h^2,
\end{equation}
which implies \clmref{claim1} for $I \subsetneq [r]$ 
from \eqref{idunno}.

\bibliographystyle{plainnat}
\bibliography{refined_BE_ic_ustat_bib}

\begin{thebibliography}{62}
\providecommand{\natexlab}[1]{#1}
\providecommand{\url}[1]{\texttt{#1}}
\expandafter\ifx\csname urlstyle\endcsname\relax
  \providecommand{\doi}[1]{doi: #1}\else
  \providecommand{\doi}{doi: \begingroup \urlstyle{rm}\Url}\fi

\bibitem[Adamczak(2006)]{adamczak2006moment}
Rados\l~aw Adamczak.
\newblock Moment inequalities for {$U$}-statistics.
\newblock \emph{Ann. Probab.}, 34\penalty0 (6):\penalty0 2288--2314, 2006.
\newblock ISSN 0091-1798,2168-894X.

\bibitem[Bakhshizadeh(2023)]{bakhshizadeh2023exponential}
Milad Bakhshizadeh.
\newblock Exponential tail bounds and large deviation principle for
  heavy-tailed u-statistics.
\newblock \emph{arXiv preprint arXiv:2301.11563}, 2023.

\bibitem[Benedetti and Risler(1990)]{RR1991}
Riccardo Benedetti and Jean-Jacques Risler.
\newblock \emph{Real algebraic and semi-algebraic sets}.
\newblock Actualit\'{e}s Math\'{e}matiques. [Current Mathematical Topics].
  Hermann, Paris, 1990.
\newblock ISBN 2-7056-6144-1.

\bibitem[Berry(1941)]{berry1941accuracy}
Andrew~C. Berry.
\newblock The accuracy of the {G}aussian approximation to the sum of
  independent variates.
\newblock \emph{Trans. Amer. Math. Soc.}, 49:\penalty0 122--136, 1941.
\newblock ISSN 0002-9947,1088-6850.

\bibitem[Bollen and Ting(2000)]{bollen2000tetrad}
Kenneth~A Bollen and Kwok-fai Ting.
\newblock A tetrad test for causal indicators.
\newblock \emph{Psychological methods}, 5\penalty0 (1):\penalty0 3, 2000.

\bibitem[Carbery and Wright(2001)]{carbery2001distributional}
Anthony Carbery and James Wright.
\newblock Distributional and $l^q$ norm inequalities for polynomials over
  convex bodies in $\mathbb{R}^n$.
\newblock \emph{Mathematical research letters}, 8\penalty0 (3):\penalty0
  233--248, 2001.

\bibitem[Chakrabortty and Kuchibhotla(2018)]{chakrabortty2018tail}
Abhishek Chakrabortty and Arun~Kumar Kuchibhotla.
\newblock Tail bounds for canonical u-statistics and u-processes with unbounded
  kernels.
\newblock 2018.

\bibitem[Chen et~al.(2014)Chen, Tian, and Pearl]{chen2014testable}
Bryant Chen, Jin Tian, and Judea Pearl.
\newblock Testable implications of linear structural equation models.
\newblock In \emph{Proceedings of the AAAI Conference on Artificial
  Intelligence}, volume~28, 2014.

\bibitem[Chen et~al.(2017)Chen, Kumor, and Bareinboim]{chen2017identification}
Bryant Chen, Daniel Kumor, and Elias Bareinboim.
\newblock Identification and model testing in linear structural equation models
  using auxiliary variables.
\newblock In \emph{International Conference on Machine Learning}, pages
  757--766. PMLR, 2017.

\bibitem[Chen et~al.(2011)Chen, Goldstein, and Shao]{chen2010normal}
Louis H.~Y. Chen, Larry Goldstein, and Qi-Man Shao.
\newblock \emph{Normal approximation by {S}tein's method}.
\newblock Probability and its Applications (New York). Springer, Heidelberg,
  2011.
\newblock ISBN 978-3-642-15006-7.

\bibitem[Chen and Kato(2019)]{chen2019randomized}
Xiaohui Chen and Kengo Kato.
\newblock Randomized incomplete $ u $-statistics in high dimensions.
\newblock \emph{The Annals of Statistics}, 47\penalty0 (6):\penalty0
  3127--3156, 2019.

\bibitem[Chernozhukov et~al.(2013)Chernozhukov, Chetverikov, and Kato]{CCKHD13}
Victor Chernozhukov, Denis Chetverikov, and Kengo Kato.
\newblock Gaussian approximations and multiplier bootstrap for maxima of sums
  of high-dimensional random vectors.
\newblock \emph{Ann. Statist.}, 41\penalty0 (6):\penalty0 2786--2819, 2013.
\newblock ISSN 0090-5364.

\bibitem[Chernozhukov et~al.(2017)Chernozhukov, Chetverikov, and Kato]{CCKHD17}
Victor Chernozhukov, Denis Chetverikov, and Kengo Kato.
\newblock Central limit theorems and bootstrap in high dimensions.
\newblock \emph{Ann. Probab.}, 45\penalty0 (4):\penalty0 2309--2352, 2017.
\newblock ISSN 0091-1798,2168-894X.
\newblock \doi{10.1214/16-AOP1113}.
\newblock URL \url{https://doi.org/10.1214/16-AOP1113}.

\bibitem[Chernozhukov et~al.(2023)Chernozhukov, Chetverikov, and
  Koike]{ccky_2023}
Victor Chernozhukov, Denis Chetverikov, and Yuta Koike.
\newblock Nearly optimal central limit theorem and bootstrap approximations in
  high dimensions.
\newblock \emph{Ann. Appl. Probab.}, 33\penalty0 (3):\penalty0 2374--2425,
  2023.
\newblock ISSN 1050-5164,2168-8737.
\newblock \doi{10.1214/22-aap1870}.
\newblock URL \url{https://doi.org/10.1214/22-aap1870}.

\bibitem[Chernozhuokov et~al.(2022)Chernozhuokov, Chetverikov, Kato, and
  Koike]{cckk2022}
Victor Chernozhuokov, Denis Chetverikov, Kengo Kato, and Yuta Koike.
\newblock Improved central limit theorem and bootstrap approximations in high
  dimensions.
\newblock \emph{Ann. Statist.}, 50\penalty0 (5):\penalty0 2562--2586, 2022.
\newblock ISSN 0090-5364,2168-8966.
\newblock \doi{10.1214/22-aos2193}.
\newblock URL \url{https://doi.org/10.1214/22-aos2193}.

\bibitem[Cox et~al.(2015)Cox, Little, and O'Shea]{CLO2015}
David~A. Cox, John Little, and Donal O'Shea.
\newblock \emph{Ideals, varieties, and algorithms}.
\newblock Undergraduate Texts in Mathematics. Springer, Cham, fourth edition,
  2015.
\newblock ISBN 978-3-319-16720-6; 978-3-319-16721-3.
\newblock \doi{10.1007/978-3-319-16721-3}.
\newblock URL \url{https://doi.org/10.1007/978-3-319-16721-3}.
\newblock An introduction to computational algebraic geometry and commutative
  algebra.

\bibitem[de~la Pe\~{n}a et~al.(2009)de~la Pe\~{n}a, Lai, and
  Shao]{SN_limit_theory}
Victor~H. de~la Pe\~{n}a, Tze~Leung Lai, and Qi-Man Shao.
\newblock \emph{Self-normalized processes}.
\newblock Probability and its Applications (New York). Springer-Verlag, Berlin,
  2009.
\newblock ISBN 978-3-540-85635-1.
\newblock \doi{10.1007/978-3-540-85636-8}.
\newblock URL \url{https://doi.org/10.1007/978-3-540-85636-8}.
\newblock Limit theory and statistical applications.

\bibitem[De~la Pena and Gin{\'e}(1999)]{de2012decoupling}
Victor De~la Pena and Evarist Gin{\'e}.
\newblock \emph{Decoupling: from dependence to independence}.
\newblock Springer Science \& Business Media, 1999.

\bibitem[de~Loera et~al.(1995)de~Loera, Sturmfels, and Thomas]{Loera1995}
Jes\'{u}s~A. de~Loera, Bernd Sturmfels, and Rekha~R. Thomas.
\newblock Gr\"{o}bner bases and triangulations of the second hypersimplex.
\newblock \emph{Combinatorica}, 15\penalty0 (3):\penalty0 409--424, 1995.
\newblock ISSN 0209-9683.
\newblock \doi{10.1007/BF01299745}.
\newblock URL \url{https://doi.org/10.1007/BF01299745}.

\bibitem[Drton and Xiao(2016)]{drton2016wald}
Mathias Drton and Han Xiao.
\newblock Wald tests of singular hypotheses.
\newblock 2016.

\bibitem[Drton et~al.(2007)Drton, Sturmfels, and
  Sullivant]{DrtonSturmfelsSullivantPTRF2007}
Mathias Drton, Bernd Sturmfels, and Seth Sullivant.
\newblock Algebraic factor analysis: tetrads, pentads and beyond.
\newblock \emph{Probab. Theory Related Fields}, 138\penalty0 (3-4):\penalty0
  463--493, 2007.
\newblock ISSN 0178-8051,1432-2064.
\newblock \doi{10.1007/s00440-006-0033-2}.
\newblock URL \url{https://doi.org/10.1007/s00440-006-0033-2}.

\bibitem[Drton et~al.(2008)Drton, Massam, and Olkin]{drton2008moments}
Mathias Drton, H{\'e}l{\`e}ne Massam, and Ingram Olkin.
\newblock Moments of minors of wishart matrices.
\newblock 2008.

\bibitem[Dufour et~al.(2017)Dufour, Renault, and Zinde-Walsh]{dufour2017wald}
Jean-Marie Dufour, Eric Renault, and Victoria Zinde-Walsh.
\newblock Wald tests when restrictions are locally singular.
\newblock \emph{Unpublished manuscript}, 2017.
\newblock URL
  \url{https://jeanmariedufour.research.mcgill.ca/Dufour_Renault_ZindeWalsh_2012_WaldTestsLocallySingularRestrictions_W.pdf}.

\bibitem[Eaton(2007)]{eatonmultistat}
Morris~L. Eaton.
\newblock \emph{Multivariate statistics}, volume~53 of \emph{Institute of
  Mathematical Statistics Lecture Notes---Monograph Series}.
\newblock Institute of Mathematical Statistics, Beachwood, OH, 2007.
\newblock ISBN 978-0-940600-69-0; 0-940600-69-2.
\newblock A vector space approach, Reprint of the 1983 original [MR0716321].

\bibitem[Esseen(1942)]{esseen1942}
Carl-Gustav Esseen.
\newblock On the liapunoff limit of error in the theory of probability.
\newblock \emph{Ark. Mat. Astron. Fys.}, A28\penalty0 (9):\penalty0 1--19,
  1942.

\bibitem[Fang and Koike(2021)]{fangkoike2021}
Xiao Fang and Yuta Koike.
\newblock High-dimensional central limit theorems by {S}tein's method.
\newblock \emph{Ann. Appl. Probab.}, 31\penalty0 (4):\penalty0 1660--1686,
  2021.
\newblock ISSN 1050-5164,2168-8737.
\newblock \doi{10.1214/20-aap1629}.
\newblock URL \url{https://doi.org/10.1214/20-aap1629}.

\bibitem[Gaffke(2002)]{gaffke2002asymptotic}
Norbert Gaffke.
\newblock On the asymptotic null-distribution of the wald statistic at singular
  parameter points.
\newblock \emph{Statistics \& Risk Modeling}, 20\penalty0 (1-4):\penalty0
  379--398, 2002.

\bibitem[Gaffke et~al.(1999)Gaffke, Steyer, and Davier]{gaffke1999asymptotic}
Norbert Gaffke, Rolf Steyer, and Alina A~von Davier.
\newblock On the asymptotic null-distribution of the wald statistic at singular
  parameter points.
\newblock \emph{Statistics \& Risk Modeling}, 17\penalty0 (4):\penalty0
  339--358, 1999.

\bibitem[Gin\'{e} and Nickl(2016)]{GN2016}
Evarist Gin\'{e} and Richard Nickl.
\newblock \emph{Mathematical foundations of infinite-dimensional statistical
  models}, volume [40] of \emph{Cambridge Series in Statistical and
  Probabilistic Mathematics}.
\newblock Cambridge University Press, New York, 2016.
\newblock ISBN 978-1-107-04316-9.

\bibitem[Gin{\'e} et~al.(2000)Gin{\'e}, Lata{\l}a, and
  Zinn]{gine2000exponential}
Evarist Gin{\'e}, Rafa{\l} Lata{\l}a, and Joel Zinn.
\newblock Exponential and moment inequalities for u-statistics.
\newblock In \emph{High Dimensional Probability II}, pages 13--38. Springer,
  2000.

\bibitem[Isserlis(1918)]{isserlis1918formula}
Leon Isserlis.
\newblock On a formula for the product-moment coefficient of any order of a
  normal frequency distribution in any number of variables.
\newblock \emph{Biometrika}, 12\penalty0 (1/2):\penalty0 134--139, 1918.

\bibitem[Janson(1984)]{janson1984asymptotic}
Svante Janson.
\newblock The asymptotic distributions of incomplete u-statistics.
\newblock \emph{Zeitschrift f{\"u}r Wahrscheinlichkeitstheorie und Verwandte
  Gebiete}, 66\penalty0 (4):\penalty0 495--505, 1984.

\bibitem[Janson(1997)]{jason_gauss_hilbert_space_1997}
Svante Janson.
\newblock \emph{Gaussian {H}ilbert spaces}, volume 129 of \emph{Cambridge
  Tracts in Mathematics}.
\newblock Cambridge University Press, Cambridge, 1997.
\newblock ISBN 0-521-56128-0.

\bibitem[Kelley(1935)]{kelley1935essential}
Truman~L Kelley.
\newblock Essential traits of mental life.
\newblock \emph{(No Title)}, 1935.

\bibitem[Koroljuk and Borovskich(1994)]{KB1994Ustat}
V.~S. Koroljuk and Yu.~V. Borovskich.
\newblock \emph{Theory of {$U$}-statistics}, volume 273 of \emph{Mathematics
  and its Applications}.
\newblock Kluwer Academic Publishers Group, Dordrecht, 1994.
\newblock ISBN 0-7923-2608-3.
\newblock Translated from the 1989 Russian original by P. V. Malyshev and D. V.
  Malyshev and revised by the authors.

\bibitem[Korolyuk and Borovskich(2013)]{korolyuk2013theory}
Vladimir~S Korolyuk and Yu~V Borovskich.
\newblock \emph{Theory of U-statistics}, volume 273.
\newblock Springer Science \& Business Media, 2013.

\bibitem[Kuchibhotla and Chakrabortty(2022)]{kuchibhotla2022moving}
Arun~Kumar Kuchibhotla and Abhishek Chakrabortty.
\newblock Moving beyond sub-gaussianity in high-dimensional statistics:
  Applications in covariance estimation and linear regression.
\newblock \emph{Information and Inference: A Journal of the IMA}, 11\penalty0
  (4):\penalty0 1389--1456, 2022.

\bibitem[Lehmann and Romano(2022)]{TSH2022}
E.~L. Lehmann and Joseph~P. Romano.
\newblock \emph{Testing statistical hypotheses}.
\newblock Springer Texts in Statistics. Springer, Cham, fourth edition, 2022.
\newblock ISBN 978-3-030-70577-0; 978-3-030-70578-7.

\bibitem[Leung and Drton(2018)]{leung2018algebraic}
Dennis Leung and Mathias Drton.
\newblock Algebraic tests of general gaussian latent tree models.
\newblock \emph{Advances in Neural Information Processing Systems}, 31, 2018.

\bibitem[Leung and Shao(2023{\natexlab{a}})]{leung2023another}
Dennis Leung and Qi-Man Shao.
\newblock Another look at stein's method for studentized nonlinear statistics
  with an application to u-statistics.
\newblock \emph{arXiv preprint arXiv:2301.02098}, 2023{\natexlab{a}}.

\bibitem[Leung and Shao(2023{\natexlab{b}})]{leung2023nonuniform}
Dennis Leung and Qi-Man Shao.
\newblock Nonuniform berry-esseen bounds for studentized u-statistics.
\newblock \emph{Bernoulli, accepted}, 2023{\natexlab{b}}.

\bibitem[Lieb and Loss(2001)]{LL2001}
Elliott~H. Lieb and Michael Loss.
\newblock \emph{Analysis}, volume~14 of \emph{Graduate Studies in Mathematics}.
\newblock American Mathematical Society, Providence, RI, second edition, 2001.
\newblock ISBN 0-8218-2783-9.

\bibitem[Lopes(2022)]{lopes2022}
Miles~E. Lopes.
\newblock Central limit theorem and bootstrap approximation in high dimensions:
  near {$1/\sqrt{n}$} rates via implicit smoothing.
\newblock \emph{Ann. Statist.}, 50\penalty0 (5):\penalty0 2492--2513, 2022.
\newblock ISSN 0090-5364,2168-8966.
\newblock \doi{10.1214/22-aos2184}.
\newblock URL \url{https://doi.org/10.1214/22-aos2184}.

\bibitem[Mitchell et~al.(2019)Mitchell, Allman, and
  Rhodes]{mitchell2019hypothesis}
Jonathan~D Mitchell, Elizabeth~S Allman, and John~A Rhodes.
\newblock Hypothesis testing near singularities and boundaries.
\newblock \emph{Electronic journal of statistics}, 13\penalty0 (1):\penalty0
  2150, 2019.

\bibitem[Peng et~al.(2022)Peng, Coleman, and Mentch]{peng2022rates}
Wei Peng, Tim Coleman, and Lucas Mentch.
\newblock Rates of convergence for random forests via generalized u-statistics.
\newblock \emph{Electronic Journal of Statistics}, 16\penalty0 (1):\penalty0
  232--292, 2022.

\bibitem[Pillai and Meng(2016)]{pillai2016unexpected}
Natesh~S. Pillai and Xiao-Li Meng.
\newblock An unexpected encounter with {C}auchy and {L}\'{e}vy.
\newblock \emph{Ann. Statist.}, 44\penalty0 (5):\penalty0 2089--2097, 2016.
\newblock ISSN 0090-5364,2168-8966.

\bibitem[Serfling(1980)]{serfling1980}
Robert~J. Serfling.
\newblock \emph{Approximation theorems of mathematical statistics}.
\newblock Wiley Series in Probability and Mathematical Statistics. John Wiley
  \& Sons, Inc., New York, 1980.
\newblock ISBN 0-471-02403-1.

\bibitem[Shafer et~al.(1993)Shafer, Kogan, and
  Spirtes]{shafer1993generalization}
Glenn Shafer, Alexander Kogan, and Peter Spirtes.
\newblock \emph{Generalization of the tetrad representation theorem}.
\newblock Rutgers University. Rutgers Center for Operations Research [RUTCOR],
  1993.

\bibitem[Shao(2010)]{shao2010icm}
Qi-Man Shao.
\newblock Stein's method, self-normalized limit theory and applications.
\newblock In \emph{Proceedings of the {I}nternational {C}ongress of
  {M}athematicians. {V}olume {IV}}, pages 2325--2350. Hindustan Book Agency,
  New Delhi, 2010.
\newblock ISBN 978-81-85931-08-3; 978-981-4324-34-2; 981-4324-34-5.

\bibitem[Shiers et~al.(2016)Shiers, Zwiernik, Aston, and Smith]{shiers2016}
N.~Shiers, P.~Zwiernik, J.~A.~D. Aston, and J.~Q. Smith.
\newblock The correlation space of {G}aussian latent tree models and model
  selection without fitting.
\newblock \emph{Biometrika}, 103\penalty0 (3):\penalty0 531--545, 2016.
\newblock ISSN 0006-3444,1464-3510.

\bibitem[Silva et~al.(2006)Silva, Scheines, Glymour, Spirtes, and
  Chickering]{silva2006learning}
Ricardo Silva, Richard Scheines, Clark Glymour, Peter Spirtes, and
  David~Maxwell Chickering.
\newblock Learning the structure of linear latent variable models.
\newblock \emph{Journal of Machine Learning Research}, 7\penalty0 (2), 2006.

\bibitem[Song et~al.(2019)Song, Chen, and Kato]{SongChenKato2019}
Yanglei Song, Xiaohui Chen, and Kengo Kato.
\newblock Approximating high-dimensional infinite-order {$U$}-statistics:
  statistical and computational guarantees.
\newblock \emph{Electron. J. Stat.}, 13\penalty0 (2):\penalty0 4794--4848,
  2019.
\newblock ISSN 1935-7524.
\newblock \doi{10.1214/19-EJS1643}.
\newblock URL \url{https://doi.org/10.1214/19-EJS1643}.

\bibitem[Spearman(1927)]{spearman1927}
C.~Spearman.
\newblock \emph{The Abilities of Man: Their Nature and Measurement}.
\newblock Macmillan, 1927.

\bibitem[Spirtes et~al.(2000)Spirtes, Glymour, and
  Scheines]{spirtes2000causation}
Peter Spirtes, Clark~N Glymour, and Richard Scheines.
\newblock \emph{Causation, prediction, and search}.
\newblock MIT press, 2000.

\bibitem[Sturma et~al.(2022)Sturma, Drton, and Leung]{sturma2022testing}
Nils Sturma, Mathias Drton, and Dennis Leung.
\newblock Testing many and possibly singular polynomial constraints.
\newblock \emph{arXiv preprint arXiv:2208.11756}, 2022.

\bibitem[Sullivant et~al.(2010)Sullivant, Talaska, and Draisma]{treksullivant}
Seth Sullivant, Kelli Talaska, and Jan Draisma.
\newblock Trek separation for {G}aussian graphical models.
\newblock \emph{Ann. Statist.}, 38\penalty0 (3):\penalty0 1665--1685, 2010.
\newblock ISSN 0090-5364,2168-8966.
\newblock \doi{10.1214/09-AOS760}.
\newblock URL \url{https://doi.org/10.1214/09-AOS760}.

\bibitem[Vaart and Wellner(1996)]{vaart1996weak}
Aad~W Vaart and Jon~A Wellner.
\newblock Weak convergence.
\newblock In \emph{Weak convergence and empirical processes}, pages 16--28.
  Springer, 1996.

\bibitem[Vladimirova et~al.(2020)Vladimirova, Girard, Nguyen, and
  Arbel]{vladimirova2020sub}
Mariia Vladimirova, St{\'e}phane Girard, Hien Nguyen, and Julyan Arbel.
\newblock Sub-weibull distributions: Generalizing sub-gaussian and
  sub-exponential properties to heavier tailed distributions.
\newblock \emph{Stat}, 9\penalty0 (1):\penalty0 e318, 2020.

\bibitem[Wick(1950)]{wick1950}
G.~C. Wick.
\newblock The evaluation of the collision matrix.
\newblock \emph{Phys. Rev. (2)}, 80:\penalty0 268--272, 1950.
\newblock ISSN 0031-899X,1536-6065.

\bibitem[Wishart(1928)]{wishart1928sampling}
John Wishart.
\newblock Sampling errors in the theory of two factors.
\newblock \emph{British Journal of Psychology}, 19\penalty0 (2):\penalty0 180,
  1928.

\bibitem[Wong et~al.(2020)Wong, Li, and Tewari]{wong2020lasso}
Kam~Chung Wong, Zifan Li, and Ambuj Tewari.
\newblock Lasso guarantees for {$\beta$}-mixing heavy-tailed time series.
\newblock \emph{Ann. Statist.}, 48\penalty0 (2):\penalty0 1124--1142, 2020.
\newblock ISSN 0090-5364,2168-8966.

\bibitem[Zhang and Wei(2022)]{zhang2022sharper}
Huiming Zhang and Haoyu Wei.
\newblock Sharper sub-weibull concentrations.
\newblock \emph{Mathematics}, 10\penalty0 (13):\penalty0 2252, 2022.

\end{thebibliography}

\end{document}